\documentclass[11pt, reqno]{amsart}

\usepackage[utf8]{inputenc}
\usepackage[T1]{fontenc}
\usepackage[english]{babel}
\usepackage{booktabs}

\usepackage{amsmath,amssymb,amsthm,mathtools}
\usepackage{bm} 

\usepackage[dvipsnames]{xcolor} 
\usepackage{microtype}
\usepackage{comment}

\usepackage{stmaryrd}

\usepackage{mathrsfs}
\usepackage[scr=boondoxo]{mathalfa}

\DeclareMathOperator{\op}{op}

\numberwithin{equation}{section}

\usepackage{bm}

\usepackage[colorlinks=true,linkcolor=blue,citecolor=blue, urlcolor=blue]{hyperref}

\makeatletter
\def\enddoc@text{} 
\makeatother

\newcommand{\contactcard}[3]{%
  \begin{minipage}[t]{0.43\textwidth}
    \raggedright\footnotesize
    \textbf{#1}\par
    #2\par
    \smallskip
    e-mail: \href{mailto:#3}{#3}
  \end{minipage}%
}

\usepackage[capitalize,nameinlink]{cleveref}
\renewcommand{\P}{\mathbb{P}}
\newcommand{\E}{\mathbb{E}}
\newcommand{\Var}{{\rm Var}}
\newcommand{\Cov}{\text{Cov}}
\newcommand{\cN}{\mathcal{N}}

\newcommand{\R}{\mathbb{R}}

\renewcommand{\S}{\mathbb{S}}

\newcommand{\<}{\langle}
\renewcommand{\>}{\rangle}

\newcommand{\diag}{\text{diag}}

\newcommand{\grad}{\nabla}
\def\sT{{\mathsf T}}

\DeclareMathOperator*{\Diag}{Diag}

\theoremstyle{plain}
\newtheorem{theorem}{Theorem}
\newtheorem*{theorem*}{Theorem}
\newtheorem{lemma}{Lemma}

\newtheorem{assumption}{Assumption}
\newtheorem{definition}{Definition}

\newtheorem{proposition}{Proposition}

\newtheorem{claim}{Claim}

\newtheorem{corollary}{Corollary}

\newtheorem{remark}{Remark}[section]
\usepackage[textsize=tiny]{todonotes}

\DeclareSymbolFont{rsfs}{U}{rsfs}{m}{n}
\DeclareSymbolFontAlphabet{\mathscrsfs}{rsfs}

\def\bA{{\boldsymbol A}}

\def\bV{{\boldsymbol V}}
\def\bW{{\boldsymbol W}}

\makeatletter
\g@addto@macro\endtheorem{\vspace{1em}}
\makeatother

\def\de{{\rm d}}
\def\Tr{{\rm tr}}

\def\de{{\rm d}}

\def\cG{{\mathcal G}}

\def\cP{{\mathcal P}}

\def\cF{{\mathscr F}}

\def\cG{{\mathcal G}}

\def\cP{{\mathcal P}}

\def\reals{{\mathbb R}}

\def\naturals{{\mathbb N}}

\def\reals{{\mathbb R}}

\def\naturals{{\mathbb N}}

\def\cD{{\mathscr D}}

\def\cF{{\mathscr F}}

\def\de{{\rm d}}

\def\bA{{\boldsymbol A}}

\def\cM{{\mathcal M}}

\def\diag{{\rm diag}}

\newcommand{\ud}{\mathrm{d}}

\def\diag{\mathrm{diag}}

\def\Cov{{\rm Cov}}

\newcommand{\Exp}{\mathsf{Exp}}
\newcommand{\Hess}{\mathsf{Hess}}
\newcommand{\dive}{\operatorname{div}}

\renewcommand{\leq}{\leqslant}
\renewcommand{\geq}{\geqslant}
\renewcommand{\le}{\leqslant}
\renewcommand{\ge}{\geqslant}

\usepackage{upgreek}

\newcommand{\proj}{\boldsymbol{\mathsf{P}}^\perp}

\usepackage{caption}
\title{Homogenized Transformers}
\author{Hugo Koubbi}
\address{CEREMADE, UMR 7534, Université Paris Dauphine PSL, 75775 Paris Cedex 16, France}
\email{hugo.koubbi@dauphine.psl.eu}
\author{Borjan Geshkovski}
\address{Laboratoire Jacques Louis-Lions, Inria \& Sorbonne Université, 75005 Paris, France}
\email{borjan.geshkovski@inria.fr}
\author{Philippe Rigollet}
\address{Department of Mathematics, Massachusetts Institute of Technology, 77 Massachusetts Ave, Cambridge, Massachusetts 02139}
\email{rigollet@mit.edu}

\date{\today}

\begin{document}

\begin{abstract}
We study a random model of deep multi-head self-attention in which the weights are resampled independently across layers and heads, as at initialization of training. Viewing depth as a time variable, the residual stream defines a discrete-time interacting particle system on the unit sphere. We prove that, under suitable joint scalings of the depth, the residual step size, and the number of heads, this dynamics admits a nontrivial homogenized limit. Depending on the scaling, the limit is either deterministic or stochastic with common noise; in the mean-field regime, the latter leads to a stochastic nonlinear Fokker--Planck equation for the conditional law of a representative token. In the Gaussian setting, the limiting drift vanishes, making the homogenized dynamics explicit enough to study representation collapse. This yields quantitative trade-offs between dimension, context length, and temperature, and identifies regimes in which clustering can be mitigated.
\end{abstract}

\maketitle

\tableofcontents

\section{Introduction}

Transformers~\cite{vaswani2017attention} have become the dominant architecture in modern deep learning, underpinning large language models and a wide range of sequence modeling tasks. A striking feature of their empirical success is the systematic improvement in performance with increasing scale. The study of how model size, dataset size, and compute should be jointly increased to optimize performance has given rise to the literature on \emph{scaling laws}~\cite{kaplan2020scaling}, which documents remarkably robust power-law relationships between loss and scale.

While this body of work establishes regularities in performance as a function of scale, it leaves largely open a complementary theoretical question: 

\medskip

\begin{center}  
\textit{  What  mathematical structure emerges when a transformer is taken to an appropriate large-scale limit?  }
\end{center}

\medskip
In particular, beyond empirical power laws, one may ask how architectural hyperparameters such as depth, width, number of attention heads, residual scaling, and weight initialization must be jointly scaled so that the model admits a non-degenerate and analytically tractable limiting description.

In this work, we propose a scaling regime under which a sequence of transformer models with random weights converges, in a suitable sense, to a continuous limiting object. The limit is neither trivial nor singular: rather, it captures the cumulative effect of depth, attention, and residual dynamics through a (stochastic) differential equation.

In turn, we use this limiting framework to study \emph{representation collapse}—a phenomenon closely related to oversmoothing and rank collapse, in which token representations become increasingly clustered or low-dimensional as depth grows, ultimately degrading expressivity. Empirical and theoretical studies have documented such degeneration effects in transformers and related architectures, where deep self-attention layers drive representations toward low-rank or clustered configurations~\cite{dong2021attention,jing2022understanding}. Understanding when and how such collapse occurs and how it depends on architectural scaling remains an open theoretical question.

A line of rigorous mathematical work has begun to formalize these phenomena. In a seminal line of  work,~\cite{geshkovski2024emergence,geshkovski2025mathematical} study this clustering phenomenon in simplified regimes with fixed weights across layers, by interpreting transformers as interacting particle systems: they show that long-time (depth) dynamics lead to clustering behavior of token representations, providing a mathematical support for representation collapse and low-rank clustering. This picture was refined in~\cite{bruno2025multiscale, burger2025analysis} to account for more general models but still ones that share weights across layers.
More recently, this line of inquiry was extended to the case where only value matrices vary randomly from layer to layer. In this case, the limiting token dynamics converge to a stochastic interacting system on the sphere~\cite{fedorov2026clustering}.

\medskip
\paragraph{\bf Our contributions.}
Our starting point is a multi-head, attention-only transformer\footnote{Adding multi-layer perceptron (MLP) blocks constitutes a straightforward technical extension with limited conceptual novelty. We comment on its effects in Remark \ref{rem: mlp}.} with post-layer normalization. To model heterogeneity across both depth and heads, we assume that the key, query, and value matrices are independent and identically distributed (i.i.d.) across layers and heads. Strictly speaking, this assumption is valid only at initialization, since training induces correlations between weights at different depths.

Our primary motivation for this choice is to reflect the structural variability present in practical architectures: distinct layers and heads are parameterized independently, and there is no a priori reason to expect identical operators to be applied repeatedly along depth. In contrast to prior theoretical works that assume tied weights across layers—leading to deterministic iterated maps, our setting allows each layer to implement a genuinely different transformation. As a consequence, the depth dynamics are no longer described by repeated application of a fixed operator, and the mechanisms underlying representation collapse are \emph{a priori} unclear. In particular, when the layer-to-layer evolution is heterogeneous and potentially chaotic, it is not obvious whether the clustering phenomena observed in tied-weight models should persist. One of our main findings is that, perhaps surprisingly, clustering and rank collapse can still arise under appropriate scaling, even in this fully heterogeneous regime.

A secondary motivation concerns initialization. Since independence across layers holds exactly at initialization, the large-depth limit we analyze may be interpreted as describing the intrinsic signal propagation properties of the architecture prior to training. This perspective isolates the role of architectural scaling and weight initialization in shaping representation dynamics, and clarifies which collapse phenomena are structural consequences of depth itself, rather than artifacts of optimization or weight sharing.

Our main tool is a limiting \emph{homogenized} transformer model given by a stochastic differential equation (SDE) on $(\S^{d-1})^n$ driven by a \emph{common} cylindrical noise; see Theorem~\ref{thm:weak_error_clean}. Homogenized transformers distill the complex interaction of random weight matrices across layers into the sum of a common drift and a noise term that is independent across layers. This results in a mathematically tractable object to study oversmoothing and guide weight initialization; see Theorems~\ref{thm:clustering_random_init}, \ref{thm:clustering_small_beta} and \ref{thm:large_beta_meta}.

\medskip
\paragraph{\bf Related work.}
On the specific question of hyperparameter scaling and initialization, our results connect to several recent approaches. We briefly review the most relevant contributions here and return to a detailed comparison after presenting our main results.
Signal propagation at initialization is analyzed in \cite{noci2022signal}, where \emph{rank collapse}—a clustering phenomenon in which token representations concentrate in a low-dimensional subspace as they propagate through layers—is identified as a central obstruction in deep Transformers. Appropriate hyperparameter scalings are derived to mitigate this degeneracy.
The contraction properties of a single attention block are studied in \cite{cowsik2025geometric}, where a phase transition in the forward dynamics is exhibited. A corresponding phase transition for gradient scaling in the backward pass is also identified, leading to the prescription of initializing Transformers at the intersection of the two critical regimes—the so-called \emph{edge of chaos}. Empirical results corroborate this principle.
More recently, \cite{bordelon2024depthwise, dey2025don, chizat2025hidden, chaintron2026resnets} propose initialization scalings designed to escape the \emph{lazy} regime, in which parameters move negligibly during training and the model behaves effectively as a random feature method. This perspective is closely connected to mean-field, NTK, and $\mu$P scaling theories \cite{chizat2025hidden, jacot2018ntk, yang2021mup}, which formalize the distinction between linearized and feature-learning training dynamics.

Our perspective is different. We focus exclusively on the forward pass as a structural lens on deep attention architectures. Avoiding oversmoothing or rank collapse at initialization is a prerequisite for meaningful feature learning in deep attention stacks, even though the classical \emph{feature learning vs.\ lazy} dichotomy s concerns training dynamics and can, in principle, arise independently of rank collapse.

Beyond concrete prescriptions for initialization scaling, our results clarify how signals propagate through deep stacks of inhomogeneous attention layers. In particular, we show that an effective drift driving tokens toward low-dimensional representations can emerge purely from diversity across layers, even in the absence of explicit contraction at the single-layer level. Furthermore, our homogenization results indicate that Transformer models with tied weights across layers provide a principled approximation to trained deep Transformers, offering theoretical support for this commonly used simplification.

\section{Main results}

In this section, we introduce our model of random attention-only Transformers and informally derive the \emph{homogenized transformer} as an appropriate scaling limit. We state our main results while deliberately suppressing certain technical definitions and assumptions in order to highlight the core ideas and mechanisms. Complete formulations and detailed derivations are deferred to the subsequent sections.

\subsection{Setup}

We study encoder, attention-only Transformers in which attention weights are i.i.d. across heads and layers. Since training dynamics is not our focus, we collapse the usual query–key pair into a single matrix $\bA = Q^\top K$ and write the parameters of head $h$ in layer $\ell$ as
\[
\bm\theta_h^\ell = (\bV_h^\ell, \bA_h^\ell)
\;\overset{\mathrm{i.i.d.}}{\sim}\;
\uprho_* \in \mathcal{P}(\Uptheta),
\qquad
\Uptheta = (\R^{d \times d})^2,
\]
for $\ell \in [L]$ and $h \in [H]$, where $L$ is the depth and $H$ the number of heads.

Token embeddings $x_i^\ell \in \mathbb{S}^{d-1}$ evolve from one layer to the next according according to the update rule
\begin{equation}\label{eq:update_tokens}
x_i^{\ell+1}
=
\mathsf{N} \left(
x_i^\ell
+
\frac{\eta}{H}
\sum_{h=1}^H
B_{\bm\theta_h^\ell} \left[\mu_{X^\ell}\right](x_i^\ell)
\right),
\end{equation}
where $\mathsf{N}(x) = x/\|x\|$ is a normalization layer and $\eta>0$ is the residual scale, and
\[
\mu_{X^\ell}
=
\frac{1}{n}
\sum_{k=1}^n
\delta_{x_k^\ell}
\]
is the empirical measure of tokens at layer~$\ell$. The attention-induced velocity field is
\begin{equation}\label{EQ:VELOCITY_FIELD_SELF_ATTENTION}
B_{\bm\theta}[\mu](x)
=
\frac{1}{\mathscr{Z}_{\bA}[\mu](x)}
\int
e^{\beta \langle \bA x, y\rangle}
\bV y \, \mu(\de y),
\end{equation}
with normalizing constant
\[
\mathscr{Z}_{\bA}[\mu](x)
=
\int
e^{\beta \langle \bA x, y\rangle}
\, \mu(\de y).
\]

\subsection{Homogenization}

Homogenization shows that, under the appropriate scaling, the discrete multi-head, multi-layer random-attention updates converge to a simple effective \emph{drift--diffusion} dynamics on the sphere: the mean attention field produces a deterministic drift, while randomness averages into a (possibly common) stochastic forcing. Below we give an informal derivation, state the scaling regimes, and present the rigorous limit.

To isolate the mechanism, in this section we ignore the spherical projection and work in $\R^d$ where the update rule~\eqref{eq:update_tokens} becomes
\[
x_i^{\ell+1}
=
x_i^\ell
+
\frac{\eta}{H}
\sum_{h=1}^H
B_{\bm\theta_h^\ell}[\mu_{X^\ell}](x_i^\ell).
\]

\subsubsection{Drift–fluctuation decomposition}
Write each head as the sum of its expected value and a centered random vector (r.v.):
\[
B_{\bm\theta_h^\ell}[\mu](x)
=
b_{\uprho_*}[\mu](x)
+
\xi_{\bm\theta_h^\ell}[\mu](x),
\qquad
b_{\uprho_*}[\mu](x):=\E_{\bm\theta\sim\uprho_*}B_{\bm\theta}[\mu](x).
\]
Averaging over heads gives
\[
x_i^{\ell+1}
=
x_i^\ell
+
\eta\, b_{\uprho_*}[\mu_{X^\ell}](x_i^\ell)
+
\frac{\eta}{H}
\sum_{h=1}^H
\xi_{\bm\theta_h^\ell}[\mu_{X^\ell}](x_i^\ell).
\]
The last term is an average of i.i.d.\ centered r.v. that we write as
\[
\sum_{h=1}^H
\xi_{\bm\theta_h^\ell}[\mu_{X^\ell}](x_i^\ell)
\approx
\sqrt{H}\,\tilde\sigma[\mu_{X^\ell}](x_i^\ell)\,\zeta^\ell,
\]
where $\zeta^\ell$ a is a centered and isotropic random vector and $\tilde \sigma$ is a volatility matrix that satisfies for any $i,j=1, \ldots, n$
$$
\tilde \sigma[\mu_{X^\ell}](x_i^\ell) \tilde \sigma[\mu_{X^\ell}(x_j^\ell)]^\top = \E_{\bm\theta\sim\uprho_*}\left[ \xi_{\bm\theta_h^\ell}[\mu_{X^\ell}](x_i^\ell)\xi_{\bm\theta_h^\ell}[\mu_{X^\ell}](x_j^\ell)^\top\right].
$$
For $H$ large enough, we can think of $\zeta^\ell$ as Gaussian from the Central Limit Theorem (CLT) but the argument outlined below is also valid for $H=1$. 
Indeed, Gaussianity will emerge from a CLT applied to a sum across \emph{layers} rather than heads. 

With this notation, the update takes the suggestive form
\[
x_i^{\ell+1}
=
x_i^\ell
+
\eta\, b_{\uprho_*}[\mu_{X^\ell}](x_i^\ell)
+
\frac{\eta}{\sqrt{H}}
\tilde\sigma[\mu_{X^\ell}](x_i^\ell)\,\zeta^\ell.
\]

\subsubsection{Formal derivation of the scaling limit}

We now think of the layers as a time discretization with step-size $\eta$  of a continuous time process $\{X(t)\}_{t\ge 0}$. The above display yields
$$
x_i(t+\eta)
=
x_i(t)
+
\eta\, b_{\uprho_*}[\mu_{X(t)}](x_i(t))
+
\frac{\eta}{\sqrt{H}}
\tilde\sigma[\mu_{X(t)}](x_i(t))\,\zeta(t) ,
$$
When passing to the limit $\eta \to 0$, to obtain a non-degenerate Gaussian white noise, we need to rescale the stochastic term to be of order $\sqrt{\eta}$. To that end, we will establish a uniform variance bound $\upsigma^2$ for the resulting process in \eqref{eq: defining.alpha}, and define
\begin{equation*}\label{eq:alpha_def_short}
\alpha:=\frac{\eta\upsigma^2}{H},
\end{equation*}
so that
$$
x_i(t+\eta) = x_i(t)
+ \eta\, b_{\uprho_*}[\mu_{X(t)}](x_i(t)) + \sqrt{\alpha} \frac{\tilde\sigma[\mu_{X(t)}](x_i(t))}{\upsigma}\,\sqrt{\eta}\zeta(t),
$$
These iterations correspond to a time discretization of the system of stochastic differential equations (SDE)
\begin{equation}
    \label{eq:SDE_noproj}
    \ud x_i(t)  =  b_{\uprho_*}[\mu_{X(t)}](x_i(t)) \ud t 
+  \frac{\tilde\sigma[\mu_{X(t)}](x_i(t))}{\upsigma} \sqrt{\alpha}  \ud \mathsf{W}(t)
\end{equation}
for $i\in[n]$, where $\{\mathsf{W}(t)\}_{t\geq0}$ is a Wiener process common to all particles.

We introduce the \emph{macroscopic time}
\[
t_L \coloneqq \eta L,
\]
which corresponds to $L$ discrete iterations with step size $\eta$.
Over the interval $[0,t_L]$, the drift accumulates to order $t_L$,
while the stochastic term accumulates to order $\sqrt{t_L\alpha}$.
Comparing these two contributions yields the following three limiting regimes:
\medskip

{
\begin{center}
\begin{tabular}{c|c|c|c}
    \textbf{Regime} & \textbf{Scaling} & \textbf{Noise behavior} & \textbf{Limit} \\
    \midrule
    Ballistic & $t_L \alpha = o(1)$ & Noise vanishes & ODE \\
    Diffusive & $t_L \alpha = \Theta(1)$ & Noise persists (order one) & SDE \\
    Super-diffusive & $t_L \alpha \gg 1$ & Noise dominates & Noise
\end{tabular}
\end{center}}
\medskip

\noindent
As we shall see, in the presence of the drift term this approximation
remains valid up to macroscopic times $t_L = O(1)$; see Theorem~\ref{thm:weak_error_clean} below.
In contrast, in the absence of drift, such as in the centered  case treated
in Section~\ref{sec:gaussian}, we only need $t_L\alpha=O(1)$ and the approximation can remain valid even for $t_L \to \infty$.
The special case $t_L = \Theta(1)$, which arises when $\eta = 1/L$, 
serves as a central and particularly illustrative example.

\begin{figure}[h]
    \centering
     \includegraphics[scale=0.35]{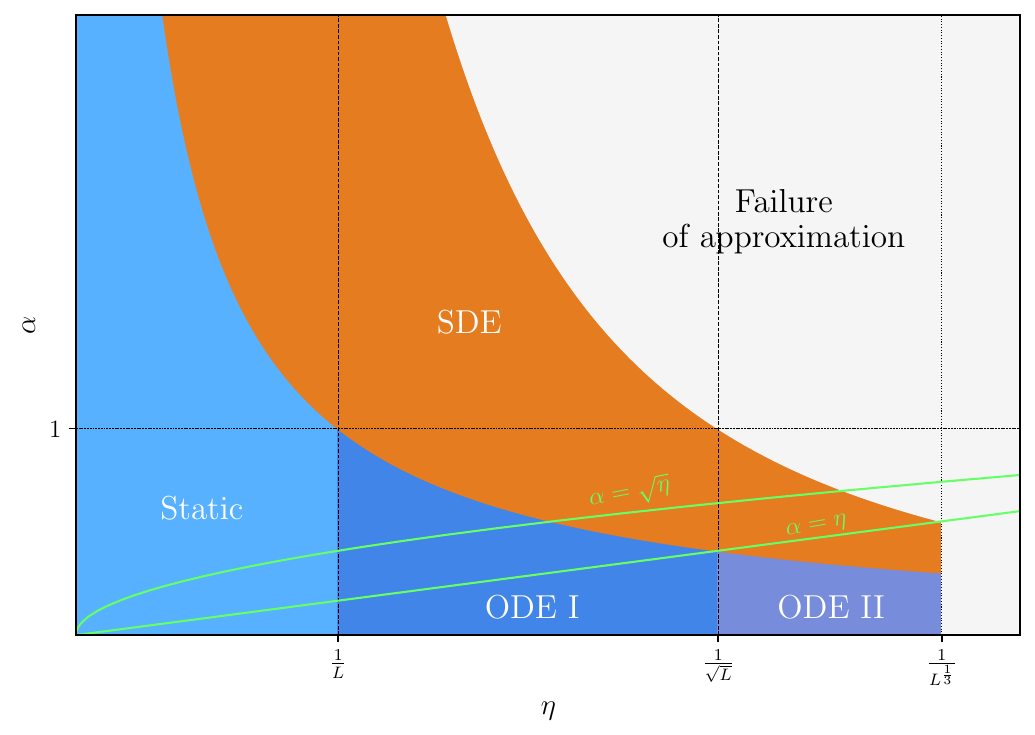}
     \hspace{0.1cm}
     \includegraphics[scale=0.35]{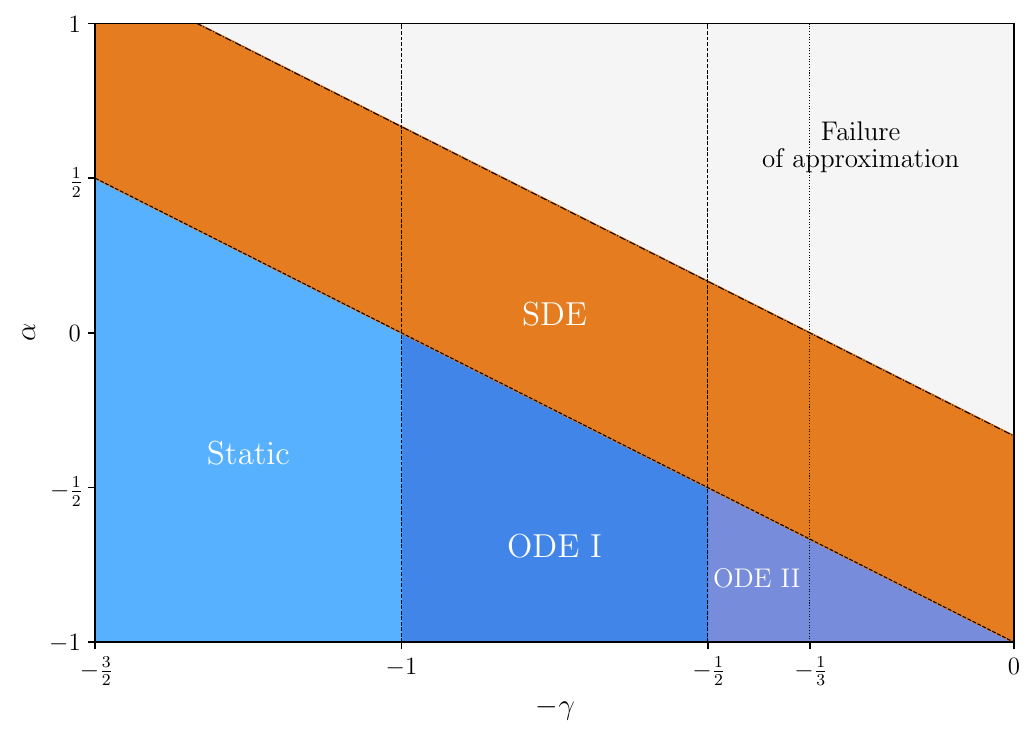}
    \caption{Qualitative phase diagrams in the $(\eta,\alpha)$-plane for fixed depth $L$ (left), together with the same picture in logarithmic coordinates (right), where straight lines correspond to power-law scalings of $\eta$ and $\alpha$ with respect to $L$. 
    Different regions correspond to the rigorous regimes obtained from Theorem~\ref{thm:weak_error_clean}. Specifically, ``ODE~I'' (Corollary~\ref{cor:ode1}) yields \eqref{eq: deterministic} in the ballistic regime $\alpha\eta L=o(1)$ and $\eta^2L=o(1)$; ``ODE~II'' (Corollary~\ref{cor:ode2}) yields \eqref{eq: deterministic.modified} in the refined deterministic regime $\alpha\eta L=o(1)$ and $\eta^3L=o(1)$; and ``SDE'' (Corollary~\ref{cor:SDE}) yields the homogenized SDE \eqref{eq:SDE_ito_clean} in the diffusive regime $\alpha\eta L=O(1)$ and $\eta^3L=o(1)$. The region labelled ``Static'' corresponds to $\eta L=o(1)$, so that no non-trivial macroscopic evolution is seen on the time scale $t_L=\eta L$, whereas ``Failure of approximation'' indicates scalings not covered by the present weak-approximation argument. We study the qualitative behavior of the homogenized model in the setting of centered Gaussian weights, for which only the diffusive regime yields non-stationary evolution---see \eqref{eq:Diffusive_gaussian_case}.}
    \label{fig: phase-diags}
\end{figure}

The projection operator $\mathsf{N}(\cdot)$ entails an It\^o correction to the drift and the SDE~\eqref{eq:SDE_noproj} becomes

\begin{align*}
\de x_i(t)= \Bigg(\proj_{x_i(t)} &b_{\uprho_*}[\mu_{X(t)}](x_i(t)) - \frac{\alpha}{2\upsigma^2}
\E\left\|\proj_{x_i(t)}\tilde\sigma[\mu_{X(t)}](x_i(t))\zeta^1 \right\|^2 x_i(t) \Bigg) \de t \\
&+ \frac{\sqrt{\alpha} }{\upsigma}\proj_{x_i(t)}
\tilde\sigma[\mu_{X(t)}](x_i(t))\de \mathsf{W}(t)\,,
\end{align*}
where  $\proj_x=I_d-xx^\top$ denotes the projection onto the tangent space of the sphere $\S^{d-1}$ at $x$. This process corresponds to the \emph{stochastic modified equation} introduced by~\cite{li2017sme, li2019sme} as the scaling limit of stochastic gradient descent. Following~\cite{gess2024stochastic, gess2024rsgd, gess2025conservative}, we study the \emph{stochastic modified flow}, a different SDE with the same time marginals but that lends itself more readily to mean-field limits as $n \to \infty$.

To obtain the stochastic modified flow, we rewrite the diffusion term in a form that reflects its origin from the fluctuations $\xi_{\bm\theta}$. Recall that the covariance structure of the vector $\tilde\sigma[\mu](x)\,\zeta^1$ is given by
\[
\int_{\Uptheta}
\xi_{\bm\theta}[\mu](x)\,\xi_{\bm\theta}[\mu](x)^\top
\,\uprho_*(\de\bm\theta).
\]
Equivalently, this Gaussian can be represented as a stochastic integral against a cylindrical Wiener process.

More precisely, let $(\mathsf W(\cdot, t))_{t\geq0}$ be a cylindrical Wiener process on $L^2(\uprho_*)$; see Section~\ref{sec: cylindrical.wiener} for details. Then one can write
\[
\tilde\sigma[\mu](x)\,\de \mathsf W(t)
\;\equiv\;
\int_{\Uptheta}
\xi_{\bm\theta}[\mu](x)
\,\mathsf W(\de\bm\theta,\de t),
\]
in the sense that both sides define Gaussian processes with the same covariance. Using this representation, we arrive at the limiting system of It\^o SDEs
\begin{align}
\de x_i(t)
=
\proj_{x_i(t)}
\Bigg(
&b_{\uprho_*}[\mu(t)](x_i(t))\,\de t
+
\frac{\sqrt{\alpha} }{\upsigma}
\int_{\Uptheta}
\xi_{\bm\theta}[\mu(t)](x_i(t))
\,\mathsf W(\de\bm\theta,\de t)
\Bigg) \nonumber \\
&
-\frac{\alpha}{2\upsigma^2}
x_i(t)
\int_{\Uptheta}
\left\|
\proj_{x_i(t)}
\xi_{\bm\theta}[\mu(t)](x_i(t))
\right\|^2
\uprho_*(\de\bm\theta)\,\de t\,  \label{eq: first.sde}
\end{align}
for all $i\in[n]$.

\subsubsection{Approximation results}

We are now in a position to state our rigorous approximation results under moment conditions on the weight distribution~$\uprho_*$.

\begin{assumption}\label{ass:high_order_short}
Whenever $(\bA,\bV)\sim\uprho_*$, we assume that $\bA-\E\bA=\bW\bW'^\top$ is independent from $\bV$. Moreover, there exists $C>0$ such that $\{(\bV-\E\bV)_{ij}\}_{ij}$ and $\{(\bW)_{ij},(\bW')_{ij}\}_{ij}$ are independent subGaussian with variance proxy $C\sigma_{\bV}^2$  and $C\sigma_{\bA}^2$ respectively.
\end{assumption}

Consider the piecewise constant process $X^\eta(t)\coloneqq X^{\lfloor t/\eta\rfloor}\in(\S^{d-1})^n$ obtained by interpolating the discrete updates in \eqref{eq:update_tokens}, as well as the limiting object $X(t)=(x_1(t),\dots,x_n(t)) \in (\S^{d-1})^n$ as the solution of the system of It\^o SDEs~\eqref{eq: first.sde}.
We have the following error bound at time $T=t_L=\eta L$. 

\begin{theorem}\label{thm:weak_error_clean}
Under Assumption~\ref{ass:high_order_short}, for any $\varphi\in C^4((\S^{d-1})^n)$ the following approximation holds uniformly until macroscopic time $t_L=\eta L$:
\[
\sup_{t \in [0, t_L]}\big|
\E\varphi(X(t))-\E\varphi(X^\eta(t))
\big|
\le
C e^{Ct_L}\,\eta (t_L+1)\,\max(1,\alpha),
\]
where $C\geq 1$ depends on $\|\varphi\|_{C^4}$ but not on $\eta,\alpha,L$.
\end{theorem}

Thus, as $\eta\to0$ and $L\to\infty$ with macroscopic time $t_L=\eta L$ fixed, the discrete attention dynamics \eqref{eq:update_tokens} \emph{homogenize} to a drift--diffusion particle system on the sphere. Full definitions and proofs appear in Sections~\ref{sec:main_result} and~\ref{sec: proof.main}. 

In fact, the proof of Theorem~\ref{thm:weak_error_clean}  shows that the discrete chain is approximated by \eqref{eq:SDE_ito_clean}---which differs from \eqref{eq: first.sde} only by an additional modified drift which is uniformly of order $O(\eta)$---at an improved rate $C e^{Ct_L}\eta (t_L+1)\max(\eta,\alpha)$, and the statement on the more illustrative homogenized model is a consequence thereof.

The derivation of Theorem~\ref{thm:weak_error_clean} draws upon seminal work on the martingale problem \cite{ethier2009markov,stroock2007multidimensional} and developments in stochastic modified equations \cite{li2017sme}. 
More specifically, our proof follows the argument of \cite{gess2024rsgd} with the main technical novelty arising from the presence of a common cylindrical noise.

\begin{remark}[Multi-layer perceptrons] \label{rem: mlp}
We expect the same homogenization argument to naturally extend to updates of the form
\begin{equation*}
x_i^{\ell+1}
=
\mathsf{N}\left(
x_i^\ell
+
\frac{\eta}{H}\sum_{h=1}^H B_{\bm\theta_h^\ell}[\mu_{X^\ell}](x_i^\ell)
+
\lambda_\eta\,M_{\bm\phi^\ell}(x_i^\ell)
\right),
\end{equation*}
where $M_{\bm\phi}(x)=U\psi(Wx)$ with \(\psi\) a smooth activation applied component-wise, and
\[
\bm\phi^\ell=(U^\ell,W^\ell)\overset{\mathrm{i.i.d.}}{\sim}\varrho_* \in \mathcal{P}(\Upphi),
\qquad
\Upphi:=\R^{d\times m}\times \R^{m\times d}.
\]
Indeed, set
\[
m_{\varrho_*}(x)\coloneqq\int_{\Upphi} M_{\bm\phi}(x)\,\varrho_*(\de\bm\phi),
\qquad
\zeta_{\bm\phi}(x)\coloneqq M_{\bm\phi}(x)-m_{\varrho_*}(x),
\]
and
\[
\uptau^2\coloneqq
\max_{X\in(\S^{d-1})^n}
\int_{\Upphi}
\sum_{i=1}^n
\left\|
\proj_{x_i}\zeta_{\bm\phi}(x_i)
\right\|^2
\,\varrho_*(\de\bm\phi).
\]
Then the two relevant scales on the macroscopic time variable \(t=\eta\ell\) are $\lambda_\eta/\eta$ and $\lambda_\eta^2\uptau^2/\eta$.
The first one controls the deterministic contribution of the MLP, whereas the second one controls its fluctuations. In particular:
\begin{itemize}
    \item if \(\lambda_\eta=\kappa\eta\), then the MLP contributes only an additional drift resulting in the homogenized model
    \begin{align*}
    \de x_i(t)
    =\proj_{x_i(t)}
    \Bigg(
    \Big(
    &b_{\uprho_*}[\mu(t)](x_i(t))
    +
    \kappa m_{\varrho_*}(x_i(t))
    \Big)\de t
    \\
    &+
    \frac{\sqrt{\alpha}}{\upsigma}
    \int_{\Uptheta}
    \xi_{\bm\theta}[\mu(t)](x_i(t))
    \,\mathsf W(\de\bm\theta,\de t)\\
    &-\frac{\alpha}{2\upsigma^2}
    x_i(t)
    \int_{\Uptheta}
    \left\|
    \proj_{x_i(t)}
    \xi_{\bm\theta}[\mu(t)](x_i(t))
    \right\|^2
\uprho_*(\de\bm\theta)\,\de t.
\end{align*}
    \smallskip
    \item if \(m_{\varrho_*}\equiv0\) and \(\lambda_\eta^2\uptau^2/\eta\to\alpha_{\mathrm{mlp}}\in(0,\infty)\), then the MLP contributes an additional common-noise term, resulting in the homogenized model
    \begin{align*}
\de x_i(t)
=\proj_{x_i(t)}
\Bigg(b_{\uprho_*}[\mu(t)]&(x_i(t))\de t+
\frac{\sqrt{\alpha}}{\upsigma}
\int_{\Uptheta}
\xi_{\bm\theta}[\mu(t)](x_i(t))
\,\mathsf W(\de\bm\theta,\de t)\\
&+
\frac{\sqrt{\alpha_{\mathrm{mlp}}}}{\uptau}
\int_{\Upphi}
\zeta_{\bm\phi}(x_i(t))
\,\widetilde{\mathsf W}(\de\bm\phi,\de t)
\Bigg)\\
&-\frac{\alpha}{2\upsigma^2}
x_i(t)
\int_{\Uptheta}
\left\|
\proj_{x_i(t)}
\xi_{\bm\theta}[\mu(t)](x_i(t))
\right\|^2
\uprho_*(\de\bm\theta)\,\de t\\
&-\frac{\alpha_{\mathrm{mlp}}}{2\uptau^2}
x_i(t)
\int_{\Upphi}
\left\|
\proj_{x_i(t)}
\zeta_{\bm\phi}(x_i(t))
\right\|^2
\varrho_*(\de\bm\phi)\,\de t,
\end{align*}
where \(\mathsf W,\widetilde{\mathsf W}\) are independent cylindrical Wiener processes. The effective noise level would now be measured by \(t_L(\alpha+\alpha_{\mathrm{mlp}})\). 
\end{itemize}
\end{remark}

\subsection{The Gaussian case}
\label{sec:gaussian}
We can obtain sharper results when we assume that the weight matrices are exactly Gaussian rather than subGaussian. In this section, we let ${\uprho_*}\in \cP(\Uptheta)$ be such that
\begin{equation} \label{eq: tformers.at.initialization}\tag{G}
\bV_{i,j}\overset{\textrm{i.i.d.}}{\sim}\mathcal{N}(0,\sigma_{\bV}^{2}),\quad
\bA=\bW\bW'^{\sT},\quad
\bW_{i,j},\bW'_{i,j}\overset{\textrm{i.i.d.}}{\sim}\mathcal{N}(0,\sigma_{\bA}^{2}).
\end{equation}
\subsubsection{Approximation error}
Assumption~\eqref{eq: tformers.at.initialization}
has two important implications:
\begin{itemize}
    \item[(i)] The scaling parameter $\alpha$, rigorously defined in \eqref{eq:Alpha_Sec2} can be computed explicitly as
    $$
        \alpha = \frac{\eta}{H}\sigma_{\bV}^2 (d-1).
    $$
    This suggests the standard scaling \(\sigma_{\bV}^2=1/d\), which we adopt henceforth. We keep $\sigma_{\bA}^2$ general though we note that the same choice  $\sigma_{\bA}^2=1/d$ is common in practice for random Gaussian initialization \cite{he2015delving,vaswani2017attention,yang2021mup}.
    \smallskip 
    
   \item[(ii)] Because $\bV$ is centered, the drift vanishes:
\( b_{\uprho_*}[\mu] \equiv 0 \).
As a result, the diffusive regime, \( \alpha \eta L = \Theta(1),\)
exhibits non-trivial  dynamics that are driven entirely by random fluctuations. By contrast, in absence of random fluctuations, when $\alpha \eta L \to 0$, the ballistic regime becomes static in this setting.
\end{itemize}
As a consequence, the  term $ b_{\uprho_*}$ vanishes and only the It\^o correction remains in  the drift. After a time rescaling by $\upsigma$, the relevant homogenized model to analyze becomes
\begin{align}\label{eq:Diffusive_gaussian_case}
   \de x_{i}(t)=\sqrt{\alpha}\proj_{x_i(t)}\int_{\Uptheta}B_{\bm{\theta}}&[\mu(t)](x_i(t))\mathsf{W}(\de \bm{\theta},\de t)\nonumber\\
   & -\frac{\alpha}{2}x_{i}(t)\int_{\Uptheta}\left\|\proj_{x_i(t)}B_{\bm{\theta}}[\mu(t)](x_i(t))\right\|^{2}\uprho_*(\de \bm{\theta}),
\end{align}
where $B_{\bm\theta}[\mu](x)$ is the attention field defined in \eqref{EQ:VELOCITY_FIELD_SELF_ATTENTION}.
An adaptation of the proof of Theorem \ref{thm:weak_error_clean} to this case is presented below. Note that the approximation holds until a macroscopic time $t_L$ that can be of order $1/\alpha \gg 1$. 

\begin{corollary}\label{cor:weak_error_centered} 
Under \eqref{eq: tformers.at.initialization}, for any $\varphi\in C^4((\S^{d-1})^n)$ the following approximation holds uniformly until macroscopic time $t_L=\eta L$:
\[
\sup_{t \in [0, t_L]}\big|
\E\varphi(X(t))-\E\varphi(X^\eta(t))
\big|
\le
C e^{Ct_L\alpha }\,\eta (t_L+1)\,\max(\eta,\alpha),
\]
where $C\geq 1$ depends on $\|\varphi\|_{C^4}$ but not on $\eta,\alpha,L$.
\end{corollary}

The proof is in Section \ref{proof: gaussian.corollary}.

\subsubsection{Asymptotics of the stochastic modified flow}
By rescaling time, we can, without loss of generality, henceforth set $\alpha=1$.
To get a handle on \eqref{eq:Diffusive_gaussian_case} it is helpful to first consider the case $\upbeta=0$, in which case \eqref{eq:Diffusive_gaussian_case} simplifies to  
\begin{equation}\label{eq:decomposition_beta=0_ito}
\de x_i(t)=-\frac{d-1}{2}\kappa(t) x_i(t) \de t+\sqrt{\kappa(t)}\proj_{x_i(t)}\de B(t),
\end{equation}
where $\kappa(t)\coloneqq d^{-1}\|\int x\mu(t,\de x)\|^2$ and $B(t)$ is a Brownian motion in $\mathbb{R}^{d}$.  
Applying the Itô formula, we find
\begin{align*}
\de \<x_i(t),x_j(t)\> &= \kappa(t)d\left(1-\<x_i(t),x_j(t)\>+\frac{\<x_i(t),x_j(t)\>^{2}+\<x_i(t),x_j(t)\>-2}{d}\right) \\  &
    +\sqrt{\kappa(t)}\left(\left\langle \proj_{x_i(t)}\de B(t),x_j(t)\right\rangle+\left\langle \proj_{x_j(t)}\de B(t),x_i(t)\right\rangle\right).
    \end{align*}
A direct computation shows that the quadratic variation vanishes as $d\to\infty$: 
\[
\left[\sqrt{\kappa(\cdot)}\left(\left\langle \proj_{x_i(\cdot)}\de B(\cdot),x_j(\cdot)\right\rangle+\left\langle \proj_{x_j(\cdot)}\de B(\cdot),x_i(\cdot)\right\rangle\right)\right]_t\lesssim \int_{0}^{t}\kappa(s)\de s \lesssim \frac{t}{d}.
\]
As the quadratic variation controls the magnitude of the martingale part of the process (this follows by virtue of the Burkholder-Davis-Gundy inequality, for example), as $d\rightarrow \infty$, the inner products of \eqref{eq:decomposition_beta=0_ito} evolve as 
\[
\frac{\de}{\de t} \<x_i(t),x_j(t)\> = \kappa(t)d(1-\<x_i(t),x_j(t)\>) +o(1)
\]
Summing across indices, we find
\begin{equation*}
    \frac{\de}{\de t}\left\|\int x\mu(t, \de x)\right\|^{2}=\left\|\int x \mu(t,\de x)\right\|^{2}\left(1-\left\|\int x \mu(t,\de x)\right\|^{2}\right) +o(1).
\end{equation*}
This is a logistic ODE for $u(t)=d\kappa(t)$ with an explicit solution that converges exponentially fast to $1$ as $t\to\infty$---see \eqref{eq: logistic.sol}.

We thus use \eqref{eq: tformers.at.initialization} to derive explicit equations describing representation collapse by taking scaling limits in terms of the context length $n$, the inverse temperature $\upbeta$, and the dimension $d$.
This is in line with a recurring theme in statistical physics, wherein one replaces 
high-dimensional microscopic dynamics by  few macroscopic \emph{order
parameters}, and then studies their effective dynamics \cite{saad1995line,benarous2024cpa}.  
The appearance of the simplifying equation for the inner products when $\upbeta=0$ is not accidental. Generally in the setting of \eqref{eq:Diffusive_gaussian_case}, the natural order parameter is the Gram matrix
\begin{equation*}
R(t)=
\begin{bmatrix}
\langle x_1(t), x_1(t)\rangle &\ldots &\langle x_1(t), x_n(t)\rangle\\
\vdots & \ddots & \vdots \\
\langle x_n(t), x_1(t)\rangle & \ldots & \langle x_n(t), x_n(t)\rangle 
\end{bmatrix},
\end{equation*}
since---by rotational invariance and Gaussian conditioning---both the drift and the conditional covariance can be expressed as functions of \(R(t)\).
This yields a closed ("self-consistent") stochastic dynamics on \(R(t)\), but it remains analytically intractable. One idea would be to study the evolution of its expectation \(\E R(t)\) but the latter is not closed because of the nonlinear drift. To circumvent this limitation, we study three specific regimes for which useful statistics admit a closed evolution, summarized below:
\medskip

{\small
\begin{center}
    \begin{tabular}{c|c|c|c}
        \textbf{Data regime} &  \textbf{Scaling} & \textbf{Behavior} &
        \textbf{Reference} \\
        \midrule
        simplex & $\frac{\upbeta^2 n}{d}=O(1)$ \& $\frac{\log n}{d}=o(1)$ & $\dot{u} = u(1-u)(1-f(u))$ &   Theorem \ref{thm:clustering_random_init} \\
       hemisphere & $\upbeta = O(d^{-\frac12})$ & $\dot{u} = u(1-u)$ & Theorem \ref{thm:clustering_small_beta}  \\
       mean-field & $d\upbeta^{-\frac12}=o(1)$ & slow motion & Theorem \ref{thm:large_beta_meta} 
    \end{tabular}
    \end{center}
}
\medskip

\begin{remark}
    The common-noise mechanism behind Theorems \ref{thm:clustering_random_init}–\ref{thm:large_beta_meta} is related in spirit to the literature on synchronization by noise \cite{flandoli2017synchronization, cranston2016weak}. (See \cite{engel2026random} for a recent application in a machine learning context.) 
    We stress, however, that Theorems \ref{thm:clustering_random_init}–\ref{thm:clustering_small_beta} concern alignment of overlaps rather than pathwise synchronization. Theorem \ref{thm:large_beta_meta} studies the corresponding common-noise two-point motion.
\end{remark}

\subsubsection{Simplex data \& low temperature}

We begin by considering "simplex-like"  configurations, a standard setting in the literature \cite{cowsik2025geometric,chen2025criticalattentionscalinglongcontext,giorlandino2025two}. This simplification reduces the dynamics to a single parameter, making the analysis highly tractable. Despite its simplicity, the simplex model remains strongly predictive of several key practical performance metrics, including those related to weight initialization~\cite{olmo20242} and context length generalization~\cite{bai2023qwen, nakanishi2025ssmax, puvvada2025swangpt}.

Assume that the Gram matrix at initial time is $R(0)=(1-\gamma(0))I_n+\gamma(0)\mathbf 1\mathbf 1^\top$. In the absence of a stochastic term, by permutation equivariance of attention with Gaussian matrices~\cite{cowsik2025geometric},  $R(t)$ remains of the form $R(t)=(1-\gamma(t))I_n+\gamma(t)\mathbf{1}\mathbf{1}^{\top}$ for all $t$. In particular, the dynamics of the whole system can be summarized by the scalar $\gamma(t)$. In our case, despite the presence of a stochastic term, this statement remains approximately true as summarized in the next theorem.

\begin{theorem} \label{thm:clustering_random_init}
Fix $d, n\geq2$ and $\upbeta>0$, and assume $\langle x_i(0),x_j(0)\rangle = \gamma_0$ for some $\gamma_0\in(-1/(n-1), 1)$.  For any configuration $(x_1,\ldots,x_n)\in(\S^{d-1})^n$ with $\langle x_i, x_j\rangle=\gamma$ for $i\neq j$, define  
\begin{equation*} \label{eq:pi_simplex_def_selfcontained}
\pi_{i\to k}^{\bA}(\gamma)
\coloneqq
\frac{e^{\upbeta\langle \bA x_i, x_k \rangle}}
{\sum_{\ell=1}^n e^{\upbeta\langle \bA x_i, x_\ell\rangle}}.
\end{equation*}
Then 
\begin{align} \label{eq:f_g_def_selfcontained}
f(\gamma)
&\coloneqq \E\,\sum_{k=1}^n \left(\pi_{i\to k}^{\bA}(\gamma)\right)^2,
\nonumber\\
g(\gamma)
&\coloneqq \E\,\sum_{k=1}^n \pi_{i\to k}^{\bA}(\gamma)\,\pi_{j\to k}^{\bA}(\gamma),
\qquad (i\ne j),
\end{align}
are well-defined (i.e. independent of the particular simplex configuration and of the choice of indices),
and satisfy $0\le f(\gamma)\le 1$ and $0\le g(\gamma)\le 1$. Let
\begin{equation}\label{eq:Phi_def_selfcontained}
b(\gamma)
\coloneqq \gamma+(1-\gamma) g(\gamma)
-\gamma\left(\gamma+(1-\gamma) f(\gamma)\right),
\end{equation}
and $\gamma\in C^1(\R_{\geq0})$ be the solution of the Cauchy problem
\begin{equation} \label{eq:gamma_ODE_selfcontained}
\begin{cases}
\dot\gamma(t)=b(\gamma(t)) \\
\gamma(0)=\gamma_0.
\end{cases}
\end{equation}
Then there exists $C>0$ such that for every $T>0$ and every $\delta\in(0,1)$, with probability at least $1-\delta$, the solution to \eqref{eq:Diffusive_gaussian_case} satisfies
\begin{equation*}
\sup_{t\in[0,T]}\max_{i\ne j}|\langle x_i(t), x_j(t)\rangle-\gamma(t)|
\le e^{CT} 
\sqrt{\frac{8T}{d} \log\left(\frac{2n^2}{\delta}\right)}.
\end{equation*}
Moreover $C=O(1+d\sigma_{\bA}^{4}\upbeta^{2}n).$ 
\end{theorem}

The proof can be found in {Section  \ref{proof: simplex}}. The simplex symmetry forces the interaction terms to be essentially governed by the Gram matrix, so the stochastic drift is close to an explicit deterministic one; a  Lipschitz stability estimate around this simplex manifold then propagates small errors in the Gram matrix, while standard martingale concentration bounds control the accumulated noise. 

\begin{figure}[h]
    \centering
\includegraphics[scale=0.57]{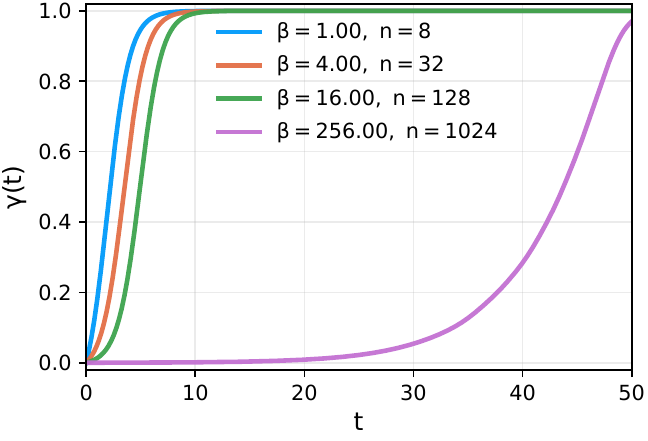}
\includegraphics[scale=0.57]{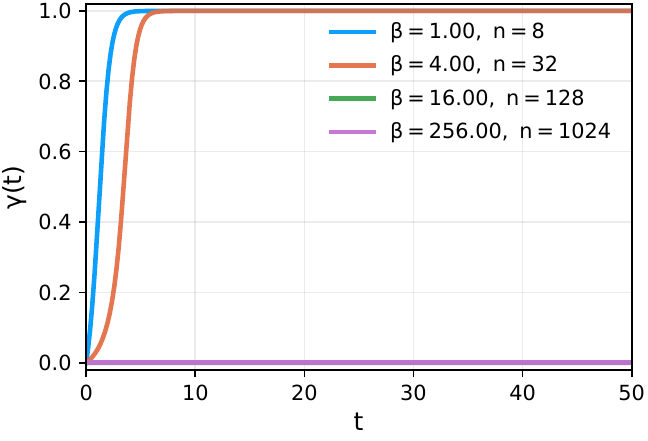}
    \caption{Comparison between the curve $\gamma$ in Theorem \ref{thm:clustering_random_init} ({\bf left}) and that of \cite[Theorem 2.13]{geshkovski2025mathematical} for $d=128$ ({\bf right}).}
\label{fig: logistic}
\end{figure}

Theorem \ref{thm:clustering_random_init} can be seen as a stochastic extension of \cite[Theorem 2.13]{geshkovski2025mathematical}. While the right hand side is not precisely the same, it leads to the same long-time behavior.
Indeed, assume for simplicity that $\gamma_0\in(0,1)$.
Since $n\ge2$, for every realization of $\bA$ and every $\gamma<1$ we have
$\pi_{i\to k}^{\bA}(\gamma)\in(0,1)$ and $\sum_{k=1}^n\pi_{i\to k}^{\bA}(\gamma)=1$, hence
\[
\sum_{k=1}^n \bigl(\pi_{i\to k}^{\bA}(\gamma)\bigr)^2 < 1.
\]
Taking expectations yields $f(\gamma)<1$ for all $\gamma<1$. Together with $g(\gamma)\ge 0$,
this implies that for every $\gamma\in(0,1)$,
\[
b(\gamma)
=(1-\gamma)\bigl(\gamma+g(\gamma)-\gamma f(\gamma)\bigr)
\ge \gamma(1-\gamma)\bigl(1-f(\gamma)\bigr)
>0.
\]
Therefore the solution $\gamma(\cdot)$ to \eqref{eq:gamma_ODE_selfcontained} is strictly
increasing, remains in $(0,1)$, and satisfies $\gamma(t)\nearrow 1$ as $t\to\infty$. 

Furthermore, when $\upbeta$ is large, attention for token $i$ concentrates on a single index $k^{\star}_i$,
so $f(\gamma)\approx 1$.
Moreover, $g(\gamma)\approx \mathbb{P}(k_i^\star=k_j^\star)$, which is typically of order $1/n$ because of independence of attention weights.
Consequently the drift $b(\gamma)\approx O((1-\gamma)/n )$ becomes small, leading to slow evolution of $\gamma(t)$, as seen in Figure \ref{fig: logistic}.

\begin{remark}
In the proof of Theorem \ref{thm:weak_error_clean}, we defined \(\alpha\) via the worst-case bound in \eqref{eq:Alpha_Sec2},
using the uniform estimate \(\mathbb{E}\|\xi_{\bm\theta}[\mu_X](\cdot)\|^2 \le \upsigma^2\) 
over all configurations \(X\in(\S^{d-1})^n\).
This choice is convenient for a global weak-error estimate, but it can be overly pessimistic.
In the present   simplex setting one may instead hope to characterize a configuration-dependent typical noise level by removing the supremum and studying \(\int_{0}^{T}\mathbb{E}\|G(X(t),\bm\theta)\|^2\de t\). We leave such a refinement for future work.
\end{remark}

\subsubsection{High temperature \& dimension limit}

We now relax the simplex assumption on the initial configuration, and do so in the high temperature regime.

\begin{theorem} \label{thm:clustering_small_beta}
Fix $d, n\ge 2$ and $\upbeta\in[0,1]$, and assume that $\langle x_i(0),x_j(0)\rangle\ge \gamma$ for $i\neq j$ and $\gamma\in(0, 1)$. Define the empirical mean
\begin{equation*} \label{eq:def_m_smallbeta}
m(t)\coloneqq\frac{1}{n^2}\sum_{i=1}^n\sum_{j=1}^n \langle x_i(t),x_j(t)\rangle.
\end{equation*}
Let $u$ be the unique solution to the  Cauchy problem for the logistic equation
\begin{equation} \label{eq:logistic_u_smallbeta}
\begin{cases}
\dot{u}(t)=u(t)(1-u(t)) \\
u(0)=m(0).
\end{cases}
\end{equation}
Then there exists a universal constant $C>0$ such that for every every $T>0$ and every $\delta\in(0,1)$, with probability at least $1-\delta$, the solution to \eqref{eq:Diffusive_gaussian_case} satisfies
\begin{equation}\label{eq:smallbeta_main_bound}
\sup_{t\in[0,T]}\left|m(t)-u(t)\right|
\le Ce^{CT}\max\left(\frac{1}{d},\,d^{2}\sigma_{\bA}^{4}\upbeta^{2}\right)
+ C\sqrt{\frac{T}{d}\log\left(\frac{4}{\delta}\right)}.
\end{equation}
\end{theorem}

The proof can be found in {Section  \ref{proof: small.beta}}, and is a perturbative argument around the case $\upbeta=0$ discussed around \eqref{eq:decomposition_beta=0_ito}.
Solving \eqref{eq:logistic_u_smallbeta} explicitly yields
\begin{equation} \label{eq: logistic.sol}
u(t)=\frac{u(0)}{u(0)+(1-u(0))e^{-t}},
\end{equation}
so $u(t)\to 1$ as $t\to\infty$. 
As such, for $\beta$ small, the model, even with random weights, is a perturbation of the original Kuramoto model \cite{geshkovski2025mathematical}.

\begin{figure}[h]
    \centering
    \includegraphics[scale=0.55]{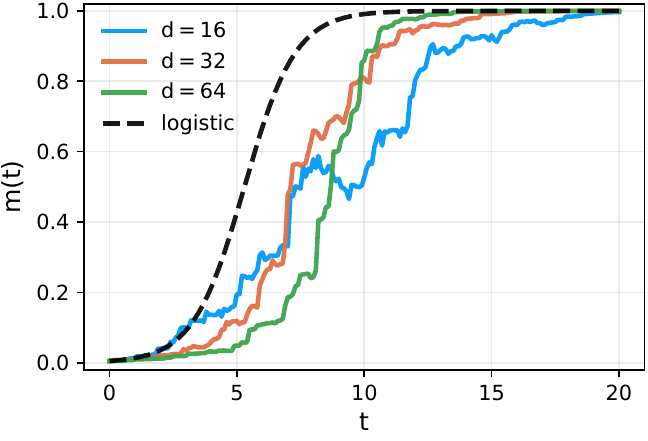}
    \includegraphics[scale=0.55]{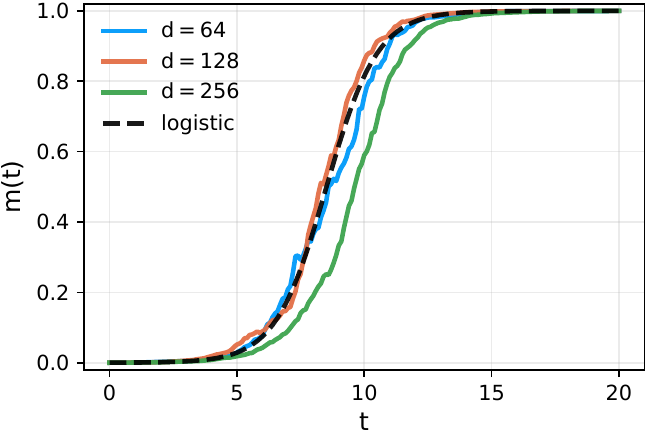}
    \caption{Theorem \ref{thm:clustering_small_beta} asserts that the second moment approaches the solution of the logistic equation \eqref{eq:logistic_u_smallbeta} when $d$ is large and $\upbeta$ is small. \textbf{Left:} 
    $n=200$ and $\upbeta=0.05$. \textbf{Right:} $n=5000$ and $\upbeta=0.001$.}
    \label{fig:placeholder}
\end{figure}

\subsubsection{Slow motion}

We conclude with the study of the limit $n\to\infty$, $\upbeta\to\infty$ and $d\to\infty$. In the mean-field limit $n\to\infty$, one needs to make sense of the limiting object. 
The empirical measure $\mu_{X(t)}$ of $X(t)=(x_i(t))_{i\in[n]}$ can be shown to converge, as $n\to\infty$, conditionally on the noise, to a probability measure $\mu(t)$ which solves a stochastic Fokker-Planck equation---this is done in Theorem \ref{thm:PoC_wellposedness}. 

In Lagrangian words, a particle $x(t)\in\S^{d-1}$ with law $\mu(t)$ satisfies the conditional McKean-Vlasov equation 
\begin{equation} \label{eq:non_linear_SDE_cylindrical}
\de x(t)=\proj_{x(t)}\int_{\Uptheta}B_{\bm{\theta}}[\mu(t)](x(t))\;\mathsf{W}(\de \bm{\theta}, \de t).
\end{equation}
We work in the following setting.

\begin{assumption} \label{ass:low-temperature}
For every $t\geq 0$, $\mu(t)\ll \upsigma_{d}$ where $\upsigma_{d}$ is the uniform measure on $\S^{d-1}$. Moreover $\rho(t)\coloneqq \frac{\de\mu(t)}{\de\upsigma_{d}}$ satisfies $\rho(t,\cdot)\in C^1(\S^{d-1})$ and 
\[
0<\rho_{\min}\le \rho(t, \cdot) \le \rho_{\max},\qquad\|\nabla \rho(t,\cdot)\|_{L^\infty(\S^{d-1})}\le L
\]
for all $t\geq0$, 
where $\rho_{\min}$ and $\rho_{\max}$ are independent of $d$ and $\upbeta$.
\end{assumption}

At initial time, this can be seen as a perturbative regime around the uniform distribution.
It is plausible that this assumption can be imposed solely on the initial density $\rho(0,\cdot)$ and then propagated over time by using the setting of stochastic flows of Kunita \cite{kunita1990stochastic}. Proving this would, however, require separate regularity analysis of the induced flow which is beyond the scope of the present paper.

\begin{theorem} \label{thm:large_beta_meta}
Under Assumption \ref{ass:low-temperature}, 
consider $x^{(1)}$ and $x^{(2)}$ two solutions to \eqref{eq:non_linear_SDE_cylindrical} with 
initial data 
$x^{(1)}(0)$ and $x^{(2)}(0)$ drawn independently from $\mu(0)$. 
Then there exists a universal constant $C>0$ such that for every $T>0$ and every
$\delta\in(0,1)$ and $t\ge 0$, with probability at least $1-\delta$,
\begin{align*} 
    \sup_{t\in[0,T]}\left|\left\<x^{(1)}(t),x^{(2)}(t)\right\>-\left\<x^{(1)}(0),x^{(2)}(0)\right\>\right|\leq \frac{T}{d}+C\sqrt{\frac{T}{d} \log\left(\frac{1}{\delta}\right)}+r(\upbeta,d),
\end{align*}
where $|r(\upbeta,d)|=O(d^{-\frac32}+\upbeta^{-\frac12}).$ 

Furthermore, if $\upbeta=\upbeta(d)$ is such that $d \upbeta^{-\frac12}=o(1)$, then
\begin{equation*}
    \left(\left\langle x^{(1)}(td), x^{(2)}(td)\right\rangle\right)_{t\in[0,T]} \,\,\overset{\mathsf{Law}}{\underset{d\rightarrow\infty}{\Longrightarrow}}\,\, (R_\infty(t))_{t\in[0,T]}
\end{equation*}
where $R_\infty(t)$ is the unique solution to the SDE
\begin{equation} \label{eq: sde.logistic}
\begin{cases}
    \de u(t)=-u(t)\left(1-u(t)^{2}\right)\de t+\sqrt{2}\left(1-u(t)^{2}\right)\de B(t) \\
    u(0)=0,
\end{cases}
\end{equation}
and $B(t)$ is the one-dimensional Brownian motion.
\end{theorem}

The proof can be found in {Section  \ref{proof: mean.field}}. The result can be readily transferred to the particle system by a mean-field limit result such as Theorem \ref{thm:PoC_wellposedness}, up to standard error terms.

Note that Theorem \ref{thm:large_beta_meta} shows that, on the original time scale, particles move very little up to times
\(T\sim d\), so the system exhibits a slow  transient.
This is reminiscent of metastability phenomena studied for self-attention dynamics
\cite{geshkovski2024dynamic,bruno2024emergence,bruno2025multiscale}, where the empirical measure can spend long times
near structured configurations before undergoing a qualitative change. However, starting from a near-uniform initialization, the dynamics does not
display a rapid deterministic escape from the vicinity of the uniform distribution, in contrast with \cite{bruno2024emergence}. 

Theorem \ref{thm:large_beta_meta} is also analogous in spirit to the Bronsard--Kohn energy method for slow motion, which gives lower bounds on the time needed to leave a neighborhood of a slow manifold in gradient flows  \cite{bronsard1990slowness, otto2007slow}. The analogy should be understood at the level of the conclusion rather than the proof: here we do not use an underlying energy--dissipation structure, and the lower bound on the escape time follows directly from Theorem~\ref{thm:large_beta_meta}. 

The long-time behavior of the limiting logistic SDE is treated separately in Section~\ref{sec: sde.logistic}.

\subsection{Notation} \label{note:probabilistic_setting}

Whenever convenient we use $A\lesssim_S B$ to imply $A\leq CB$ for $C>0$ depending only on $S$.
We take $(\cF_t)$ to be the usual augmentation of the filtration generated by 
$\{x_1,\ldots,x_n\}\cup\{\beta^k(s): 0\le s\le t, k\ge1\}$. We define the common-noise filtration by $$
\cF_t^{\mathsf W,0}\coloneqq\sigma(\beta^k(s): 0\le s\le t,\ k\ge1),
$$ and we designate by $\cF_t^{\mathsf W}$ the usual augmentation of $\cF_t^{\mathsf W,0}$.

\subsection*{Acknowledgments}

B.G. was supported by a Sorbonne Emergences grant, and a gift from Google. P.R. was supported by
NSF grants DMS-2022448, CCF-2106377, and a gift from Apple.

\part{The homogenized model}

\section{Preliminaries} \label{sec:main_result}

Fix $n\ge 2$. 
For $X=(x_1,\dots,x_n)\in(\S^{d-1})^n$ write $\mu_X$ for the empirical measure.
Define 
$$
b(X)_i\coloneqq\proj_{x_i} b_{{\uprho_*}}[\mu_X](x_i)
$$ 
and
\begin{equation}\label{eq:G_def}
G(X,\bm\theta)_i\coloneqq\proj_{x_i}\,\xi_{\bm{\theta}}[\mu_X](x_i)
\end{equation}
for $i\in[n]$. 

\subsection{Defining $\alpha$} \label{sec: defining.alpha}

Looking at the increments \eqref{eq:update_tokens} and applying Taylor-Lagrange we have
\begin{equation*}
    x_i^{\ell+1} - x_i^{\ell} = \eta \proj_{x_i^\ell} b_{\uprho_*}[\mu_{X^\ell}](x_i^\ell) +  \frac{\eta}{H}\sum_{h=1}^H\proj_{x_i^\ell} \xi_{\bm{\theta}^\ell_h}[\mu_{X^\ell}](x_i^{\ell}) + O(\eta^2).
\end{equation*}
To construct a continuous-time object, we shall follow the stochastic modified equations framework of \cite{li2017sme, li2019sme}, readapted in \cite{gess2024rsgd, gess2024stochastic, gess2025conservative}. 
To this end, we want the fluctuations to be multiplied by $\sqrt\eta$. We compute the variance of the second term
\begin{align} \label{eq: defining.alpha}
    \E \left\|\frac{1}{\sqrt{H}}\sum_{h=1}^H\proj_{x_i^\ell} \xi_{\bm{\theta}^\ell_h}[\mu_{X^\ell}](x_i^{\ell})\right\|^2 &= \frac{1}{H}\sum_{h=1}^H\E\left\|\proj_{x_i^\ell} \xi_{\bm{\theta}^\ell_h}[\mu_{X^\ell}](x_i^{\ell})\right\|^2\nonumber\\
    &\leq \max_{X\in(\S^{d-1})^n}\max_{j\in[n]} \E\left\|\proj_{x_j}\xi_{\bm\theta}[\mu_X](x_j)\right\|^2 .
\end{align}
Now, define $\upsigma^{2}$ as the upper bound in \eqref{eq: defining.alpha}. 
In this way, we can naturally set 
\begin{equation}\label{eq:Alpha_Sec2}
    \alpha\coloneqq \frac{\eta\upsigma^{2}}{H}.
\end{equation}
Assumption \ref{ass:high_order_short} crucially implies the following. 
 
\begin{lemma}\label{lem:Kernel_regularity}
Under Assumption \ref{ass:high_order_short}, there exist $C_1,C_2>0$ independent of $\alpha, \eta, L$ such that
\begin{enumerate}
\item $b\in C^4((\S^{d-1})^n;\mathsf{T}(\S^{d-1})^n)$.
\smallskip
\item We have
\begin{equation*}
    \upsigma \coloneqq \max_{X\in(\S^{d-1})^n}\|G(X,\cdot)\|_{L^2({\uprho_*})}<\infty
\end{equation*}
and 
\begin{equation*}
     \underset{X \in (\S^{d-1})^{n}}{\max}\|G(X,\cdot)\|_{L^3({\uprho_*})}\leq C_1\upsigma .
\end{equation*}
\smallskip 
\item For ${\uprho_*}$-a.e.\ $\bm{\theta}$, $G(\cdot,\bm{\theta})\in C^5((\S^{d-1})^n;\mathsf{T}(\S^{d-1})^n)$ and
\begin{equation*}
    \left\|\Big\|G(\cdot, \bm\theta)\Big\|_{C^5} \right\|_{L^2({\uprho_*})} + \left\|\Big\|\nabla_{G(\cdot,\bm{\theta})}G(\cdot,\bm{\theta})\Big\|_{C^5}\right\|_{L^2({\uprho_*})} \leq C_2 \upsigma .
\end{equation*}
\end{enumerate}
Moreover we can choose $C_1=O(1+\frac{\sigma_{\bA}\|\E\bV\|_{\mathrm{op}}\upbeta}{\sigma_{\bV}})$ and $C_2=O(1+\upbeta^{10}\sigma_{\bA}^{10}d^5)$.
\end{lemma}

The proof is deferred to {Section \ref{sec:APP_Diffusive}}. 

\subsection{The SDE}

Set 
\begin{equation}\label{eq:G_alpha_def}
G^\upsigma(X) \coloneqq \upsigma^{-1}G(X).
\end{equation}
Given an initial configuration $X(0)\in (\S^{d-1})^n$, the limiting continuous-time object is the $(\S^{d-1})^n$-valued It\^o SDE
\begin{equation}\label{eq:SDE_ito_clean}
\de X(t)=b(X(t)) \de t - \frac{\eta}{2} \nabla_{b(X(t))}b(X(t)) \de t +  \sqrt{\alpha}\int_\Uptheta G^{\upsigma}(X(t),\bm{\theta}) \, \mathsf{W}(\de\bm{\theta},\de t).
\end{equation}
As before, $(\mathsf{W}(t))_{t\geq 0}$ is a cylindrical Wiener process on $L^{2}({\uprho_*})$.

Note that \eqref{eq:SDE_ito_clean} differs from \eqref{eq: first.sde} only in the corrector drift $-\frac{\eta}{2} \nabla_{b(X)}b(X)$ which is uniformly of order $O(\eta)$, and can thus be absorbed in the final error term, as we ultimately choose to do in the statement of Theorem \ref{thm:weak_error_clean}.

By virtue of Lemma \ref{lem:Kernel_regularity}, the map $X\mapsto b(X)-\frac{\eta}{2} \nabla_{b(X)}b(X)$ is a $C^{1}$ vector field on $(\S^{d-1})^n$, and for $\uprho_*$-a.e. $\theta$,  
$X\mapsto G(X,\theta)$ is $C^{2}$.
According to Proposition \ref{prop:existence_SDE_sphere}, \eqref{eq:SDE_ito_clean} admits a unique global strong solution
defined for all $t\ge0$. Note that this assumption is stronger than what is strictly needed to construct the SDE; the extra regularity is used in the proof of Theorem \ref{thm:weak_error_clean}. This choice is made so that the It\^o generator matches the discrete generator.

\begin{remark}
The generator of the discrete process (we omit the drift for the sake of simplicity), denoted $\mathsf{L}^{\eta}$, is given by 
\[
(\mathsf{L}^{\eta}\varphi)(x)=\frac{\alpha}{2}\int_{\Uptheta}\Hess \varphi(x)[G^{\upsigma}(x,\bm{\theta}),G^{\upsigma}(x,\bm{\theta})]\uprho_*(\de \bm{\theta}).
\]
Now, consider $X(t)$ a solution to the Itô SDE
\[
\de X(t)=\int_{\Uptheta}G^{\upsigma}(X(t),\bm{\theta})\mathsf{W}(\de \bm{\theta},\de t).
\]
By definition of the Itô SDE, the generator $\mathsf{L}$ of this process is given by 
\[
(\mathsf{L}\varphi)(x)=\frac{\alpha}{2}\int_{\Uptheta}\Hess\varphi(x)[G^{\upsigma}(x,\bm{\theta}),G^{\upsigma}(x,\bm{\theta})]\uprho_*(\de \bm{\theta}).
\]
Notice that the two generators match. Since the generator characterizes the law of the process, the Itô SDE is a legitimate candidate for being the correct homogenized model.
\end{remark}

\section{Ballistic scaling} \label{sec:deterministic}

For the sake of completeness, we now list several immediate corollaries of the proof of Theorem \ref{thm:weak_error_clean}, which are displayed in Figure \ref{fig: phase-diags}.

\begin{corollary}[Ballistic regime]\label{cor:ode1}
Under Assumption \ref{ass:high_order_short}, in the subcritical scaling $\alpha \eta L = o(1)$ and $\eta^2 L = o(1)$, the time-interpolated discrete chain $X^\eta(t)=X^{\lfloor t/\eta\rfloor}$ defined in \eqref{eq:update_tokens} with initial configuration $X^0$ is weakly approximated by the solution $X(t)$ of  
\begin{equation} \label{eq: deterministic}
    \frac{\de X(t)}{\de t} = b(X(t)),
\end{equation}
with $X(0) = X^0$ up to time $t_L=\eta L$. 
Namely for all $\varphi\in C^4((\S^{d-1})^n)$,
\[
\sup_{t\in[0, t_L]}\left|\E\varphi(X^\eta(t))-\E\varphi(X(t))\right|
\le Ce^{C t_L} \eta (t_L+1)\,\max(\eta,\alpha).
\]
\end{corollary}

This corresponds to the region "ODE~I" in  Figure \ref{fig: phase-diags}. In the case $\eta = L^{-\gamma}$ for some $\gamma\in[0,1]$, the scaling in   Corollary \ref{cor:ode1} becomes
\begin{equation*}
\alpha L^{1-\gamma}=o(1),\qquad L^{1-2\gamma}=o(1),
\end{equation*}
so in the large-depth regime, the second inequality is equivalent to $\gamma>1/2$.

\subsection{Modified regime}\label{subsec:modified_regime}

Corollary \ref{cor:ode1} relies on the condition
$\eta^2L=o(1)$, which ensures that all $O(\eta)$ corrections to the effective drift
remain negligible over the macroscopic time horizon $T=\eta L$.
When $\eta^2L$ is not small---for instance, when $\eta=L^{-\gamma}$ with
$\gamma\in(1/3,1/2]$---, these $O(\eta)$ effects can accumulate over depth \eqref{eq: deterministic} may no longer provide an accurate deterministic effective
model.

Theorem \ref{thm:weak_error_clean} is obtained by matching
the discrete one-step generator with the It\^o generator of a continuous-time model.
When the noise vanishes at macroscopic scale,
this generator matching yields a corrected drift on $(\S^{d-1})^n$ of the form
\(
b(X)-\frac{\eta}{2}\nabla_{b(X)}b(X).
\)
This corresponds to the ``ODE~II'' region in Figure \ref{fig: phase-diags}.

\begin{corollary}[Modified regime]\label{cor:ode2}
Under Assumption \ref{ass:high_order_short}, in the refined deterministic scaling
$\alpha\eta L=o(1)$ and $\eta^3L=o(1)$, the time-interpolated discrete chain $X^\eta(t)=X^{\lfloor t/\eta\rfloor}$ defined in \eqref{eq:update_tokens} with initial configuration $X^0$ is weakly approximated by the solution $X(t)$ of the corrected ODE
\begin{equation} \label{eq: deterministic.modified}
\frac{\de X(t)}{\de t}= b(X(t))-\frac{\eta}{2}\nabla_{b(X(t))}b(X(t)),
\end{equation}
with $X(0)=X^0$ up to time $t_L = \eta L$. More precisely, for all
$\varphi\in C^4((\S^{d-1})^n)$,
\[
\sup_{t\in[0, t_L]}\left|\E\varphi(X^\eta(t))-\E\varphi(X(t))\right|
\le Ce^{C t_L} \eta (t_L+1)\,\max(\eta,\alpha).
\]
\end{corollary}

\section{Diffusive scaling} \label{sec:diffusive}

Now suppose $\eta\alpha L=\Theta(1)$. We apply Theorem \ref{thm:weak_error_clean} to obtain 

\begin{corollary}[Diffusive regime]\label{cor:SDE}
In the diffusive scaling $\alpha\eta L=O(1)$ and $\eta^3 L=o(1)$, the time-interpolated discrete chain $X^\eta(t)=X^{\lfloor t/\eta\rfloor}$ defined in \eqref{eq:update_tokens} with initial configuration $X^0$ is weakly approximated by the solution $X(t)$ of \eqref{eq:SDE_ito_clean}
with $X(0)=X^0$ up to time $t_L = \eta L$. More precisely, for all
$\varphi\in C^4((\S^{d-1})^n)$,
\[
\sup_{t\in[0, t_L]}\left|\E\varphi(X^\eta(t))-\E\varphi(X(t))\right|
\le Ce^{Ct_L}\eta( t_L+1)\max(\eta,\alpha).
\]
\end{corollary}

\begin{remark} \label{rem:shared_noise_rewrite}

A key point of \eqref{eq:SDE_ito_clean} is that all particles are driven by the
same cylindrical Wiener process $\mathsf W(t)$ on $L^2(\uprho_*)$.
This is not a modeling choice but a consequence of the requirement that the diffusion limit
reproduce the cross-covariances of the discrete layer fluctuations.
Indeed, writing the martingale part as
\[
\de M_i(t)=\sqrt{\alpha}\,\proj_{x_i(t)}\int_{\Uptheta}\xi_{\bm\theta}\left[\mu_{X(t)}\right](x_i(t)) \mathsf{W}(\de\bm\theta,\de t),
\]
we obtain the cross-variation identity
\begin{equation}\label{eq:cross_variation_kernel}
\de\langle M_i,M_j\rangle_t
=\alpha\;\proj_{x_i(t)}\mathsf{K}\left[\mu_{X(t)}\right](x_i(t),x_j(t))\proj_{x_j(t)}\de t,
\end{equation}
for $i,j\in[n]$, where 
\begin{equation*}\label{eq:covariance_kernel}
\mathsf{K}[\mu](x,y)
\coloneqq \E \left[\xi_{\bm{\theta}}[\mu](x)\,\xi_{\bm{\theta}}[\mu](y)^{\sT}\right]\in\reals^{d\times d}
\end{equation*}
is the covariance kernel.

The kernel $\mathsf{K}[\mu](x,y)$ should be read as a matrix-valued correlation kernel for the
martingale increments of different particles.
Two extreme cases help build intuition.

\begin{description}
    \item[(i) Purely idiosyncratic noise] If $\mathsf{K}[\mu](x,y)=0$ for all $x\neq y$, then \eqref{eq:cross_variation_kernel} gives $\de\langle M_i,M_j\rangle_t\equiv0$ for $i\neq j$. 
    In that case, one may equivalently realize the same law by driving each particle with its own independent Brownian motion. This is the setting of
classical McKean--Vlasov diffusions with idiosyncratic noise, as in recent diffusive variants of Transformers dynamics \cite{balasubramanian2025structure,shalova2024solutions,gerber2025formation,leimkuhler2025cluster}.
\smallskip 

\item[(ii) Perfectly common noise]
If $\mathsf{K}[\mu](x,y)\equiv \Sigma$ is constant in $(x,y)$, then there exists
$\sigma_0\in\R^{d\times d}$ with $\Sigma=\sigma_0\sigma_0^\top$, and the cross-variations satisfy
\[
\de \<M_i, M_j\>_{t}=\sqrt{\alpha}\,\proj_{x_i(t)}\sigma_0\sigma_0^{\sT}\proj_{x_j(t)}\de t.
\]
\end{description}
\end{remark}

\subsection{Mean-field limit}

Due to the considerations of Theorem \ref{thm:large_beta_meta}, it is necessary to give a sense to the mean-field limit of the homogenized model. We shall focus on the driftless case \eqref{eq:Diffusive_gaussian_case} (with $\alpha=1$) since that is the one of interest in Theorem \ref{thm:large_beta_meta}, but the drift can easily be included in all that follows.

Let $(\Omega, \mathscr{F}, (\mathscr{F}_t)_{t\geq0}, \mathbb{P})$ be a filtered, complete probability space with right-continuous filtration $\mathscr{F}_t$.
Classical mean-field limits of weakly interacting diffusions with idiosyncratic noise lead to McKean--Vlasov SDEs whose law solves a deterministic nonlinear Fokker--Planck PDE. In contrast, when the particles are additionally driven by a common noise, the relevant limit object the conditional law
\begin{equation*}\label{eq:conditionnal_law}
    \mu(t) \coloneqq \mathsf{Law}\left(x(t) \mid \mathscr F_t^{\mathsf W}\right).
\end{equation*}
The process $(\mu(t))_{t\ge 0}$ is then random and evolves according to a nonlinear Fokker--Planck SPDE. 
  
Set $\mu_n(t)\coloneqq \mu_{X(t)}$.
For any $\mu\in\mathcal{P}(\S^{d-1})$, let $G_\mu(x, \bm \theta)$ be defined through \eqref{eq:G_def}.
The Itô formula gives 
\begin{equation*}
    \de\langle \varphi,\mu_{n}(t)\rangle=\<\mathsf{L}_{\mu_{n}(t)}\varphi,\mu_{n}(t)\>\de t+\int_{\Uptheta}\<\nabla \varphi \cdot G_{\mu_{n}(t)}(\cdot,\bm{\theta}),\mu_{n}(t)\> \mathsf{W}(\de \bm{\theta},\de t),
\end{equation*}
for any smooth test function $\varphi$, where $\<\cdot,\cdot\>$ denotes the duality bracket between $C^0(\S^{d-1})$ and positive Borel measures on $\S^{d-1}$, $\cdot$ is the Euclidean dot product, and $\mathsf{L}_{\mu_n(t)}$ is the generator of the process $\mu_n(t)$.
Thus, formally, $\mu_n(t)$ solves the nonlinear Fokker-Planck SPDE
\begin{equation} \label{eq: the.spde}
\partial_t \mu(t)=\mathsf{L}_{\mu(t)}^{*}\mu(t) +\int_{\Uptheta}\dive_x\left(G_{\mu(t)}(\cdot,\bm{\theta})\; \mu(t)\right)\,\mathsf{W}(\de t,\de \bm{\theta}),
\end{equation}
where $\mathsf{L}_{\mu}^{*}$ is the adjoint of $\mathsf{L}_\mu$; it reads 
\begin{equation*}
\mathsf{L}_{\mu}^{*}\mu=\frac{1}{2}\dive_x \dive_x\left(\int_{\Uptheta}G_{\mu}(\cdot,\bm{\theta})G_{\mu}(\cdot,\bm{\theta})^{\sT}\uprho_*(\de \bm{\theta}) \mu\right).
\end{equation*}
(Note that $\dive_x$ denotes the adjoint of the spherical gradient $\nabla$.)

\begin{definition} \label{def:weak_spde}
A $\cF_t^{\mathsf W}$-adapted process $(\mu(t))_{t\in[0,T]}$ with values in $\mathcal \cP(\S^{d-1})$
is called a weak solution to \eqref{eq: the.spde} if for all $\varphi\in C^\infty(\mathbb S^{d-1})$,
\begin{align*}
\langle \mu(t),\varphi\rangle
=
\langle \mu(0),\varphi\rangle
&+\int_0^t \langle \mathsf{L}_{\mu(s)}\varphi,\mu(s)\rangle\,\de s \nonumber\\
&+\int_0^t\int_{\Uptheta} \langle \nabla\varphi\cdot G_{\mu(s)}(\cdot,\bm{\theta}),\mu(s) \rangle\,\mathsf{W}(\de s,\de \bm{\theta}),
\end{align*}
where $\mathsf{L}_{\mu}$ is the generator
\[
\mathsf{L}_{\mu}\varphi(x)=\frac{1}{2}\int_{\Uptheta}\Hess \varphi(x)[G_{\mu}(x,\bm{\theta}),G_{\mu}(x,\bm{\theta})]\uprho_*(\de \bm{\theta})
\]
for all $\varphi \in C^{\infty}(\S^{d-1})$.
\end{definition}

Well-posedness and weak/strong notions of solutions for McKean--Vlasov SDEs with common noise have been developed in 
\cite{hammersley2021weak,kumar2022well}, and the corresponding SPDE formulations go back to 
\cite{kurtz1999particle} and more recently in \cite{coghi2019stochastic}.

The trajectorial characterization is the following.

\begin{definition} \label{def:nonlinear_SDE_common}
Let $\mu_0\in\mathcal{P}(\S^{d-1})$, let $x_0$ be an $\S^{d-1}$-valued random variable independent of $\mathsf W$ with $\mathsf{Law}(x_0) = \mu_0$. 
A pair $(x,\mu)$ is a solution
to the McKean-Vlasov SDE associated to \eqref{eq: the.spde} if $(x(t))_{t\geq0}$ is $(\cF_t)_{t\geq0}$-adapted, 
$\mu(t)$ is $(\cF_t^{\mathsf W})_{t\geq0}$-adapted, with
\[
x(0) = x_0 \quad \text{ a.s.}, \qquad \mu(0)=\mathsf{Law}\left(x_0\mid \mathscr{F}^{\mathsf{W}}_0\right) = \mu_0,
\]
as well as
\[
\mu(t)=\mathsf{Law}\left(x(t)\mid \cF_t^{\mathsf W}\right)\quad\text{a.s.},
\]
and
\begin{equation}\label{eq:non_linear_SDE_common}
\de x(t)=\int_{\Uptheta}G_{\mu(t)}(x(t),\bm{\theta})\mathsf{W}(\de \bm{\theta},\de t).
\end{equation}
\end{definition}

Inspired by the literature on mean-field games with common noise \cite{lacker2022superposition,carmona2018probabilistic}, we can show the following.

\begin{theorem} \label{thm:PoC_wellposedness}

Under Assumption \ref{ass:high_order_short}, for every 
$\mu_0\in\mathscr{P}(\S^{d-1})$,
there exists a unique solution $(x,\mu)$ to the McKean-Vlasov SDE \eqref{eq:non_linear_SDE_common} in the sense of Definition
\ref{def:nonlinear_SDE_common}. Moreover, $\mu$ is the unique weak solution to \eqref{eq: the.spde} in the sense of Definition \ref{def:weak_spde}. 
\end{theorem}

The associated mean-field limit is given in Proposition \ref{prop: poc}, and the proof can be found in Section \ref{sec: thm.5.proof}.

\part{Proofs}

\section{The homogenized model}

\subsection{Proof of Theorem \ref{thm:weak_error_clean}} \label{sec: proof.main}

\begin{proof}[Proof of Theorem \ref{thm:weak_error_clean}] The proof follows the ideas of \cite[Theorem 2]{gess2024rsgd}. 
Let $X^{1}$ be a one-step iterate of the Markov chain. 
We also specify the initial condition $x\in(\S^{d-1})^{n}$ by writing $X(t;x)$ and $X^{1}(x)$.

\begin{lemma}\label{lem:stability_generator}
Let $\varphi \in C^{4}((\S^{d-1})^{n}).$ There exists a constant $C>0$ such that 
\begin{equation*}
    \max_{x\in(\S^{d-1})^{n}}\left|\E \varphi(X(\eta,x))-\E \varphi(X^{1}(x))\right|\leq C\eta^2\max\left(\eta,\alpha \right).
\end{equation*}
\end{lemma}

\begin{proof}[Proof of Lemma \ref{lem:stability_generator}] We split the proof in two steps.

\subsubsection*{Discrete process}

Recall that 
\begin{equation*}
    X^{1}(x)_i=\mathsf{N}\left(x_i + \eta  b_{\uprho_*}[\mu_{x}](x_i) +\frac{\eta}{H}\sum_{h=1}^{H}\xi_{\bm{\theta}_h}[\mu_x](x_i)\right).
\end{equation*}
Using Taylor-Lagrange, we have 
\begin{align*}
    \varphi(X^{1}(x))&=\varphi(x)+\eta\left\<b(x)+\frac{1}{H}\sum_{h=1}^{H}G(x,\bm\theta_h),\nabla \varphi(x)\right\>\\
    &+\frac{\eta^{2}}{2}\Hess \varphi(x):\left(b(x)+\frac{1}{H}\sum_{h=1}^{H}G(x,\bm\theta_h)\right)^{\otimes 2}+ R(\eta,x),
\end{align*}
where $|R(\eta,x)|\lesssim \eta^{3}(\|b(x)+\frac{1}{H}\sum_{h=1}^{H}G(x,{\bm{\theta}_h})\|^{3})\|\varphi\|_{C^{3}}$. Taking the expectation yields
\begin{align*}
    \E\varphi(X^{1}(x))&=\varphi(x)+\eta\left\<b(x),\nabla \varphi(x)\right\>\\
    &+\frac{\eta^{2}}{2}\Hess \varphi(x):b(x)b(x)^{\sT}\\
    &+\frac{1}{2}\left(\frac{\eta}{H}\right)^{2} \Hess \varphi(x):\E\left[\left(\sum_{h=1}^{H}G(x,{\bm{\theta}_h}) \right)^{\otimes 2}\right]+ R(\eta,x).
\end{align*}
Using independence and the fact that $G(x,\bm{\theta})$ is centered, we have 
\begin{align*}
    \E[\varphi(X^{1}(x))-\varphi(x)]&=\eta\left\<b(x),\nabla \varphi(x)\right\>\\
    &+\frac{\eta^{2}}{2}\Hess \varphi(x)\left[b(x),b(x)\right]\\
    &+\frac{\eta}{2}\;\frac{\eta \upsigma^2 }{H}\; \Hess \,\varphi(x):\E\left[\left(\frac{G(x,\bm{\theta})}{\upsigma}\right)\left(\frac{G(x,\bm{\theta})}{\upsigma}\right)^{\sT}\right]\\
    &+R(\eta,x).
\end{align*}
From \eqref{eq: defining.alpha} and \eqref{eq:G_alpha_def}, we obtain 
\begin{align*}
    \E[\varphi(X^{1}(x))-\varphi(x)]&=\eta\left\<b(x),\nabla \varphi(x)\right\>+\frac{\eta^{2}}{2}\Hess \varphi(x)\left[b(x), b(x)\right]\\
    &+\frac{\eta}{2}\;\alpha\; \Hess \,\varphi(x):\E\left[G^{\upsigma}(x,\bm{\theta})\; G^{\upsigma}(x,\bm{\theta})^{\sT}\right]+R(\eta,x).
\end{align*}
We bound the remainder term. 
Using the Rosenthal inequality (\cite[Theorem 15.11]{boucheron2013concentration}), we have 
\begin{equation*}
\E\left\|\sum_{h=1}^{H}G(x,{\bm{\theta}_h})\right\|^{3}\lesssim\left(H \E\|G(x,{\bm{\theta}})\|^{2}\right)^{\frac32} + H\E\|G(x,{\bm{\theta}})\|^{3}.
\end{equation*}
Using Lemma \ref{lem:Kernel_regularity}, we obtain 
\begin{align*}
\E\left\|\frac{1}{H}\sum_{h=1}^{H} G(x,\bm{\theta}_h)\right\|^{3}&\lesssim \left(\frac{\E\| G(x,\bm{\theta})\|^{2}}{H}\right)^{\frac32} + \frac{\E\| G(x,\bm{\theta})\|^{3}}{H^{2}}\\
&\lesssim \left(\frac{\upsigma^{2}}{H}\right)^{\frac32}+\frac{\upsigma^{3}}{H^{2}}\lesssim \left(\frac{\upsigma^{2}}{H}\right)^{\frac32}.
\end{align*}
Combining both estimates, we find
\begin{equation*}\label{eq:uniform_bound_rest}
     R(\eta,x)\lesssim \eta^{3}\left(\left\|b(x)\right\|^{3}+\left(\frac{\upsigma^{2}}{H}\right)^{\frac32}\right)\lesssim \eta^{3}\left\|b(x)\right\|^{3}+ \eta^{\frac32}\left(\frac{\eta \upsigma^{2}}{H}\right)^{\frac32}.
\end{equation*}

\subsubsection*{Continuous process.}
We denote $(X(\eta;x))_{t\geq 0}$ the solution of the SDE starting from $x$. By definition, we have for $t=\eta$, and $V_0(X) = b(X)-\frac{\eta}{2}\nabla_{b(X)}b(X)$ with $V_0 \varphi = \langle \nabla \varphi, V_0\rangle$, 
\begin{align*}
    &\varphi(X(\eta;x))-\varphi(x)=\int_{0}^{\eta}V_0\varphi(X(s;x))\de s\\
    &+\frac{\alpha}{2}\iint_{[0,\eta]\times \Uptheta}\Hess\varphi(X(s;x))[G^\upsigma(X(s;x), \bm\theta),G^\upsigma(X(s;x), \bm\theta)]{\uprho_*}(\de \bm{\theta})\de s\\
    &+\sqrt{\alpha}\iint_{[0,\eta]\times \Uptheta}G^{\upsigma}\varphi(X(s;x),\bm{\theta})\mathsf{W}(\de s,\de \bm{\theta}).
\end{align*}
The quadratic variation of the second term is given by 
\begin{equation*}
    \left\<\int_{0}^{\cdot}\int_{\Uptheta}G^{\upsigma}\varphi(X(s;x),\bm{\theta})\mathsf{W}(\de s,\de \bm{\theta}) \right\>_t=\int_{0}^{t}\|G^{\upsigma}\varphi(X(s;x),\cdot)\|_{L^2({\uprho_*})}^{2}\; \de s.
\end{equation*}
By Fubini, 
\[ 
\E \int_{0}^{t}\int_{\Uptheta} G^{\upsigma}V_0\varphi(X(s;x),\bm{\theta})\mathsf{W}(\de s,\de \bm{\theta})=0.
\]
Taking expectations we have 
\begin{align*}
    &\E\varphi(X(\eta;x))-\varphi(x)=\int_{0}^{\eta}\E[V_0\varphi(X(s,x))]\de s\\
    &+\frac{\alpha}{2}\int_{0}^{\eta}\int_{\Uptheta}\E\Big[\Hess\varphi(X(s;x))[G^\upsigma(X(s;x), \bm\theta),G^\upsigma(X(s;x), \bm\theta)]\Big]{\uprho_*}(\de \bm{\theta})\de s.
\end{align*}
We now bound the right-hand side terms by using the Itô formula for $V_0 \varphi$: 
\begin{align*}
&V_0\varphi(X(t;x))=V_0\varphi(x)\\
&+\frac{\alpha}{2}\int_{0}^{t}\int_{\Uptheta}\Hess V_0\varphi(X(s;x))[G^\upsigma(X(s;x), \bm\theta),G^\upsigma(X(s;x), \bm\theta)]{\uprho_*}(\de \bm{\theta})\de s\\
    &+\int_{0}^{t}V_0^{2}\varphi(X(s;x))\de s\\
&+\int_{0}^{t}\int_{\Uptheta}\sqrt{\alpha}\;G^{\upsigma}(V_0\varphi)(X(s;x),\bm{\theta})\mathsf{W}(\de s,\de \bm{\theta}).
\end{align*}
Moreover
\begin{align*}
&\E\left| \int_{0}^{\eta}\int_{0}^{s}\int_{\Uptheta}\alpha\; \Hess V_0\varphi(X(u;x))[G^\upsigma(X(u;x), \bm\theta),G^\upsigma(X(u;x), \bm\theta)]{\uprho_*}(\de \bm{\theta})\de u\de s\right|\\
&\leq \alpha \eta^{2}\max_{x \in (\S^{d-1})^{n}}\E\Big|\Hess V_0\varphi(X(s;x))[G^\upsigma(X(s;x), \bm\theta),G^\upsigma(X(s;x), \bm\theta)]\Big|.
\end{align*}
Using Itô's formula for $V_0^{2}\varphi(X(s;x))$, 
\begin{align*}
V_0^{2}\varphi(X(s;x))&=V_0^{2}\varphi(x)+\sqrt{\alpha}\int_{0}^{s}\int_{\Uptheta}G^{\upsigma}V_0^{2}\varphi(X(u;x),\bm{\theta})\mathsf{W}(\de u,\de \bm{\theta})\\
&+\int_{0}^{s}V_0^{3}\varphi(X(u;x))\de u\\
&+\frac{\alpha}{2}\int_{0}^{s}\int_{\Uptheta} \Hess V_0^2\varphi(X(u;x))[G^\upsigma(X(u;x), \bm\theta),G^\upsigma(X(u;x), \bm\theta)]{\uprho_*}(\de \bm{\theta})\de u.
\end{align*}
All in all, 
\begin{align*}
&\Bigg|\E[\varphi(X(\eta;x))]-\varphi(x)-V_0\varphi(x)-\frac{\eta^{2}}{2}V_0^{2}\varphi(x)\\
&-\frac{\alpha}{2}\int_{\Uptheta}\Hess\varphi(X(s;x))[G^\upsigma(X(s;x), \bm\theta),G^\upsigma(X(s;x), \bm\theta)]{\uprho_*}(\de \bm{\theta})\Bigg|\\
&\leq \alpha \eta^{2}\max_{X \in (\S^{d-1})^{n}}\E\Big|\Hess V_0\varphi(X(\eta;x))[G^\upsigma(X(\eta;x), \bm\theta),G^\upsigma(X(\eta;x),\bm\theta)]\Big|\\
&+\left|\int_{0}^{\eta}\int_{0}^{t}\int_{0}^{s}V_0^{3}\varphi(X(u;x))\de u \de s\de t\right|\\
&+\frac{\alpha}{2}\left|\int_{0}^{\eta}\int_{0}^{t}\int_{0}^{s}\int_{\Uptheta} \Hess V_0^2\varphi(X(u;x))[G^\upsigma(X(u;x), \bm\theta),G^\upsigma(X(u;x), \bm\theta)]{\uprho_*}(\de \bm{\theta})\de u \de s\de t\right|.
\end{align*}
and using the identity
\begin{equation*}
    V_0^2\varphi(x)=\< \Hess \varphi(x) V_0(x),V_0(x)\> +\<\nabla_{V_0}V_0(x),\nabla \varphi(x)\>
\end{equation*} 
(see Section \ref{sec:Toolkit}) we have 
\begin{align*}
    &\max_{X \in (\S^{d-1})^{n}}\Big| \E\varphi(X(\eta;x))-\eta V_0\varphi(x)-\frac{\eta^{2}}{2} \<\Hess\varphi(x)V_0(x),V_0(x)\>\\
    &\hspace{3cm}-\frac{\alpha \eta }{2}\;\E\left\<\Hess \varphi(x)G^{\upsigma}(x,\bm{\theta}),G^{\upsigma}(x,\bm{\theta})\right\>\Big|\leq C \eta^{2}\alpha.
\end{align*}
All in all,
\begin{equation*}
\max_{x\in (\S^{d-1})^{n}}\left| \E\varphi(X(\eta;x))-\E\varphi(X^{1}(x))\right|\leq  C\eta^2\max\{\alpha,\eta\}.\qedhere
\end{equation*}
\end{proof}

By virtue of Assumption \ref{ass:high_order_short}, Lemma \ref{lem:Kernel_regularity} ensures that $V_{0} \in \mathfrak{X}^{4}((\S^{d-1})^{n}))$, $G^{\upsigma}\in \mathfrak{X}^{5}((\S^{d-1})^{n})$, and 
\begin{equation*}
\|G^{\upsigma}\|_{\mathfrak{X}^{4}((\S^{d-1})^{n})}\vee \|V_0\|_{\mathfrak{X}^{5}((\S^{d-1})^{n})}< \infty,
\end{equation*}
where $\mathfrak{X}^{k}((\S^{d-1})^{n}))$ is defined in Definition \ref{def: derivatives}.
Fix $T>0$.
Consider 
$$
\psi^{\eta}_t(x)\coloneqq\E \varphi(X(t;x)).
$$ 
There exists a constant $C>0$ depending only on $\|\varphi\|_{C^4}, \|V_0\|_{\mathfrak{X}^5}, \|G^\upsigma\|_{\mathfrak{X}^5}$, but otherwise uniform in $x,\eta,\alpha,L$, such that 
$$
\|\psi^{\eta}_t\|_{C^{4}}\leq Ce^{Ct}
$$ 
by virtue of \cite[Definition 4.3]{gess2024rsgd}. Consider the probability spaces $(\Omega,\cF,\P)$ and $(\Tilde{\Omega},\Tilde{\cF},\Tilde{\P})$ as the projections of the product space $(\Omega\times \Tilde{\Omega},\cF \otimes \Tilde{\cF},\P \times \Tilde{\P})$ in such a way that $(X^{k}(x))_{k\in \naturals}$ and $(X(t,x))_{t\geq 0}$ are independent processes. 
For $x\in (\S^{d-1})^{n}$ and $k\in \naturals$, we get by a telescoping argument
\begin{align*}
&\left|\E \varphi(X^{L}(x))-\E\varphi(X(L\eta,x)) \right|\\
&\quad\leq \left|\sum_{\ell=1}^{L}\E\psi_{(L-\ell)\eta}^{\eta}(X^{\ell}(x))-\E \psi_{(L-\ell)\eta}^{\eta}(X(\eta;\, X^{\ell-1}(x)))\right|\\
&\quad\leq \left|\sum_{\ell=1}^{L}\E \psi_{(L-\ell)\eta}^{\eta}(X^{1}(X^{\ell-1}(x)))-\E\psi_{(L-\ell)\eta}^{\eta}(X(\eta; X^{\ell-1}(x)))\right|\\
&\quad\leq \sum_{\ell=1}^{L}\underset{x\in (\S^{d-1})^{n}}{\max}\left|\E\psi_{(L-\ell)\eta}^{\eta}(X^{1}(x))-\E\psi_{(L-\ell)\eta}^{\eta}(X(\eta; x))\right|\\
&\quad\leq C e^{C\eta L} \eta^2 L \max\{\alpha, \eta\}
\end{align*}
for a universal $C>0$. 
This argument gives, for all $k\in \{0,\ldots,L\}$,
\begin{equation}\label{eq:error_for_all_k}
\left|\E \varphi(X^{k}(x))-\E \varphi(X(k\eta,x) \right|\leq C e^{C\eta L} \eta^2 L \max\{\alpha, \eta\}.
\end{equation}
For $t\in [0,t_L]$ consider $k=\lfloor t/\eta \rfloor$. 
We have 
\begin{align*}
\left|\E\varphi(X^{\eta}(t))-\E\varphi(X(t))\right|\leq &\left|\E\varphi(X^{k}(x))-\E\varphi(X(k\eta,x))\right|\\
+&\left|\E\varphi(X(k\eta,x))-\E\varphi(X(t,x))\right|.
\end{align*}
The first term is bounded by \eqref{eq:error_for_all_k}, and the second is bounded by Dynkin’s formula for the semigroup generated by \eqref{eq:SDE_ito_clean}, namely 
\[
\left|\E\varphi(X(k\eta,x))-\E\varphi(X(t,x))\right|\leq \int_{k\eta}^{t}|\E\mathsf{L}\varphi(X(s,x))|\de s\leq C\eta\|\varphi\|_{C^{4}}.
\]
Combining both the estimates, we obtain the desired result.
\end{proof}

\subsection{Proof of Corollary \ref{cor:weak_error_centered}} \label{proof: gaussian.corollary}

\begin{proof}[Proof of Corollary \ref{cor:weak_error_centered}]
The proof is similar to the one Theorem \ref{thm:weak_error_clean}---we only discuss the differences. A Taylor-Lagrange expansion gives
\begin{equation*}
     \E\varphi(X^{1}(x))-\varphi(x)=\frac{\eta}{2}\;\alpha\; \Hess \,\varphi(x):\E\left[G^{\upsigma}(x,\bm{\theta})\; G^{\upsigma}(x,\bm{\theta})^{\sT}\right]+R(\eta,x),
\end{equation*}
where $R(\eta,x)$ is bounded by 
\[
  R(\eta,x)\lesssim \eta^{\frac32}\left(\frac{\eta \upsigma^{2}}{H}\right)^{\frac32}.
\]
For the continuous process, the same computation gives
\begin{equation*}
    \max_{x\in (\S^{d-1})^{n}}\Big| \E\varphi(X(\eta;x))-\frac{\alpha \eta }{2}\;\E\left\<\Hess \varphi(x)G^{\upsigma}(x,\bm{\theta}),G^{\upsigma}(x,\bm{\theta})\right\>\Big|\leq C \eta^{2}\alpha.
\end{equation*}
All in all,
\begin{equation*}
\max_{x\in (\S^{d-1})^{n}}\left| \E\varphi(X(\eta;x))-\E\varphi(X^{1}(x))\right|\leq  C\eta^2\max\{\alpha,\eta\}.\qedhere
\end{equation*}
Again consider 
$$\psi^{\eta}_t(x)\coloneqq\E \varphi(X(t;x)).$$
In this centered Gaussian case, the limiting generator is a pure diffusion of size $O(\alpha)$, hence there exists a constant $C>0$ depending only on $\|\varphi\|_{C^4}, \|V_0\|_{\mathfrak{X}^5}, \|G^\upsigma\|_{\mathfrak{X}^5}$, but otherwise uniform in $x,\eta,\alpha,L$, such that   $$\|\psi^{\eta}_t\|_{C^{4}}\leq Ce^{C\alpha t}
$$ by virtue of \cite[Definition 4.3]{gess2024rsgd}. We may then conclude as in the previous proof.    
\end{proof}

\subsection{Proof of Lemma \ref{lem:Kernel_regularity}} \label{sec:APP_Diffusive}

\begin{proof}[Proof of Lemma \ref{lem:Kernel_regularity}]
\label{proof:lem_regularity}

For any random object $Y$ depending on $\left(\bV,\bA\right)$, we use
\[
\left\|Y\right\|_{L^p\left(\bA\right)}\coloneqq \left(\E_{\bA}\left\|Y\right\|^p\right)^{\frac1p},
\]
and the conditional norm
\[
\left\|Y\right\|_{L^p\left(\bV\mid \bA\right)} \coloneqq\left(\E_{\bV}\left[\left\|Y\right\|^p \,\middle|\, \bA\right]\right)^{\frac1p}.
\]
We write $\Bar{\bV}\coloneqq\E\bV$, $\Tilde{\bV}\coloneqq \bV-\Bar{\bV}$, $\Bar{\bA}\coloneqq \E \bA$, and  $\Tilde{\bA}\coloneqq \bA-\Bar{\bA}.$
Then
\[
u^{\bA}(X)=\sum_{j=1}^{n}\pi_{i\to j}^{\bA}(X)x_j.
\]
We decompose
\begin{equation}\label{eq:Decomposition_G_i}
G_{i}\left(X,\bm{\theta}\right)
=
\proj_{x_i}\Tilde{\bV}\,u^{\bA}(X)
+
\proj_{x_i}\Bar{\bV}\left(u^{\bA}(X)-\E_{\bA}[u^{\bA}(X)]\right).
\end{equation}
Consequently, using $\E \Tilde{\bV}=0$ and the independence of $\bV$ and $\bA$, the cross term vanishes by conditioning on $\bA$:
\[
\E\left[\left\langle \proj_{x_i}\Tilde{\bV}\,u^{\bA}(X),\;\proj_{x_i}\Bar{\bV}\left(u^{\bA}(X)-\E_{\bA}[u^{\bA}(X)]\right)\right\rangle\right]
=
0,
\]
and therefore
\begin{align*}
\left\|G_i\left(X,\cdot\right)\right\|_{L^{2}\left({\uprho_*}\right)}^{2}
&=
\E\left\|\proj_{x_i}\Tilde{\bV}\, u^{\bA}(X)\right\|^{2}
+
\E\left\|\proj_{x_i}\Bar{\bV}\left(u^{\bA}(X)-\E_{\bA}[u^{\bA}(X)]\right)\right\|^{2}.
\end{align*}

Recall the Hanson--Wright inequality.

\begin{lemma}{\cite[Thm 6.2.1]{vershynin2018high}} \label{lem:Hanson_wright}
Suppose $X\in \reals^{n}$ is a random vector with independent, mean-zero, subGaussian coordinates and $B \in \reals^{n\times n}$. Then for every $t\geq 0$, we have
\begin{equation*}
\P\left(\left|X^{\sT}BX-\E\left[X^{\sT}BX\right]\right|\geq t\right)
\leq
2\exp\left(-c\min\left\{\frac{t^{2}}{M^{4}\left\|B\right\|_{\mathrm{F}}^{2}},\frac{t}{M^{2}\left\|B\right\|_{\op}}\right\}\right),
\end{equation*}
where $M=\max_{i}\left\|X_i\right\|_{\psi_2}$ and $c>0$ is a universal constant.
\end{lemma}

Fix $X\in(\S^{d-1})^n$ and condition on $\bA$, so that $u\coloneqq u^{\bA}(X)$ is deterministic.
Since $\Tilde{\bV}$ has independent, mean-zero, subGaussian entries, the Hoeffding inequality gives
\[
\|(\Tilde{\bV} u)_i\|_{\psi_2}\leq \|\Tilde{\bV} \|_{\psi_2}\left\|u\right\|.
\]
Apply Lemma \ref{lem:Hanson_wright} to $X=\Tilde{\bV}u$, $B=\proj_{x_i}$, using $\|\proj_{x_i}\|_{\op}=1$ and $\|\proj_{x_i}\|_{\mathrm{F}}=\sqrt{d-1}$ to find
\begin{align*}
&\P\left(\left|(\Tilde{\bV}u)^{\sT}\proj_{x_i}\Tilde{\bV} u-\E\left[(\Tilde{\bV}u)^{\sT}\proj_{x_i}\Tilde{\bV} u \,\middle|\, \bA\right]\right|\geq t \,\middle|\, \bA\right)\\
&\hspace{2cm}\leq
2\exp\left(
-\frac{c}{\|\Tilde{\bV}\|_{\psi_2}^{2}\left\|u\right\|^{2}}
\min\left\{
\frac{t^{2}}{\|\Tilde{\bV}\|_{\psi_2}^{2}\left\|u\right\|^{2}\left(d-1\right)},
t
\right\}
\right).
\end{align*}
In particular,
\[
Z\coloneqq
(\Tilde{\bV}u)^{\sT}\proj_{x_i}\Tilde{\bV} u-\E\left[(\Tilde{\bV}u)^{\sT}\proj_{x_i}\Tilde{\bV} u \,\middle|\, \bA\right]
\]
is a subexponential variable satisfying the Orlicz bound
\begin{equation}\label{eq:psi1_bound_Z}
\left\|Z\right\|_{\psi_1\left(\bV\mid\bA\right)}
\lesssim
\|\Tilde{\bV}\|_{\psi_2}^{2}\,\left\|u\right\|^{2}\sqrt{d-1}.
\end{equation}
Subexponential variables satisfy
\begin{equation}\label{eq:moment_exponential}
\left\|X\right\|_{L^{p}}\lesssim p\left\|X\right\|_{\psi_1}.
\end{equation}
(See \cite{vershynin2018high}.)
Applying \eqref{eq:moment_exponential} with $p=\frac32$ and \eqref{eq:psi1_bound_Z} gives
\begin{equation*}
\left\|Z\right\|_{L^{\frac32}\left(\bV\mid\bA\right)}
\lesssim
\|\Tilde{\bV}\|_{\psi_2}^{2}\,\left\|u\right\|^{2}\,\sqrt{d-1}.
\end{equation*}
By the triangle inequality,
\begin{equation}\label{eq:L3_bound_conditionnal}
\left\lVert
\left\langle
\Tilde{\bV} u,\proj_{x_i}\Tilde{\bV} u
\right\rangle
\right\rVert_{L^{\frac32}\left(\bV\mid\bA\right)}
\lesssim
\E_{\bV}\left[\left\langle \Tilde{\bV} u,\proj_{x_i}\Tilde{\bV} u\right\rangle \,\middle|\, \bA\right]
+
\|\Tilde{\bV}\|_{\psi_2}^{2}\left\|u\right\|^{2}\sqrt{d}.
\end{equation}
By the tower property,
\begin{align*}
\E\left\|\proj_{x_i}\Tilde{\bV}u^{\bA}(X)\right\|^{3}
&=
\E_{\bA}\E_{\bV}\left\|\proj_{x_i}\Tilde{\bV}u^{\bA}(X)\right\|^{3}\\
&=
\E_{\bA}\left\|
\left\|\proj_{x_i}\Tilde{\bV}u^{\bA}(X)\right\|^{2}
\right\|_{L^{\frac32}\left(\bV\mid\bA\right)}^{\frac32}.
\end{align*}
Moreover, using $\mathrm{Tr}(\proj_{x_i})=d-1$ and the isotropy of $\Tilde{\bV}$, we have
\[
\E_{\bV}\left[\left\langle \Tilde{\bV}u,\proj_{x_i}\Tilde{\bV}u\right\rangle \,\middle|\, \bA\right]
=
\E_{\bV}\left[\left\|\proj_{x_i}\Tilde{\bV}u\right\|^{2}\,\middle|\,\bA\right]
=
\sigma_{\bV}^{2}\left(d-1\right)\left\|u\right\|^{2}.
\]
Plugging this and \eqref{eq:L3_bound_conditionnal} in the above yields
\begin{equation}\label{eq:from_conditionnal_to_uncondi}
\E\left\|\proj_{x_i}\Tilde{\bV}u^{\bA}(X)\right\|^{3}
\lesssim 
\left(\sigma_{\bV}^{2}\left(d-1\right)\right)^{\frac32}
\left(1+\frac{1}{d^{\frac32}}\right)
\E\left\|u^{\bA}(X)\right\|^{3}.
\end{equation}
Combining \eqref{eq:from_conditionnal_to_uncondi} with the lower bound on the second moment, we obtain
\begin{equation}\label{eq:L3_L2_Vterm}
\left\|\proj_{x_i}\Tilde{\bV} u^{\bA}(X)\right\|_{L^{3}\left({\uprho_*}\right)}
\lesssim 
\left\|\proj_{x_i}\Tilde{\bV} u^{\bA}(X)\right\|_{L^{2}\left({\uprho_*}\right)}.
\end{equation}
We now turn our attention to the second term in \eqref{eq:Decomposition_G_i}.
Set
\(
\bar g\coloneqq \Bar{\bA}x_i,\)
and \(
\tilde g\coloneqq \bW (\bW'^{\sT}x_i)=\bW v,\)
where $v\coloneqq \bW'^{\sT}x_i$. Conditional on $\bW'$, \( (\tilde{g}_k)_{k\in[d]}\) are independent since $\Tilde{g}_k$ only depends on the $k$-th row of $\bW$, and each row is independent.
Moreover, $\tilde g$ has independent centered coordinates with 
\(
\Var(\Tilde{g}_k\,|\,\bW')=\sigma_{\bA}^{2}\|v\|^{2}.
\)
Let $\chi_j:\reals^{n}\rightarrow\reals$ be defined as
\[
\chi_{j}\left(z\right)\coloneqq\frac{e^{z_j}}{\sum_{k=1}^{n} e^{z_k}}.
\]
Define the map $\Phi:\R^d\to\R^d$ by
\[
\Phi\left(g\right)\coloneqq\sum_{j=1}^n \chi_j\left(\upbeta\left\langle \bar g+g,x_1\right\rangle,\ldots,\upbeta\left\langle \bar g+g,x_n\right\rangle\right)x_j,
\]
so that $u^{\bA}(X)=\Phi\left(\tilde g\right)$.
Write
\begin{align*}
u^{\bA}(X)-\E u^{\bA}(X)&=\left(u^{\bA}(X)-\E_{\bW}[u^{\bA}(X)\,|\,\bW']\right)\\
&+\left(\E_{\bW}[u^{\bA}(X)\,|\,\bW']-\E u^{\bA}(X)\right).
\end{align*}
Consider the Doob martingale
\[
M_k\coloneqq \E\left[\Phi\left(\tilde g\right)\,\middle|\bW', \, \tilde g_1,\ldots,\tilde g_k\right]
\]
with increments $\Delta_k\coloneqq M_k-M_{k-1}$.
Then $\Phi\left(\tilde g\right)-\E[\Phi\left(\tilde g\right)|\bW']=\sum_{k=1}^d\Delta_k$ and the Burkholder--Davis--Gundy inequality yields
\begin{equation}\label{eq:BDG_p3}
\left\|\Phi\left(\tilde g\right)-\E[\Phi\left(\tilde g\right)|\bW']\right\|_{L^3\left(\bW|\bW'\right)}
\lesssim 
\left\|\left(\sum_{k=1}^d \E\left[\left\|\Delta_k\right\|^2\,\middle|\, \cF_{k-1}\right]\right)^{\frac12}\right\|_{L^3\left(\bW|\bW'\right)},
\end{equation}
where $\cF_{k}\coloneqq\sigma\left(\bW',\tilde g_1,\ldots,\tilde g_k\right)$.
We bound the increments as
\begin{equation}\label{eq:bound_martingale}
\E\left[\left\|\Delta_k\right\|^{2}\,\middle|\,\cF_{k-1}\right]
\leq
\E\left[\left\|\Phi\left(\tilde g\right)-\Phi(\tilde g^{k})\right\|^{2}\,\middle|\,\cF_{k-1}\right],
\end{equation}
where $\tilde g^{k}$ is obtained from $\tilde g$ by replacing the $k$-th coordinate with an independent copy.
Writing $\tilde g^{k}=\tilde g+\Delta e_k$, we have
\begin{equation}\label{eq:bound_Phi}
\Phi(\tilde g)-\Phi(\tilde g^{k})
=
\upbeta\Delta\int_{0}^{1}B\left(\tilde g+s\,\Delta e_k\right)e_k\,\de s,
\end{equation}
where
\[
B\left(g\right)\coloneqq
\sum_{j=1}^{n}\chi_j\left(g\right) x_jx_j^{\sT}-\Phi\left(g\right)\Phi\left(g\right)^{\sT}.
\]
Plugging \eqref{eq:bound_Phi} into \eqref{eq:bound_martingale} yields
\begin{equation*}
\E\left[\left\|\Delta_k\right\|^{2}\,\middle|\,\cF_{k-1}\right]
\leq
2\upbeta^{2}\sigma_{\bA}^{2}\|v\|^{2}\,\E\left\|\int_{0}^{1}B\left(\tilde g+s\Delta e_k\right)e_k\,\de s\right\|^{2}.
\end{equation*}
Summing over $k$ and using $\left\|B\left(g\right)\right\|_{\op}\leq 1$, we obtain
\begin{align}\label{eq:sum_increments_bound}
\sum_{k=1}^{d}\E\left[\left\|\Delta_k\right\|^{2}\,\middle|\,\cF_{k-1}\right]
&\leq
2\upbeta^{2}\sigma_{\bA}^{2}\|v\|^{2}\sum_{k=1}^{d}\E\left\|\int_{0}^{1}B\left(\tilde g+s\Delta e_k\right)e_k\,\de s\right\|^{2}\nonumber\\
&\leq
2d\upbeta^{2}\sigma_{\bA}^{2}\|v\|^{2}.
\end{align}
Plugging \eqref{eq:sum_increments_bound} in \eqref{eq:BDG_p3} yields
\begin{equation*}
    \left\|\Phi\left(\tilde g\right)-\E[\Phi\left(\tilde g\right)\,|\,\bW']\right\|_{L^3\left(\bW|\bW'\right)}
\lesssim \upbeta \sigma_{\bA}\sqrt{d}\, \left(\E \left\|\bW'^{\sT}x_i\right\|^{3}\right)^{\frac13}
\end{equation*}
Using that the entries of $\bW'$ are subGaussian, we obtain
\begin{equation}\label{eq:ineq_first_part}
    \left\|\Phi\left(\tilde g\right)-\E[\Phi\left(\tilde g\right)|\bW']\right\|_{L^3\left(\bW|\bW'\right)}
\lesssim \upbeta \sigma_{\bA}^{2}d.
\end{equation}
We now look at the second term. 
Define $\phi:\reals^{d}\rightarrow \reals^{d}$ by 
\[
\phi(v)\coloneqq \E \sum_{j=1}^{n} \chi_{j}(\upbeta\<\Bar{g}+\bW v,x_1\>,\ldots,\upbeta\<\Bar{g}+\bW v,x_n\>)x_j.
\]
Define the Doob martingale $$
N_k=\E\left[\phi(v)\,|\, v_1,\ldots,v_k\right],
$$
with increments $\Delta_k\coloneqq N_k-N_{k-1}$.  Notice that $v_k$ are independent since the columns $\bW$ are independent. Then, a similar argument shows that 
\begin{equation}\label{eq:increment_bound}
    \|\phi(v)-\E \phi(v)\|_{L^{3}(\bW')}\lesssim \left\|\left(\sum_{k=1}^{d}\E\left[\|\Delta_k\|^{2}|\cG_{k-1}\right]\right)^{\frac12}\right\|_{L^{3}(\bW')},
\end{equation}
where $\cG_k=\sigma(w_1,\ldots,w_k)$ with $w_i$ is the $i$-th column of $\bW'.$ 
We now bound the increments in the same way as above:
\[
\E\left[\left\|\Delta_k\right\|^{2}\,\middle|\,\cG_{k-1}\right]
\leq
\E\left[\left\|\phi(v)-\phi(\tilde  v^{k})\right\|^{2}\,\middle|\,\cG_{k-1}\right],
\]
where $\tilde v^{k}$ is obtained from $v$ by replacing the $k$-th coordinate with an independent copy. Thus
\begin{equation*}
    \|\phi(v)-\phi(v')\|\leq \upbeta\,\E\|\bW(v-v')\|\leq \upbeta\sqrt{d}\sigma_{\bA}\|v-v'\|.
\end{equation*}
Plugging the above bound in \eqref{eq:increment_bound}, we obtain
\begin{equation}\label{eq:ineq_second_part}
    \|\phi(v)-\E\phi(v)\|_{L^{3}(\bW')}\lesssim \upbeta\sqrt{d}\sigma_{\bA}\sum_{k=1}^{d}\E\left\|\left(\bW'^{\sT} x_i\right)_k\right\|^{2}\lesssim 2\upbeta\sqrt{d}\sigma_{\bA}.
    \end{equation}
Combining \eqref{eq:ineq_second_part} and \eqref{eq:ineq_first_part}, we obtain
\begin{equation*}
\left\|u^{\bA}(X)-\E_{\bA}[u^{\bA}(X)]\right\|_{L^{3}\left(\bA\right)}
\lesssim 
\sqrt{d}\upbeta\sigma_{\bA}.
\end{equation*}
Therefore,
\begin{equation}\label{eq:L3_bound_A_term}
\left\|\proj_{x_i}\Bar{\bV}\left(u^{\bA}(X)-\E_{\bA}[u^{\bA}(X)]\right)\right\|_{L^{3}\left({\uprho_*}\right)}
\lesssim 
\sqrt{d}\upbeta\sigma_{\bA}\left\|\Bar{\bV}\right\|_{\op}.
\end{equation}
Combining \eqref{eq:L3_bound_A_term} and \eqref{eq:L3_L2_Vterm}, we obtain
\begin{equation*}
\left\|G_i\left(X,\bm{\theta}\right)\right\|_{L^{3}\left({\uprho_*}\right)}
\lesssim 
\left\|G_i\left(X,\bm{\theta}\right)\right\|_{L^{2}\left({\uprho_*}\right)}
+
\sqrt{d}\upbeta\sigma_{\bA}\left\|\Bar{\bV}\right\|_{\op}.
\end{equation*}
Taking the maximum over $X\in(\S^{d-1})^n$ and using the lower bound
\begin{align} \label{eq:sigma_lower_bound}
\upsigma=\max_{X\in(\S^{d-1})^n}\|G(X,\cdot)\|_{L^2({\uprho_*})}
&\geq
\left\|G\left((x,\ldots,x),\cdot\right)\right\|_{L^2({\uprho_*})}\nonumber\\
&=\sigma_{\bV}\sqrt{d-1},
\end{align}
valid for any fixed $x\in\S^{d-1}$, since then $u^{\bA}(X)=x$ and the second term in \eqref{eq:Decomposition_G_i} vanishes,
we get item (2) with
\[
C_1=O\left(1+\frac{\upbeta\sigma_{\bA}\|\E\bV\|_{\op}}{\sigma_{\bV}}\right).
\]

We now check item (3). We use the following fact.

\begin{claim}\label{claim:boumal}
If $F:(\S^{d-1})^n\rightarrow\reals^{nd}$ is a smooth tangent vector field and $\tilde{F}$ is any smooth extension to an open neighborhood of $(\S^{d-1})^n$, then for every $k\in \naturals$, there exists $C_k>0$ such that
\begin{equation*}
\max_{X\in (\S^{d-1})^n}\left\|\nabla^{k}F(X)\right\|
\leq
C_k
\max_{s\in\{0,\ldots,k\}}\max_{X\in (\S^{d-1})^n}\left\|D^{s}\tilde{F}(X)\right\|.
\end{equation*}
\end{claim}
(See Proposition 5.31, Theorem 5.9 and Proposition 3.31 in \cite{boumal2023introduction}. Recall that $D^r$ denotes Euclidean derivatives.) Recalling \eqref{eq:Decomposition_G_i}, we look at Euclidean derivatives of $\proj_{x_i}\Tilde{\bV} u^{\bA}(X)$ and $\proj_{x_i}\Bar{\bV}(u^{\bA}(X)-\E_{\bA}[u^{\bA}(X)])$.
We only tackle the first term since the computations are similar. Write, for an open neighborhood of $(\S^{d-1})^n$ in $(\R^d)^n$,
\[
\tilde G_i(X,\bm\theta)\coloneqq \sum_{j=1}^{n}\pi_{i\to j}^{\bA}(X)\proj_{x_i}\left(\Tilde{\bV} x_j\right),
\]
so that $\tilde G_i(X,\bm\theta)=\proj_{x_i}\Tilde{\bV}u^{\bA}(X)$ on $(\S^{d-1})^n$. By the multilinear operator norm identity, for every $s\in\naturals$,
\begin{equation}\label{eq:multilinear_norm}
\left\|D_X^{s}\tilde G_i(X,\bm\theta)\right\|
=
\sup_{H_1,\ldots,H_s\neq 0}
\frac{\left\|D_X^{s}\tilde G_i(X,\bm\theta)\left[H_1,\ldots,H_s\right]\right\|}{\left\|H_1\right\|\cdots\left\|H_s\right\|}.
\end{equation}
We now bound derivatives of the softmax. 
Note that $
\pi_{i\to j}^{\bA}(X)
=\chi_j\left(Z_i(X)\right),$
where
\[
Z_i(X)=\left(\upbeta\left\langle\bA x_i,x_1\right\rangle,\ldots,\upbeta\left\langle\bA x_i,x_n\right\rangle\right)\in\R^n.
\]
Using the Fa\`a di Bruno formula, we have for every $s\in\naturals$,
\begin{equation}\label{eq:faa_di_bruno}
D_{X}^{s}\pi_{i\to j}^{\bA}(X)
=
\sum_{\pi\in \cP_s}
D^{\left|\pi\right|}\chi_j\left(Z_i(X)\right)\;
\left[D^{\left|B_1\right|}Z_i\left[H_{B_1}\right],\ldots,D^{\left|B_{\left|\pi\right|}\right|}Z_i\left[H_{B_{\left|\pi\right|}}\right]\right],
\end{equation}
where $\pi=\left\{B_1,\ldots,B_{\left|\pi\right|}\right\}$ runs over partitions of $[s]$.

\begin{claim}\label{claim:softmax_is_smooth}
There exist $\left(C^{k}_I\right)$ such that
\begin{equation*}
\frac{\partial^{k}\chi_{j}}{\partial z_{r_1}\ldots\partial z_{r_k}}\left(z\right)
=
\chi_{j}\left(z\right)
\sum_{r_1,\ldots,r_k}\sum_{I\subset \left\{r_1,\ldots,r_k\right\}}C^{k}_I
\prod_{q\in\left\{r_1,\ldots,r_k\right\}\setminus I}\chi_{r_q}\left(z\right)\;
1_{r_q=r_q'\, \forall \left(q,q'\right)\in I}.
\end{equation*}
\end{claim}
The claim follows by induction on
\[
\frac{\partial}{\partial z_{r}}\chi_j\left(z\right)
=
\chi_{j}\left(z\right)\chi_r\left(z\right)-1_{r=j}\chi_{j}\left(z\right).
\]
Notice that $Z_i$ is bilinear in $X$:
\begin{align*}
D_{X}Z_i(X)\left[H\right]
&=
\left(\left\langle\bA H_i,x_\ell\right\rangle+\left\langle\bA x_i,H_\ell\right\rangle\right)_{\ell\in[n]},\\
D^{2}_{X}Z_i(X)\left[H^1,H^2\right]
&=
2\left(\left\langle\bA H^{1}_i,H^{2}_\ell\right\rangle\right)_{\ell\in[n]},
\end{align*}
and $D_{X}^{s}Z_i(X)=0$ for all $s\geq 3$.
Combining   Claim \ref{claim:softmax_is_smooth} with \eqref{eq:faa_di_bruno} and the above bounds on derivatives of $Z_i$
gives the following uniform estimate: for every $s\in\{0,\ldots, 5\}$ there exists $C_s>0$ such that for all $X\in(\S^{d-1})^n$ and all $H_1,\ldots,H_s$,
\begin{equation}\label{eq:pi_derivative_bound}
\left|D_X^{s}\pi_{i\to j}^{\bA}(X)\left[H_1,\ldots,H_s\right]\right|
\leq
C_s\,\chi_j\left(Z_i(X)\right)(\upbeta\|\bA\|_{\op})^{s}\left\|H_1\right\|\cdots\left\|H_s\right\|.
\end{equation}
Equivalently, defining the $s$-linear form $\varphi_s$ by
\[
D_X^{s}\pi_{i\to j}^{\bA}(X)\left[H_1,\ldots,H_s\right]
=
\chi_j\left(Z_i(X)\right)\,\varphi_s\left(\upbeta,\bA,X\right)\left[H_1,\ldots,H_s\right],
\]
we have
\begin{equation*}
\left\|\varphi_s(\upbeta,\bA,X)\right\|
\leq
C_s(\upbeta\|\bA\|_{\op})^{s},
\end{equation*}
for $s\in\{0,\ldots,5\}$ uniformly in $X\in(\S^{d-1})^n$.
A direct computation also shows that for all $s\in \naturals$,
\begin{equation}\label{eq:proj_derivative_bound}
\left\|D^{s}_{X}\left(\proj_{x_i}\Tilde{\bV} x_j\right)\right\|\lesssim_s \|\Tilde{\bV}\|_{\op},
\qquad
D_{X}^{s}\left(\proj_{x_i}\Tilde{\bV} x_j\right)=0 \ \text{ for } s\geq 4.
\end{equation}
Fix $s\in\{0,\ldots,5\}$. By the Leibniz formula, for any $H_1,\ldots,H_s$,
\begin{align*}
&D_{X}^{s}\left(\pi_{i\to j}^{\bA}(X)\proj_{x_i}\Tilde{\bV} x_j\right)\left[H_1,\ldots,H_s\right]\\
&\hspace{2cm}=
\sum_{k=0}^{s}\sum_{\substack{S\subset [s]\\ \left|S\right|=k}}
D_{X}^{k}\left(\pi_{i\to j}^{\bA}(X)\right)\left[H_{S}\right]\,
D_{X}^{s-k}\left(\proj_{x_i}\Tilde{\bV} x_j\right)\left[H_{S^{c}}\right].
\end{align*}
Using \eqref{eq:pi_derivative_bound} and \eqref{eq:proj_derivative_bound}, we obtain
\begin{align*}
&\left\|D_{X}^{s}\left(\pi_{i\to j}^{\bA}(X)\proj_{x_i}\Tilde{\bV} x_j\right)\left[H_1,\ldots,H_s\right]\right\|\\
&\hspace{2cm}\lesssim_s
\chi_j\left(Z_i(X)\right)\|\Tilde{\bV}\|_{\op}
\left(1+(\upbeta\|\bA\|_{\op})^{s}\right)
\left\|H_1\right\|\cdots\left\|H_s\right\|.
\end{align*}
Summing in $j$ and using $\sum_{j=1}^n\chi_j\left(Z_i(X)\right)=1$, we get
\begin{equation}\label{eq:Dr_bound}
\left\|D_X^{s}\tilde G_i\left(X,\bm\theta\right)\left[H_1,\ldots,H_s\right]\right\|
\lesssim_s
\|\Tilde{\bV}\|_{\op}
\left(1+(\upbeta\|\bA\|_{\op})^{s}\right)
\left\|H_1\right\|\cdots\left\|H_s\right\|.
\end{equation}
Combining \eqref{eq:Dr_bound} with \eqref{eq:multilinear_norm} yields
\begin{equation}\label{eq:Dr_op_norm_bound}
\max_{X\in(\S^{d-1})^n}\left\|D_X^{s}\tilde G_i\left(X,\bm\theta\right)\right\|
\lesssim_s
\|\Tilde{\bV}\|_{\op}
\left(1+(\upbeta\|\bA\|_{\op})^{s}\right),
\end{equation}
for $s\in\{0,\ldots,5\}$.
Now apply   Claim \ref{claim:boumal} with $F(X)=\proj_{x_i}\Tilde{\bV}u^{\bA}(X)$ and $\tilde F=\tilde G_i$.
For $k\in\{0,\ldots,5\}$ we obtain, using \eqref{eq:Dr_op_norm_bound},
\begin{align*}
\max_{X\in(\S^{d-1})^n}\left\|\nabla^{k}\left(\proj_{x_i}\Tilde{\bV}u^{\bA}(X)\right)\right\|^{2}
&\lesssim_k
\max_{s\in\{0,\ldots, k\}}\max_{X\in(\S^{d-1})^n}\left\|D_X^{s}\tilde G_i\left(X,\bm\theta\right)\right\|^{2}\nonumber\\
&\lesssim_k
\|\Tilde{\bV}\|_{\op}^{2}\left(1+(\upbeta\|\bA\|_{\op})^{10}\right).
\end{align*}
Note that $\bA=\bW \bW'^{\sT}$ is sub-exponential as the product of two subGaussian variables. Taking expectations and using Assumption \ref{ass:high_order_short} (and the independence of $\bV$ and $\bA$), we get
\begin{align*}
\E\max_{X\in(\S^{d-1})^n}\left\|\nabla^{k}\left(\proj_{x_i}\Tilde{\bV}u^{\bA}(X)\right)\right\|^{2}
&\lesssim_k
\E\left[\|\Tilde{\bV}\|_{\op}^{2}\left(1+(\upbeta\|\bA\|_{\op})^{10}\right)\right]\nonumber\\
&\lesssim_k
\E\|\Tilde{\bV}\|_{\op}^{2}
+
\upbeta^{10}\E\|\Tilde{\bV}\|_{\op}^{2}\E\left\|\bA\right\|_{\op}^{10}\nonumber\\
&\lesssim_k
d\sigma_{\bV}^{2}\left(1+\upbeta^{10}\sigma_{\bA}^{10}d^{5}\right).
\end{align*}
The same argument applies to the second term in \eqref{eq:Decomposition_G_i} (with $\Tilde{\bV}$ replaced by $\Bar{\bV}$ and $u^{\bA}(X)$ replaced by $u^{\bA}(X)-\E_{\bA}[u^{\bA}(X)]$),
yielding the same type of bound. Summing the two contributions, we conclude that for every $k\in\{0,\ldots,5\}$,
\begin{equation}\label{eq:moment_bound_fullG}
\E\max_{X\in(\S^{d-1})^n}\left\|\nabla^{k}G\left(X,\bm\theta\right)\right\|^{2}
\lesssim_k
d\sigma_{\bV}^{2}\left(1+\upbeta^{10}\sigma_{\bA}^{10}d^{5}\right).
\end{equation}
Using the definition
$\left\|G(\cdot,\bm\theta)\right\|_{C^5}=\max_{0\leq k\leq 5}\max_{X\in(\S^{d-1})^n}\left\|\nabla^{k}G(X,\bm\theta)\right\|$,
we obtain from \eqref{eq:moment_bound_fullG}
\begin{align*}
\left\|\left\|G(\cdot,\bm\theta)\right\|_{C^5}\right\|_{L^2({\uprho_*})}^{2}
=
\E\left\|G(\cdot,\bm\theta)\right\|_{C^5}^{2}
&\leq
\sum_{k=0}^{5}\E\max_{X\in(\S^{d-1})^n}\left\|\nabla^{k}G(X,\bm\theta)\right\|^{2}\\
&\lesssim
d\sigma_{\bV}^{2}\left(1+\upbeta^{10}\sigma_{\bA}^{10}d^{5}\right).
\end{align*}
Combining this with the lower bound \eqref{eq:sigma_lower_bound} yields
\begin{equation*}
\left\|\left\|G(\cdot,\bm\theta)\right\|_{C^5}\right\|_{L^2({\uprho_*})}
\lesssim
\left(1+\upbeta^{10}\sigma_{\bA}^{10}d^{5}\right)\upsigma.
\end{equation*}
In particular, $\left\|G(\cdot,\bm\theta)\right\|_{C^5}<\infty$ for ${\uprho_*}$-a.e.\ $\bm\theta$, so
$G(\cdot,\bm\theta)\in C^5\left((\S^{d-1})^n;\mathsf{T}(\S^{d-1})^n\right)$ for ${\uprho_*}$-a.e.\ $\bm\theta$.

The bound for $\|\|\nabla_{G(\cdot,\bm\theta)}G(\cdot,\bm\theta)\|_{C^5}\|_{L^2({\uprho_*})}$
is obtained by applying the same Euclidean-derivative bounds to the map
$X\mapsto \nabla_{G(\cdot,\bm\theta)}G(\cdot,\bm\theta)(X)$, noting that it is smooth in $X$ and its Euclidean derivatives
up to order $5$ can be bounded by the same polynomial in $\upbeta\|\bA\|_{\op}$ times the appropriate moments of $\bV$;
this yields again a bound of the form
\[
\left\|\left\|\nabla_{G(\cdot,\bm\theta)}G(\cdot,\bm\theta)\right\|_{C^5}\right\|_{L^2({\uprho_*})}
\lesssim
\left(1+\upbeta^{10}\sigma_{\bA}^{10}d^{5}\right)\upsigma,
\]
possibly after enlarging the implicit constant and using \eqref{eq:sigma_lower_bound}.
This completes the proof of item (3).

Finally, item (1) follows from item (3): since $\left\|\left\|G(\cdot,\bm\theta)\right\|_{C^5}\right\|_{L^2({\uprho_*})}<\infty$,
we have $\E\max_{X}\left\|\nabla^{k}G(X,\bm\theta)\right\|<\infty$ for $k\leq 4$, hence we may differentiate under the expectation to obtain
$\nabla^{k}b(X)=\E\nabla^{k}G(X,\bm\theta)$ and continuity in $X$ by dominated convergence.
\end{proof}

\subsection{Proof of Theorem \ref{thm:PoC_wellposedness}} \label{sec: thm.5.proof}

The proof is an adaptation to the cylindrical case of results in \cite{carmona2018probabilistic}. 

\begin{proposition}\label{prop:satisfying_MF}
Under Assumption \ref{ass:high_order_short}, 
\begin{equation*}
    \|G_{\mu}(x, \cdot)-G_{\nu}(y, \cdot)\|_{L^{2}({\uprho_*})}\leq C\left(W_{2}(\mu,\nu)\,+\, \|x-y\|\right)
    \end{equation*}
for all $(\mu,\nu)\in \cP(\S^{d-1})$ and $(x,y)\in \S^{d-1}$, where 
\begin{equation*}
C=O\left(\upbeta^{2}e^{4\upbeta\|\Bar{\bA}\|}e^{\upbeta^{2}d\sigma_{\bA}^{2}}\right).
\end{equation*}
\end{proposition}

\begin{proof}[Proof of  Proposition \ref{prop:satisfying_MF}]
We write $\Bar{\bV}\coloneqq\E\bV$, $\Tilde{\bV}\coloneqq \bV-\Bar{\bV}$, $\Bar{\bA}\coloneqq \E \bA$, and  $\Tilde{\bA}\coloneqq \bA-\Bar{\bA}.$
Recall \eqref{eq:Decomposition_G_i}:
\begin{equation}\label{eq:decomposition_G_mf}
    G_{\mu}(x,\bm{\theta})=\proj_{x}\Tilde{\bV} \mathsf{m}_{\bA}[\mu](x)+\proj_{x}\Bar{\bV}\int\left(\frac{e^{\upbeta\<\bA x,y\>}}{\mathscr{Z}_{\bA}[\mu](x)}-\E\frac{e^{\upbeta\<\bA x,y\>}}{\mathscr{Z}_{\bA}[\mu](x)}\right)y\mu(\de y).
\end{equation}
Denote 
$$\Cov_{\mu}^{\bA}(x)\coloneqq\int \left(\frac{e^{\upbeta\<\bA x,y\>}}{\mathscr{Z}_{\bA}[\mu](x)}-\E\frac{e^{\upbeta\<\bA x,y\>}}{\mathscr{Z}_{\bA}[\mu](x)}\right)y\mu(\de y),
$$ 
and 
$$
\mathsf{n}_{\bA}[\mu](x)=\int \frac{e^{\upbeta\<\bA x,y\>}}{\mathscr{Z}_{\bA}[\mu](x)} y\mu(\de y).
$$
By independence of $\bV,\bA$, taking the second moment of \eqref{eq:decomposition_G_mf} yields
\begin{equation*}
\|G_{\mu}(x,\cdot)\|_{L^{2}({\uprho_*})}=d\sigma_{\bV}^{2}\left\|\proj_{x}\mathsf{m}_{\bA}[\mu](x)\right\|_{L^{2}({\uprho_*})}^{2}+\E\left\|\proj_{x}\Bar{\bV}\Cov_{\mu}^{\bA}(x)\right\|^{2},
\end{equation*}
where 
\[
\mathsf{m}_{\bA}[\mu](x)\coloneqq\int \frac{e^{\upbeta\<\bA x,y\>}}{\mathscr{Z}_{\bA}[\mu](x)}y\mu(\de y) = \frac{\mathsf{n}_{\bA}[\mu](x)}{\mathscr{Z}_{\bA}[\mu](x)}.
\]
Then, we have the following identity
\begin{align*}
    \|G_{\mu}(x,\cdot)-G_{\nu}(y,\cdot)\|_{L^{2}({\uprho_*})}^{2}= d\sigma_{\bV}^{2}&\left\|\proj_{x}\mathsf{m}_{\bA}[\mu](x)-\proj_{y}\mathsf{m}_{\bA}[\nu](y)\right\|_{L^{2}({\uprho_*})}^{2}\\
    +&\left\|\proj_{x}\Bar{\bV}\Cov_{\mu}^{\bA}(x)-\proj_{y}\Bar{\bV}\Cov_{\nu}^{\bA}(y)\right\|_{L^{2}({\uprho_*})}^{2}.
\end{align*}
Almost surely $\|\bA\|_{\op}<\infty$, so for all $x\in \S^{d-1}$, $e^{-\upbeta\|\bA\|_{\op}}\leq \mathscr{Z}_{\mu}(x)\leq e^{\upbeta\|\bA\|_{\op}}$.
So we have
\begin{align*}
    \left\|\frac{\mathsf{n}_{\bA}[\mu](x)}{\mathscr{Z}_{\bA}[\mu](x)}-\frac{\mathsf{n}_{\bA}[\nu](x)}{\mathscr{Z}_{\bA}[\nu](x)}\right\|&=\left\|\frac{\mathsf{n}_{\bA}[\mu](x)-\mathsf{n}_{\bA}[\nu](x)}{\mathscr{Z}_{\bA}[\mu](x)}\right\|\\
    &+\left\|\mathsf{n}_{\bA}[\nu](x)\left(\frac{1}{\mathscr{Z}_{\bA}[\mu](x)}-\frac{1}{\mathscr{Z}_{\bA}[\nu](x)}\right)\right\|.
\end{align*}
Focus on the first term on the right-hand side. Using that Lipschitz constant of $y\mapsto e^{\upbeta\<\bA x,y\>}\proj_{x}y$ is bounded by $e^{\upbeta\|\bA\|_{\op}}\|\bV\|_{\op}(1+\upbeta\|\bA\|_{\op})$, and Kantorovich-Rubinstein duality, we have 
\begin{equation*}
    \left\|\mathsf{n}_{\bA}[\mu](x)-\mathsf{n}_{\bA}[\nu](x)\right\|\leq e^{\upbeta\|\bA\|_{\op}}\|\bV\|_{\op}(1+\upbeta\|\bA\|_{\op})W_1(\mu,\nu).
\end{equation*}
It ensues that
\begin{equation*}
\left\|\frac{\mathsf{n}_{\bA}[\mu](x)-\mathsf{n}_{\bA}[\nu](x)}{\mathscr{Z}_{\bA}[\mu](x)}\right\|\leq e^{2\upbeta\|\bA\|_{\op}}\|\bV\|_{\op}(1+\upbeta\|\bA\|_{\op})W_2(\mu,\nu).
\end{equation*}
By the same reasoning, we have 
\begin{equation*}
    \left\|\mathsf{n}_{\bA}[\nu](x)\left(\frac{1}{\mathscr{Z}_{\bA}[\mu](x)}-\frac{1}{\mathscr{Z}_{\bA}[\nu](x)}\right)\right\|\leq \upbeta\|\bV\|_{\op}e^{4\upbeta\|\bA\|_{\op}}W_{2}(\mu,\nu).
\end{equation*}
Combining the two, we end up with 
\begin{equation}\label{eq:Lipschitz_measure}
    \left\|\frac{\mathsf{n}_{\bA}[\mu](x)}{\mathscr{Z}_{\bA}[\mu](x)}-\frac{\mathsf{n}_{\bA}[\nu](x)}{\mathscr{Z}_{\bA}[\nu](x)}\right\|\leq \|\bV\|_{\op}e^{4\upbeta\|\bA\|_{\op}}(1+\upbeta\|\bA\|_{\op}+\upbeta)W_{2}(\mu,\nu).
    \end{equation}
By the same kind of argument, 
\begin{equation}\label{eq:Lipschitz_x}
    \left\|\frac{\mathsf{n}_{\bA}[\nu](x)}{\mathscr{Z}_{\bA}[\nu](x)}-\frac{\mathsf{n}_{\bA}[\nu](y)}{\mathscr{Z}_{\bA}[\nu](y)}\right\| \leq \upbeta\|\bA\|_{\op}e^{2\upbeta\|\bA\|_{\op}} \|x-y\|.
    \end{equation}
    Introduce $C(\bA,\bV,\upbeta)\coloneqq \upbeta(\|\bA\|_{\op}e^{2\upbeta\|\bA\|_{\op}}+\|\bV\|_{\op}e^{4\upbeta\|\bA\|_{\op}}).$ Combining \eqref{eq:Lipschitz_x} and \eqref{eq:Lipschitz_measure}, and using Lipschitz character of $x\mapsto \proj_x$, we obtain
\begin{equation}\label{eq:bound_first_term}
    \left\|\proj_{x}\mathsf{m}_{\bA}[\mu](x)-\proj_{y}\mathsf{m}_{\bA}[\nu](y)\right\|_{L^{2}({\uprho_*})}\leq \E C(\bA,I_d,\upbeta) \left(\|x-y\|+W_{2}(\mu,\nu)\right).
\end{equation}
On the other hand,
\begin{align*}
    &\left\|\proj_x\Bar{\bV}\Cov_\mu^{\bA}(x)-\proj_y\Bar{\bV}\Cov_\mu^{\bA}(y)\right\|_{L^{2}({\uprho_*})}\\
    &\hspace{2cm}\leq \left\|\proj_x\Bar{\bV}\mathsf{m}_{\bA}[\mu](x)-\proj_y \Bar{\bV}\mathsf{m}_{\bA}[\nu](y)\right\|_{L^{2}({\uprho_*})}\\
    &\hspace{2.5cm}+\left\|\E_{\bA}[\proj_x\Bar{\bV}\mathsf{m}_{\bA}[\mu](x)-\proj_y \Bar{\bV}\mathsf{m}_{\bA}[\nu](y)]\right\|_{L^{2}({\uprho_*})}\\
    &\hspace{2cm}\leq 2\left\|\proj_x\Bar{\bV}\mathsf{m}_{\bA}[\mu](x)-\proj_y \Bar{\bV}\mathsf{m}_{\bA}[\nu](y)\right\|_{L^{2}({\uprho_*})}.
\end{align*}
All in all,
\begin{equation}\label{eq:Lipschitz_en_cov}
   \left\|\proj_x\Bar{\bV}\Cov_\mu^{\bA}(x)-\proj_y\Bar{\bV}\Cov_\mu^{\bA}(y)\right\|_{L^{2}({\uprho_*})}\leq \E K(\Bar{\bV},\bA,\upbeta)\left(\|x-y\|+W_{2}(\mu,\nu)\right),
\end{equation}
where 
\begin{equation*}
K(\Bar{\bV},\bA,\upbeta)\coloneqq \upbeta\left(\|\bA\|_{\op}e^{2\upbeta\|\bA\|_{\op}}+\|\Bar{\bV}\|_{\op}e^{4\upbeta\|\bA\|_{\op}}\right).
\end{equation*}
Combining \eqref{eq:bound_first_term} and \eqref{eq:Lipschitz_en_cov}, we deduce 
\begin{align*}
\|G_{\mu}(x,\cdot)&-G_{\nu}(y,\cdot)\|_{L^{2}({\uprho_*})}^{2}\\
&\leq \left(d\sigma_{\bV}^{2}\E C(\bA,I_d,\upbeta)+\E K(\Bar{\bV},\bA,\upbeta)\right)(\|x-y\|+W_{2}(\mu,\nu)).
\end{align*}
Using \cite[Theorem 4.4.5]{vershynin2018high}, we know that $\|\bA\|_{\op}$ is subGaussian with $\|\,\|\bA\|_{\op}\,\|_{\psi_2}\lesssim \sqrt{d}\sigma_{\bA}$, so
\[
\E e^{c\upbeta \|\bA\|_{\op}}\leq e^{c\upbeta \|\Bar{\bA}\|_{\op}}e^{c^{2}\upbeta^{2}d\sigma_{\bA}^{2}}.
\]
From the above bounds, and using $\upsigma^2\gtrsim d \sigma_{\bV}^{2}$ and  $\|\Hat{\bV}\|_{\op}\lesssim d\sigma_{\bV}^{2}$, we deduce 
\begin{align*}
   \|G_{\mu}(x,\cdot)&-G_{\nu}(y,\cdot)\|_{L^{2}({\uprho_*})}^{2}\\
   &\leq \left(d\sigma_{\bV}^{2}+\|\Bar{\bV}\|_{\op}\right)\upbeta^{2}e^{4\upbeta\|\Bar{\bA}\|_{\mathrm{op}}}e^{c^{2}\upbeta^{2}d\sigma_{\bA}^{2}}\,(\|x-y\|+W_{2}(\mu,\nu))\\
   &\lesssim \upbeta^{2}e^{4\upbeta\|\Bar{\bA}\|_{\mathrm{op}}}e^{c^{2}\upbeta^{2}d\sigma_{\bA}^{2}} \upsigma^2 \left(\|x-y\|+W_{2}(\mu,\nu)\right).\qedhere
\end{align*}
\end{proof}

\begin{proposition}\label{prop:mckean_vlasov_common_noise}

Let $\mu_0\in\mathcal{P}(\S^{d-1})$ and let $x_0$ be $\cF_0$-measurable, independent of $\mathsf{W}$, with law $\mu_0$.
Under Assumption \ref{ass:high_order_short}, there exists a unique adapted continuous process $x$ such that, with $\mu(t) = \mathsf{Law}(x(t) \mid \mathscr{F}^{\mathsf{W}}_t)$, the pair $(x,\mu)$ solves \eqref{eq:non_linear_SDE_common}
on $[0, T]$ in the sense of Definition \ref{def:nonlinear_SDE_common}.
\end{proposition}

\begin{proof}[Proof of  Proposition \ref{prop:mckean_vlasov_common_noise}]
The proof is based on a fixed-point argument and is given in \cite[Proposition 2.8, Volume II]{carmona2018probabilistic}. We briefly explain the adaptation. Fix $T>0$.
Let $(x,\mu)$ and $(x',\mu')$ be two solutions on the same space, driven by the same common noise $\mathsf W$, with the same initial condition.
Consider the martingale 
\begin{equation}\label{eq:M_t_mart_MF}
M_t\coloneqq \sqrt{\alpha}\int_{0}^{t}\int_{\Uptheta}\left(G_{\mu(s)}(x(s);\bm{\theta})-G_{\mu'(s)}(x'(s);\bm{\theta})\right)\,\mathsf{W}(\de s,\de \bm{\theta}).
\end{equation}
Applying the Burkholder--Davis--Gundy inequality for $\R^{d}$-valued martingales to \eqref{eq:M_t_mart_MF}, we have
\[
\E \sup_{s\in[0, t]}\|M_s\|^2
\lesssim\alpha\E\int_0^t \left\|G_{\mu(s)}(x(s);\cdot)-G_{\mu'(s)}(x'(s);\cdot)\right\|_{L^2({\uprho_*})}^2\,\de s.
\]
Now use the coupling bound for conditional laws:
since $\mu(s)=\mathsf{Law}(x(s)\mid \cF_s^{\mathsf{W}})$ and $\mu'(s)=\mathsf{Law}(x'(s)\mid \cF_s^{\mathsf{W}})$, the conditional joint law
of $(x(s),x'(s))$ given $\cF_s^{\mathsf{W}}$ is a coupling of $(\mu(s),\mu'(s))$, hence
\[
W_2(\mu(s),\mu'(s))^2
\le
\E\left[\|x(s)-x'(s)\|^2\,\middle|\,\cF_s^{\mathsf{W}}\right]\quad\text{a.s. in } s.
\]
The right-hand side term can be bounded as
\begin{align*}
\E\sup_{s\in[0,t]}\|x(s)-x'(s)\|^{2}
&\leq \E\sup_{0\leq s\leq t}\|M_s\|^{2}+C(T)\int_{0}^{t}\E\|x(s)-x'(s)\|^2\,\de s\\
&\lesssim_T \E\int_{0}^{t}\left\|G_{\mu(s)}(x(s);\cdot)-G_{\mu'(s)}(x'(s);\cdot)\right\|_{L^{2}({\uprho_*})}^{2}\,\de s\\
&\hspace{1.1cm}+\int_{0}^{t}\E\|x(s)-x'(s)\|^2\,\de s\\
&\lesssim_T \E \int_{0}^{t}\left\|G_{\mu(s)}(x(s);\cdot)-G_{\mu(s)}(x'(s);\cdot)\right\|^{2}_{L^{2}({\uprho_*})}\\
&\hspace{1.15cm}+\left\|G_{\mu(s)}(x'(s);\cdot)-G_{\mu'(s)}(x'(s);\cdot)\right\|^{2}_{L^{2}({\uprho_*})}\,\de s\\
&\hspace{1.1cm}+\int_{0}^{t}\E\|x(s)-x'(s)\|^2\,\de s\\
&\lesssim_T \E \int_{0}^{t}\left(\|x(s)-x'(s)\|^{2}+W_2(\mu(s),\mu'(s))^{2}\right)\de s\\
&\lesssim_T \E \int_{0}^{t}\|x(s)-x'(s)\|^{2}\de s,
\end{align*}
where we used the regularity of $G$ with respect to $\mu$ and $x$. The end of the argument is similar to the proof of \cite[Proposition 2.8, Volume II]{carmona2018probabilistic}.
\end{proof}

\begin{proposition} \label{prop: prop.3}
    Let $(x, \mu)$ be a solution of \eqref{eq:non_linear_SDE_common} in the sense of Definition \ref{def:nonlinear_SDE_common}. Then $\mu$ is a weak solution of \eqref{eq: the.spde} in the sense of Definition \ref{def:weak_spde}.
\end{proposition}

The proof is a routine application of Itô’s formula followed by conditioning on the common-noise filtration.

\begin{proposition} \label{prop: poc}
Let $X^n = §(x^{n}_1,\ldots, x^{n}_1)$ solve 
\begin{equation*}
\de x^{n}_i(t)=\int_{\Uptheta}G_{\mu^{n}(t)}(x_{i}^n(t);\bm{\theta})\mathsf{W}(\de \bm{\theta}, \de t),
\end{equation*}
with $\mu^n = \mu_{X^n}$. 
Let $(\bar x_1,\ldots,\bar x_n)$ be conditionally i.i.d., given $(\mathscr{F}^{\mathsf W}_t)_{t\geq0}$, copies of the unique solution $x$ of \eqref{eq:non_linear_SDE_common}, driven by the same common noise $\mathsf{W}$, and let $\bar\mu$ be their empirical measure. Then for every $T>0$ there exists $C_T>0$ such that 
\begin{align*}
    \max_{i\in[n]} \E \sup_{t\in[0, T]} \left\|x^{n}_i(t)-\bar x_i(t)\right\|^2 &+ \E \sup_{t\in[0, T]} W_2^2(\mu^n(t), \bar\mu^n(t)) \\
    &\leq  \frac{C_T}{n} \sum_{i=1}^n \E \left\|x^{n,i}(0)-\bar x^i(0)\right\|^2.
\end{align*}
In particular, after choosing an optimal coupling of the initial data, the right hand side may be written as $C_T \E W_2^2(\mu^n(0), \bar \mu^n(0))$.
\end{proposition}

The proof follows directly from  \cite[Theorem 2.12, Volume II]{carmona2018probabilistic}.
We can now conclude with

\begin{proof}[Proof of Theorem \ref{thm:PoC_wellposedness}]
By Proposition \ref{prop:satisfying_MF}, the coefficient $G_\mu(x,\cdot)$
 is globally Lipschitz in $x$ and $\mu$. Proposition \ref{prop:mckean_vlasov_common_noise} therefore yields existence and pathwise uniqueness for the conditional McKean–Vlasov SDE \eqref{eq:non_linear_SDE_common}. Proposition \ref{prop: prop.3} shows that the conditional law of any solution is a weak solution of the SPDE \eqref{eq: the.spde}. Conversely, by a superposition principle for conditional McKean–Vlasov equations \cite{lacker2022superposition}, any weak solution of \eqref{eq: the.spde} can be lifted to a solution of \eqref{eq:non_linear_SDE_common}; uniqueness of the latter implies uniqueness of the former. 
\end{proof}

\section{The Gaussian case} \label{sec:EXAMPLES_APP}

\subsection{Preliminary lemmas} We begin with a couple of particularly useful lemmas.

\begin{lemma} \label{lem:lemma_app}
Under \eqref{eq: tformers.at.initialization}, we have  $b_{{\uprho_*}}[\mu]\equiv 0$ and
\begin{equation*}\label{eq:Kernel_Gaussian}
    \mathsf{K}[\mu](x_i,x_j)=\frac1d\proj_{x_i} \E_\bA \< \mathsf{m}_{\upbeta,\bA}[\mu](x_i),\mathsf{m}_{\upbeta,\bA}[\mu](x_j)\>  I_d\,\proj_{x_j},
\end{equation*} 
where 
\begin{equation*} \label{eq: mbetaA}
    \mathsf{m}_{\upbeta,\bA}[\mu](x):=\frac{1}{\mathscr{Z}_{\bA}[\mu](x)}\int e^{\upbeta\<\bA x,y\>}y\mu(\de y).
\end{equation*}
Here, $\mu=\mu_X$ is the empirical measure.
\end{lemma}

\begin{proof}[Proof of Lemma \ref{lem:lemma_app}]

We compute
\begin{align*}
    \mathsf{K}[\mu](x_i, x_j)&=\proj_{x_i}\,\E\iint\frac{e^{\upbeta\<\bA x_i,y\>}}{\mathscr{Z}_{\bA}[\mu](x_i)}\frac{e^{\upbeta\<\bA x_j,y'\>}}{\mathscr{Z}_{\bA}[\mu](x_j)}\bV yy'^{\sT}\bV^{\sT}\mu(\de y)\mu(\de y')\proj_{x_j}\\
    &=\proj_{x_i}\,\E\int\frac{e^{\upbeta\<\bA x_i,y\>}}{\mathscr{Z}_{\bA}[\mu](x_i)}\frac{e^{\upbeta\<\bA x_j,y'\>}}{\mathscr{Z}_{\bA}[\mu](x_j)}\E\left[\bV yy'^{\sT}\bV^{\sT}\right] \mu(\de y)\mu(\de y')\proj_{x_j}\\
&=\sigma_{\bV}^{2}\proj_{x_i}\,\E\int\frac{e^{\upbeta\<\bA x_i,y\>}}{\mathscr{Z}_{\bA}[\mu](x_i)}\frac{e^{\upbeta\<\bA x_j,y'\>}}{\mathscr{Z}_{\bA}[\mu](x_j)}\langle y,y'\rangle \mu(\de y)\mu(\de y') I_d\, \proj_{x_j}
\end{align*}
where we used $\E_{\bV}[\bV y y'^{\sT}\bV^{\sT}]=\sigma_{\bV}^{2}\<y,y'\>I_d$ and $\sigma_{\bV}^{2}=1/d$.
\end{proof}

\begin{lemma}\label{lem:Ito_formula}
Fix $i,j\in[n]^2$ and define
\(
R_{ij}(t)\coloneqq \langle x_i(t),x_j(t)\rangle.
\)
We have
\begin{equation}\label{eq:Rij_decomp_clean}
\de R_{ij}(t)
=
\de M_{ij}(t)
+
\cD_{ij}(t)\,\de t,
\end{equation}
where the local martingale part is
\begin{equation}\label{eq:martingale_Mij_clean}
\de M_{ij}(t)
=
\int_\Uptheta
\left(\langle G_i(t,\bm{\theta}),x_j(t)\rangle+\langle x_i(t),G_j(t,\bm{\theta})\rangle\right)
\,\mathsf{W}(\de\bm{\theta},\de t),
\end{equation}
and the drift is
\begin{align}\label{eq:Dij_explicit_clean}
\cD_{ij}(t)
&=
\frac1d\left(d-2+R_{ij}(t)^2\right)\,s_{\mu(t)}(x_i(t),x_j(t))\nonumber\\
&-
\frac{(d-1)}{2d}\,R_{ij}(t)\left(s_{\mu(t)}(x_i(t))+s_{\mu(t)}(x_j(t))\right),
\end{align} 
where $\mu(t) = \mu_{X(t)}$ and 
\begin{align*}\label{eq:s_mu_defs_clean}
s_\mu(x)&\coloneqq \E\left\|\mathsf{m}_{\upbeta,\bA}[\mu](x)\right\|^2,\nonumber\\
s_\mu(x,x')&\coloneqq \E\left\langle \mathsf{m}_{\upbeta,\bA}[\mu](x),\mathsf{m}_{\upbeta,\bA}[\mu](x')\right\rangle.
\end{align*}
\end{lemma}

\begin{proof}[Proof of Lemma \ref{lem:Ito_formula}]
    Since $\E \bV=0$ entrywise and $\bV$ is independent of $\bA$, we have $b_{\uprho_*}[\mu](x)=\E B_{\bm{\theta}}[\mu](x)=0$ and hence
\(
\xi_{\bm{\theta}}[\mu](x)=\bV\,\mathsf{m}_{\upbeta,\bA}[\mu](x).
\)
Set
\[
G_i(X,\bm{\theta})
\coloneqq \proj_{x_i} \bV\,\mathsf{m}_{\upbeta,\bA}[\mu](x_i) 
\]
and $G_i(t,\bm{\theta})\coloneqq G_i(X(t),\bm{\theta})$.
Because \eqref{eq:Diffusive_gaussian_case} is a manifold It\^o SDE with diffusion
vector fields tangent to $\S^{d-1}$, each coordinate process admits the following ambient
$\R^d$ representation:
\begin{equation} \label{eq:ambient_form_clean}
\de x_i(t)
= \int_\Uptheta G_i(t,\bm{\theta})\,\mathsf{W}(\de\bm{\theta},\de t)
-
\frac{1}{2}\left(\int_\Uptheta \|G_i(t,\bm{\theta})\|^2\,{\uprho_*}(\de\bm{\theta})\right)x_i(t) \de t.
\end{equation}
Indeed, applying It\^o to $\|x_i(t)\|^2$ in $\R^d$ gives
\[
\de \|x_i(t)\|^2
=
2\langle x_i(t),\de x_i(t)\rangle + \de\langle x_i\rangle_t,
\]
and since $\langle x_i,G_i\rangle=0$, the stochastic part contributes no drift,
while $\de\langle x_i\rangle_t=\int_\Uptheta \|G_i\|^2{\uprho_*}(\de\bm{\theta})\,\de t$.
Choosing the drift in \eqref{eq:ambient_form_clean} cancels the quadratic variation and
ensures $\|x_i(t)\|\equiv 1$.

Fix $i,j\in[n]^2$ and define
\[
R_{ij}(t)\coloneqq \langle x_i(t),x_j(t)\rangle.
\]
Applying It\^o in $\R^d$ to $R_{ij}(t)$ using \eqref{eq:ambient_form_clean} yields
\begin{equation}\label{eq:Rij_decomp_clean}
\de R_{ij}(t)
=
\de M_{ij}(t)
+\cD_{ij}(t)\,\de t,
\end{equation}
where the local martingale part is
\begin{equation}\label{eq:martingale_Mij_clean}
\de M_{ij}(t)
=
\int_\Uptheta
\left(\langle G_i(t,\bm{\theta}),x_j(t)\rangle+\langle x_i(t),G_j(t,\bm{\theta})\rangle\right)
\,\mathsf{W}(\de\bm{\theta},\de t),
\end{equation}
and the drift is
\begin{equation}\label{eq:drift_Dij_clean}
\cD_{ij}(t)
=
\int_\Uptheta \langle G_i(t,\bm{\theta}),G_j(t,\bm{\theta})\rangle\,{\uprho_*}(\de\bm{\theta})
-
\frac{R_{ij}(t)}{2}\int_\Uptheta\left(\|G_i(t,\bm{\theta})\|^2+\|G_j(t,\bm{\theta})\|^2\right){\uprho_*}(\de\bm{\theta}).
\end{equation}
This is the standard identity
\[
\de\langle x_i,x_j\rangle
=
\langle \de x_i,x_j\rangle+\langle x_i,\de x_j\rangle+\de\langle x_i,x_j\rangle,
\]
with the additional linear drifts coming from the sphere-preserving correction in
\eqref{eq:ambient_form_clean}.
For fixed $u,v\in\R^d$, the Gaussian identity
\begin{equation*}\label{eq:gaussian_V_identity_clean}
\E[\bV^{\sT}M\bV]=\frac1d\,\Tr(M)\,I_d
\end{equation*}
for all $M\in\R^{d\times d}$ implies
\begin{align*}
\E\left\langle \proj_{x_i}\bV u,\proj_{x_j}\bV v\right\rangle
&=
\frac1d\,\Tr\left(\proj_{x_i}\proj_{x_j}\right)\,\langle u,v\rangle\\
\E\left\|\proj_{x_i}\bV u\right\|^2
&=
\frac1d\,\Tr\left(\proj_{x_i}\right)\,\|u\|^2.
\end{align*}
Since $\Tr(\proj_{x_i})=d-1$ and
\begin{equation}\label{eq:trace_proj_proj_clean}
\Tr\left(\proj_{x_i}\proj_{x_j}\right)
=\Tr\left((I-x_ix_i^{\sT})(I-x_jx_j^{\sT})\right)
=d-2+\langle x_i,x_j\rangle^2
=d-2+R_{ij}^2,
\end{equation}
we introduce the scalar quantities
\begin{align*}
s_\mu(x)&\coloneqq \E\|\mathsf{m}_{\upbeta,\bA}[\mu](x)\|^2,\nonumber\\
s_\mu(x,x')&\coloneqq \E\left\langle \mathsf{m}_{\upbeta,\bA}[\mu](x),\mathsf{m}_{\upbeta,\bA}[\mu](x')\right\rangle.
\end{align*}
(They satisfy $0\le s_\mu(\cdot)\le 1$ and $|s_\mu(\cdot,\cdot)|\le 1$ since
$\mathsf{m}_{\upbeta,\bA}[\mu](x)$ is a convex combination of unit vectors.) Conditioning on $X(t)$ and using \eqref{eq:drift_Dij_clean}--\eqref{eq:trace_proj_proj_clean},
we obtain the explicit drift formula
\begin{align}\label{eq:Dij_explicit_clean}
\cD_{ij}(t)
&=
\frac1d\left(d-2+R_{ij}(t)^2\right)\,s_{\mu(t)}(x_i(t),x_j(t))\\
&-
\frac{(d-1)}{2d}\,R_{ij}(t)\left(s_{\mu(t)}(x_i(t))+s_{\mu(t)}(x_j(t))\right).\nonumber\qedhere
\end{align}
\end{proof}

\subsection{Proof of Theorem \ref{thm:clustering_random_init}} \label{proof: simplex}

We briefly overload the notation for the drift $\mathscr{D}$ and write $\mathscr{D}_{ij}(t) = \mathscr{D}_{ij}(X(t))$. 
We begin with the following lemma.

\begin{lemma}\label{lem:drift_on_simplex_selfcontained}
For every configuration $X\in(\S^{d-1})^n$ with Gram matrix
$R\in\R^{n\times n}$, the quantities
$s_{\mu_X}(x_i)$ and $s_{\mu_X}(x_i,x_j)$, and hence $\mathscr D_{ij}(X)$ defined in
\eqref{eq:Dij_explicit_clean}, depend on $X$ only through $R$.
Equivalently, there exists a well-defined map $\mathscr D_{ij}(\cdot)$ on the set of realizable
Gram matrices such that $\mathscr D_{ij}(X)=\mathscr D_{ij}(R)$.

Furthermore, if $R(t)=S(\gamma)$ for some
$\gamma\in(-1/(n-1),1)$ and $t\geq0$, then for every $i\neq j$,
\(
\mathscr D_{ij}(R(t))=b(\gamma),
\)
where $b(\gamma)$ is the function defined in Theorem \ref{thm:clustering_random_init}.
\end{lemma}

We can thus further overload notation and write $\mathscr{D}_{ij}(X(t)) = \mathscr{D}_{ij}(R(t))$.

\begin{proof}[Proof of Lemma \ref{lem:drift_on_simplex_selfcontained}]
 Write 
 $$m_i\coloneqq \mathsf{m}_{\upbeta,\bA}[\mu_{X(t)}](x_i(t))=\sum_{k=1}^n \pi_{i\to k}^{\bA}(\gamma)\,x_k(t).$$ 
 Let $X$ and $\Tilde{X}$ be two configurations of particles having the same Gram matrix. There exists an orthogonal $d\times d$ matrix $O$ such that $x_i=O\Tilde{x}_i$ for all $i\in [n]$. We then have 
 $$\left\<\bA x_i,x_j\right\>=\<O^{\sT}\bA O \Tilde{x}_i,\Tilde{x}_j\>,
 $$
and by orthogonal invariance of the Gaussian law,  $\<\bA x_i,x_j\>$ and $\<\bA \Tilde{x}_i,\Tilde{x}_j\>$ have the same law. 
We deduce that $\E[\pi_{i\to k}^\bA(\gamma)\pi_{i\to \ell}^\bA(\gamma)]$ is a function of $R(t)$ only. 
This implies that $\E\<m_i,m_j\>$ and $\E\|m_i\|^2$ are functions of $R(t)$ only as well. So, the dynamics depends on $R(t)$ only.

We now explicitly compute the drift. We have
\begin{align*}
\|m_i\|^2&=\sum_{k,\ell}\pi_{i\to k}^\bA(\gamma)\pi_{i\to \ell}^\bA(\gamma)\langle x_k(t),x_\ell(t)\rangle\\
&=
\gamma\left(\sum_k\pi_{i\to k}^\bA(\gamma)\right)^2
+(1-\gamma)\sum_{k}\left(\pi_{i\to k}^\bA(\gamma)\right)^2\\
&=
\gamma+(1-\gamma)\sum_k\left(\pi_{i\to k}^\bA(\gamma)\right)^2,
\end{align*}
and similarly, for $i\ne j$,
\begin{align*}
\langle m_i,m_j\rangle
&=\sum_{k,\ell}\pi_{i\to k}^\bA(\gamma)\pi_{j\to \ell}^\bA(\gamma)\langle x_k(t),x_\ell(t)\rangle\\
&=\gamma+(1-\gamma)\sum_{k}\pi_{i\to k}^\bA(\gamma)\pi_{j\to k}^\bA(\gamma).
\end{align*}
Taking expectations and using \eqref{eq:f_g_def_selfcontained}, we obtain $s_{\mu_{X(t)}}(x_i(t))=\gamma+(1-\gamma)f(\gamma)$ as well as $s_{\mu_{X(t)}}(x_i(t), x_j(t))=\gamma+(1-\gamma)g(\gamma).$
Plugging these expressions into the drift formula
\eqref{eq:Dij_explicit_clean} yields exactly  \eqref{eq:Phi_def_selfcontained}.
\end{proof}

To close the dynamics, we show that the drift is Lipschitz near the simplex line.

\begin{lemma} \label{lem:Gram_stability}
Let $X\in(\S^{d-1})^{n}$ and let $R\in\R^{n\times n}$ be its Gram matrix.
Set
\(
R(\gamma)\coloneqq (1-\gamma)I_n+\gamma\,\mathbf 1\mathbf 1^{\sT}.
\)
Then there exists $L>0$ such that for all $\gamma\in(-1/(n-1),1)$,
\begin{equation*}
\max_{1\leq i\neq j\leq n}|\cD_{ij}(R)-b(\gamma)|
\leq L\max_{(i,j)\in[n]}|R_{ij}-R(\gamma)_{ij}|.
\end{equation*}
(Note that $\mathscr{D}_{ij}(R(\gamma))= b(\gamma)$.)
Moreover one can take $L=O(1+d\sigma_{\bA}^{4}\upbeta^{2}n)$.
\end{lemma}

We henceforth write $\|A\|_{\infty} = \max_{(i,j)\in[n]^{2}}|A_{ij}|$ for $A\in\R^{n\times n}$.

\begin{proof}[Proof of Lemma \ref{lem:Gram_stability}]
Since $\bA=\bW'^{\sT}\bW$, we have that $\bW X$ and $\bW' X$ are Gaussian variables. 
Denote $h_i=\bW x_i$, and $g_j=\bW' x_j$. We have
\[
\Cov(h_i,h_j)=\sigma_{\bA}^{2}R_{ij}I_d, \quad \Cov(g_i,g_j)=\sigma_{\bA}^{2}R_{ij}I_d,\quad \Cov(h_i,g_j)=0.
\]
Using tensor product notations, we can write $h\sim \cN(0,\sigma_{\bA}^{2}R\otimes I_d)$, and $(h,g)\sim \cN(0,\Sigma(R)),$ where 
\begin{equation}\label{eq:SIGMA_R}
\Sigma(R)=
\begin{bmatrix}
\sigma_{\bA}^{2}R\otimes I_d &0  \\
0 &\sigma_{\bA}^{2}R\otimes I_d 
\end{bmatrix}.
\end{equation}
Since $R$ is a Gram matrix, it is symmetric positive semi-definite with $\diag(R)=1$ (may be singular when $n>d$). Therefore $\Sigma(R)$ defined in \eqref{eq:SIGMA_R} is positive semi-definite and $(h,g)\sim \cN(0,\Sigma(R))$ is well-defined but possibly degenerate\footnote{If this were to arrive, one slight caveat is that Lemma \ref{lem:gaussian_interp} doesn't immediately apply. But one can use a standard trick and replace by $\Sigma_\varepsilon(R)=\Sigma(R)+\epsilon I_{2nd}$, then apply the lemma for $\Sigma_\varepsilon(R)$, and send $\epsilon\to0$ by using dominated convergence.}. 
Note that $\<\bA x_i,x_j\>=h_i^{\sT}g_j.$ In particular, we can rewrite $$\pi^{\bA}_{i\to j}(X)=\frac{e^{\upbeta h_i^{\sT}g_j}}{\sum_{k=1}^{n} e^{\upbeta h_i^{\sT}g_k}}.$$
Recall $s_{\mu_X}(x_i,x_j)=\E[\<\sum_{k}\pi_{i\to k}^{\bA}x_k,\sum_{\ell}\pi_{j\to \ell}^{\bA}x_{\ell}\>]$, hence
\begin{equation*}
    s_{\mu_X}(x_i,x_j)=\sum_{k,\ell=1}^{n}\E \left[\frac{e^{\upbeta h_i^{\sT}g_k}}{\sum_{k=1}^{n} e^{\upbeta h_i^{\sT}g_k}}\, \frac{e^{\upbeta h_j^{\sT}g_{\ell}}}{\sum_{k=1}^{n} e^{\upbeta h_j^{\sT}g_k}}\right]\<x_k,x_{\ell}\>,
\end{equation*}
and also 
\begin{equation*}
    s_{\mu_X}(x_i)=\sum_{k,\ell=1}^{n}\E\left[\frac{e^{\upbeta h_i^{\sT}g_k}}{\sum_{k=1}^{n} e^{\upbeta h_i^{\sT}g_k}}\, \frac{e^{\upbeta h_i^{\sT}g_{\ell}}}{\sum_{k=1}^{n} e^{\upbeta h_i^{\sT}g_k}}\right]\<x_k,x_{\ell}\>.
\end{equation*}
We can study all these quantities as functions of $(R,h,g).$ Since clearly
\begin{equation*}
    |\<x_k,x_{\ell}\>-\<x_{k}',x_{\ell}'\>|\leq \|R-R'\|_{\infty},
\end{equation*} 
we focus on the softmax weights.

Let $F(R)=\E f(h,g)$.
The idea in what follows is to leverage the Gaussianity of $h,g$ in order to compute the derivative of $F$ with respect to $R$. We recall the celebrated \emph{Gaussian interpolation} lemma.

\begin{lemma}[Lemma 7.2.7 from \cite{vershynin2018high}]\label{lem:gaussian_interp}
Consider two independent Gaussian random vectors $X\sim\cN(0,\Sigma^{X})$ and $Y\sim \cN(0,\Sigma^{Y})$. Define the interpolated Gaussian vector 
\begin{equation*}
    X_t=\sqrt{t}X+\sqrt{1-t}Y.
\end{equation*}
Then for all $\varphi\in C^2(\R^m)$, we have
\begin{equation*}
    \frac{\de}{\de t}\E \varphi(X_t)=\frac{1}{2}\E\left[\Tr\left((\Sigma^{X}-\Sigma^{Y})^{\sT}\; \nabla^{2}\varphi(X_t)\right)\right].
\end{equation*}
\end{lemma}

Consider $(H_0,G_0)$ centered Gaussian random variables on $\reals^{2nd}$ with covariance $\Sigma(R(\gamma))$ and $(H_1,G_1)$ independent Gaussian variables with covariance $\Sigma(R)$.
Define $X_t$ as the interpolation between $(H_0,G_0)$ and $(H_1,G_1)$: 
\begin{equation}\label{eq:interpolation_gaussian}
X_t=\sqrt{t}(H_1,G_1)+\sqrt{1-t}(H_0,G_0).    
\end{equation}
Let $\psi(t)\coloneqq\E f(X_t)$. Applying Lemma \ref{lem:gaussian_interp} to \eqref{eq:interpolation_gaussian} yields
\begin{equation}\label{eq:gaussian_interp_0}
    \frac{\de}{\de t}\psi(t)=\frac{1}{2}\E\left[\Tr\left(\, \Sigma\left(R-R(\gamma)\right)^{\sT} \nabla^{2}f(X_t)\,\right)\right].
\end{equation}
Using the tensor product structure, we simplify \eqref{eq:gaussian_interp_0} in  
\begin{align}\label{eq:steins_lemma}
\frac{\de}{\de t}\psi(t)=\frac{\sigma_{\bA}^{2}}{2}\Big(&\E\left[\Tr\left(\nabla^{2}_g f(X_t) \, (\varepsilon(R) \otimes I_d)\right)\right]\nonumber\\
+&\E\left[\Tr\left(\nabla^{2}_h f(X_t) \, (\varepsilon(R) \otimes I_d)\right)\right]\Big),
\end{align}
where $\varepsilon(R)\coloneqq R-R(\gamma)$, and $(\nabla^{2}_h f,\nabla^{2}_{g} f)$ are the $(nd)\times (nd)$ Hessians with respect to the vectors $h,g \in \reals^{nd}$ respectively, with coefficients  
\begin{align*}
    (\nabla^{2}_h f)_{(i,k),(j,q)}&=\partial_{(h_i)_k,(h_{j})_q} f(h,g)\\
    (\nabla^{2}_g f)_{(i,k),(j,q)}&=\partial_{(g_i)_k,(g_{j})_q} f(h,g).
\end{align*}
We specialize \eqref{eq:steins_lemma} to $$
f(h,g)=\frac{e^{\upbeta h_i^{\sT}g_k}}{\sum_{k'=1}^{n} e^{\upbeta h_i^{\sT}g_{k'}}}\, \frac{ e^{\upbeta h_j^{\sT}g_\ell}}{\sum_{\ell'=1}^{n} e^{\upbeta h_j^{\sT}g_{\ell'}}}.
$$
Since $f$ is $C^{\infty}$ and its derivatives are bounded by polynomials times bounded softmax factors, Gaussian moments ensure that the differentiation under the expectation is justified. It yields the following bound
\begin{align}\label{eq:interpolation_interm}
\frac{\de}{\de t}\psi(t)&=\frac{\sigma_{\bA}^{2}}{2}\sum_{a,b=1}^{n}(\varepsilon(R))_{ab}\;\E\;\Tr \Big((\nabla_{h}^2f(X_t))_{(a,\cdot),(b,\cdot)}\Big)\nonumber\\
&+\frac{\sigma_{\bA}^{2}}{2}\sum_{a,b=1}^{n}(\varepsilon(R))_{ab}\;\E\;\Tr \Big((\nabla_{g}^2f(X_t))_{(a,\cdot),(b,\cdot)}\Big).
\end{align}
The computations are similar as the ones in the proof of Lemma \ref{lem:Kernel_regularity} (see {Section  \ref{proof:lem_regularity}}). For $i\in [n]$ let $u_i\coloneqq \upbeta G^{\sT}h_i$ where $G\coloneqq(g_1,\ldots,g_n)\in \reals^{d \times n}$. Let $\chi\in \reals^{n}\rightarrow \reals^{n}$ denote the softmax function, so $\pi_{i\to k}^{\bA}(u)=\chi_{k}(u_i).$ 
Direct computations show 
\begin{align*}
\nabla^{2}_{h_i,h_i}f(h,g)&=\upbeta^{2}\chi_{\ell}(u_j)\sum_{r,m=1}^{n}\chi_{k,rm}''(u_i)g_rg_{m}^{\sT},\\
\nabla^{2}_{h_j,h_j}f(h,g)&=\upbeta^{2}\chi_{k}(u_i)\sum_{r,m=1}^{n}\chi_{\ell,rm}''(u_{j})g_rg_{m}^{\sT},\\
\nabla^{2}_{h_i,h_j}f(h,g)&=\upbeta^{2}\sum_{r,m=1}^{n}\chi'_{\ell,r}(u_j)\chi'_{k,m}(u_i)g_rg_m^{\sT},\\
\nabla^{2}_{h_a,h_b}f(h,g)&=0,\qquad a\notin\{i,j\}\ \text{or}\ b\notin\{i,j\}.
\end{align*}
Taking the trace yields
\begin{align*}
\Tr(\nabla^{2}_{h_i,h_i}f(h,g))&=\upbeta^{2}\chi_{\ell}(u_j)\sum_{r,m=1}^{n}\chi_{k,rm}''(u_i)\<g_r,g_{m}\>,\\
\Tr(\nabla^{2}_{h_j,h_j}f(h,g))&=\upbeta^{2}\chi_{k}(u_i)\sum_{r,m=1}^{n}\chi_{\ell,rm}''(u_{j})\<g_r,g_{m}\>,\\
\Tr(\nabla^{2}_{h_i,h_j}f(h,g))&=\upbeta^{2}\sum_{r,m=1}^{n}\chi'_{\ell,r}(u_j)\chi'_{k,m}(u_i)\<g_r,g_m\>,\\
\Tr(\nabla^{2}_{h_a,h_b}f(h,g))&=0,\qquad a\notin\{i,j\}\ \text{or}\ b\notin\{i,j\}.
\end{align*}
We deduce that the first term in the right-hand side of \eqref{eq:interpolation_interm} is bounded by
\begin{align*}
    &(\varepsilon(R))_{ii}\upbeta^{2}\chi_{\ell}(u_j)\sum_{r,m=1}^{n}\chi_{k,rm}''(u_i)\<g_r,g_{m}\>\\
    &+(\varepsilon(R))_{jj}\upbeta^{2}\chi_{k}(u_i)\sum_{r,m=1}^{n}\chi_{\ell,rm}''(u_j)\<g_r,g_{m}\>\\
    &+(\varepsilon(R)_{ij}+\varepsilon(R)_{ji})\upbeta^{2}\sum_{r,m=1}^{n}\chi'_{\ell,r}(u_j)\chi'_{k,m}(u_i)\<g_r,g_m\>.
\end{align*}
Using $\chi'_{i,r}(u)=\chi_{i}(u)\bigl(1_{r=i}-\chi_{r}(u)\bigr)$ and 
\begin{align*}
   \chi''_{i,rm}(u)&=\chi_i(u)\Big((1_{r=i}-\chi_r(u))(1_{m=i}-\chi_m(u))-\chi_m(u)(1_{r=m}-\chi_r(u))\Big),
\end{align*}
Cauchy-Schwarz, and the fact that $\chi_k(u)\geq 0,$ and $\sum_{k=1}^{n}\chi_k(u)=1$, we find
\begin{equation}\label{eq:estimate_bound_h}
    \sum_{a,b=1}^{n}(\varepsilon(R))_{ab}\E\Tr \Big((\nabla_{h}^2f(X_t))_{(a,\cdot),(b,\cdot)}\Big)\leq C\upbeta^{2}\|\varepsilon(R)\|_{\infty}\chi_{k}(u_i)\chi_{\ell}(u_j)\sum_{m=1}^{n}\E\|g_m\|^{2}.
\end{equation}
The same computations for the second term in $g$ yields that the trace is given by
\begin{align*}
\Tr(\nabla^{2}_{g_a,g_b}f)&=\upbeta^{2}\chi_{\ell}(u_j)\chi''_{k,ab}(u_i)\|h_i\|^{2}+\upbeta^{2}\chi_{k}(u_i)\chi''_{\ell,ab}(u_j)\|h_j\|^{2}\\
    &+\upbeta^{2}\left(\chi'_{k,a}(u_i)\chi'_{\ell,b}(u_{j})+\chi'_{k,b}(u_i)\chi'_{\ell,a}(u_{j})\right)\<h_j, h_i\>.
\end{align*}
Using the above estimates on the softmax derivatives, and summing as before, we obtain 
\begin{align} \label{eq:estimate_bound_g}
&\sum_{a,b=1}^{n}(\varepsilon(R))_{ab}\,\E\,\Tr \left(\left(\nabla_{g}^2f(X_t)\right)_{(a,\cdot),(b,\cdot)}\right)\nonumber\\
&\hspace{2cm}\leq C \upbeta^{2}\|\varepsilon(R)\|_{\infty}\chi_{k}(u_i)\chi_{\ell}(u_j)\left(\E\|h_i\|^{2}+\E\|h_j\|^{2}\right).
\end{align}
Plugging \eqref{eq:estimate_bound_h} and \eqref{eq:estimate_bound_g} into \eqref{eq:interpolation_interm}, we obtain 
\begin{equation*}
\left|\frac{\de}{\de t}\psi(t)\right|\leq C \sigma_{\bA}^{2}\upbeta^{2}\chi_{k}(u_i)\chi_{\ell}(u_j)\|\varepsilon(R)\|_{\infty}\left(\E\|h_i\|^{2}+\E\|h_j\|^{2}+\sum_{m=1}^{n}\E\|g_m\|^{2}\right).
\end{equation*}
Recall that $\psi(t)=\E f(X_t)$ and that $\psi(1)=F(R)$ and $\psi(0)=F(R(\gamma))$. Using the expression of the moment of a Gaussian variable and integrating \eqref{eq:steins_lemma}, we find
\begin{equation}\label{eq:Lispchitz_bound_softmax}
|F(R)-F(R(\gamma))|\leq C\sigma_{\bA}^{4}d\upbeta^{2}n \chi_{k}(u_i)\chi_{\ell}(u_j)\|\varepsilon(R)\|_{\infty}.
\end{equation}
Denote $g(R)\coloneqq s_{\mu_X}(x_i,x_j)$. Using the estimate \eqref{eq:Lispchitz_bound_softmax}, we obtain
\begin{align*}
    |g(R+\varepsilon(R))-g(R)|&\leq \sum_{k,\ell}|\varepsilon(R)_{k\ell}||F_{k,\ell}(R)|\\
    &+\sum_{k,\ell}\;|(R+\varepsilon(R))_{k\ell}| \;|F_{k,\ell}(R)-F_{k,\ell}(R+\varepsilon(R))|\\
    &\leq  \left(1+C\sigma_{\bA}^{4}d\upbeta^{2}n\sum_{k,\ell} \chi_{k}(u_i)\chi_{\ell}(u_j)\right)\|\varepsilon(R)\|_{\infty}\\
    &\leq \left(1+C\sigma_{\bA}^{4}d\upbeta^{2}n\right)\|\varepsilon(R)\|_{\infty}.
\end{align*}
Since the drift is a function of $s_{\mu_X}(x_i),$ $s_{\mu_X}(x_i,x_j)$ and of polynomials of $R$, we can conclude. 
\end{proof}

\begin{proof}[Proof of Theorem~\ref{thm:clustering_random_init}]  By Lemma \ref{lem:Ito_formula}, for $i\ne j$,
\[
R_{ij}(t)
=
R_{ij}(0)+M_{ij}(t)+\int_0^t \mathscr D_{ij}(R(s))\,\de s.
\]
Subtract the ODE \eqref{eq:gamma_ODE_selfcontained}:
\[
R_{ij}(t)-\gamma(t)
=
R_{ij}(0)-\gamma_0+M_{ij}(t)
+\int_0^t (\cD_{ij}(R(s))-b(\gamma(s)))\de s.
\]
Since the initial overlaps satisfy $R_{ij}(0)=\gamma_0$, the initial error vanishes.
By Lemma \ref{lem:drift_on_simplex_selfcontained} and
Lemma \ref{lem:Gram_stability}, we obtain
\[
\max_{i\ne j}|\mathscr D_{ij}(R(s))-b(\gamma(s))|
\le
L\max_{i\ne j}|R_{ij}(s)-\gamma(s)|.
\]
Define 
\(
E(t)\coloneqq \max_{i\ne j}|R_{ij}(t)-\gamma(t)|.
\)
Then for all $t\in[0,T]$,
\[
E(t)
\le
\max_{i\ne j}\sup_{s\in[0,t]}|M_{ij}(s)|
+
\alpha L\int_0^t E(s)\,\de s.
\]
By Gr\"onwall's inequality,
\begin{equation}\label{eq:Gronwall_E_selfcontained}
E(t)
\le e^{L t}
\max_{i\ne j}\sup_{s\in[0,t]}|M_{ij}(s)|.
\end{equation}
From \eqref{eq:martingale_Mij_clean} and the It\^o isometry for cylindrical Wiener integrands \cite{da2014stochastic},
the quadratic variation satisfies
\[
\de\langle M_{ij}\rangle_t
= \int_\Uptheta
\left(\langle G_i(t,\bm{\theta}),x_j(t)\rangle+\langle x_i(t),G_j(t,\bm{\theta})\rangle\right)^2
{\uprho_*}(\de\bm{\theta})\,\de t.
\]
Using $\E[\<\bV u,a\>\<\bV v,b\>]=\<u,v\>\<a,b\>/d$, we compute all the terms:
\begin{equation*}
    \E\left\<\bV \mathsf{m}_{\upbeta,\bA}[\mu](x_i),\proj_{x_i(t)}x_j(t)\right\>^{2}=\frac1d\left\|\mathsf{m}_{\upbeta,\bA}[\mu](x_i)\right\|^{2}\left\|\proj_{x_i(t)}x_j(t)\right\|^{2},
\end{equation*}
and 
\begin{align*}
&\E\left[\left\<\bV \mathsf{m}_{\upbeta,\bA}[\mu](x_i),\proj_{x_i(t)}x_j(t)\right\>\left\<\bV \mathsf{m}_{\upbeta,\bA}[\mu](x_j),\proj_{x_j(t)}x_i(t)\right\>\right]\\
    &=\frac1d\left\<\mathsf{m}_{\upbeta,\bA}[\mu](x_i),\mathsf{m}_{\upbeta,\bA}[\mu](x_j)\right\> \left<\proj_{x_i(t)}x_j(t),\proj_{x_j(t)}x_i(t)\right\>.
\end{align*}
Then,
\begin{align*}
\de\langle M_{ij}\rangle_t
&=\frac{2}{d}\left\<\mathsf{m}_{\upbeta,\bA}[\mu](x_i),\mathsf{m}_{\upbeta,\bA}[\mu](x_j)\right\> \left\<\proj_{x_i(t)}x_j(t),\proj_{x_j(t)}x_i(t)\right\> \\
&+\frac1d\left(\|\mathsf{m}_{\upbeta,\bA}[\mu](x_i)\|^{2}\left\|\proj_{x_i(t)}x_j(t)\right\|^{2}+\|\mathsf{m}_{\upbeta,\bA}[\mu](x_j)\|^{2}\left\|\proj_{x_j(t)}x_i(t)\right\|^{2} \right)
\end{align*}
Bounding the above expression, we find
\[
\<M_{ij}\>_T
=
\frac{4T}{d}.
\]
By a standard maximal inequality for continuous martingales \cite{dzhaparidze2001bernstein},
 for each fixed pair $(i,j)$,
\[
\P\left(\sup_{t\in[0,T]}|M_{ij}(t)|\ge r\right)
\le 2\exp\left(-\frac{d r^2}{8T}\right).
\]
Applying a union bound over $(i,j)\in[n]^2$ and taking
$r=C\sqrt{4T\log(2n^2/\delta)/d}$ gives, with probability at least $1-\delta$,
\[
\sup_{t\in[0,T]}|M_{ij}(t)|
\le
C\sqrt{\frac{4T}{d}\log\left(\frac{2n^2}{\delta}\right)}
\qquad
\text{for all }i,j.
\]
Since $\sup|\widetilde M|\le \sup_{i,j}\sup|M_{ij}|$, we deduce
\begin{equation}\label{eq:martingale_sup_bound_clean}
\sup_{t\in[0,T]}|\widetilde M(t)|
\le
C\sqrt{\frac{4T}{d}\log\left(\frac{2n^2}{\delta}\right)}
\end{equation}
with probability at least $1-\delta$. We insert \eqref{eq:martingale_sup_bound_clean} into \eqref{eq:Gronwall_E_selfcontained}
(with $t=T$) to conclude.
\end{proof}

\subsection{Proof of Theorem \ref{thm:clustering_small_beta}} \label{proof: small.beta}

\begin{proof}[Proof of Theorem \ref{thm:clustering_small_beta}] 
By definition,
\[
m(t)=\frac{1}{n^2}\sum_{i,j}R_{ij}(t),
\]
so summing \eqref{eq:Rij_decomp_clean} over $i,j$ and applying Lemma \ref{lem:Ito_formula} gives
\begin{equation}\label{eq:m_decomp_clean}
m(t)=m(0)+\int_0^t g(X(s))\,\de s+\widetilde M(t),
\end{equation}
where
\begin{equation}\label{eq:F_def_clean}
g
\coloneqq
\frac{\alpha}{n^2}\sum_{i,j}\cD_{ij},
\end{equation}
with $\cD_{ij}$ given by \eqref{eq:Dij_explicit_clean}, and
\[
\widetilde M(t)\coloneqq \frac{1}{n^2}\sum_{i,j} M_{ij}(t)
\]
is a continuous local martingale.
Set $f(u)\coloneqq r(1-u)$ for $u\in[0, 1]$. We show that
\begin{equation}\label{eq:drift_mismatch_target_clean}
|g(X)-f(m(X))|
\le
C\left(\frac{1}{d}+\upbeta^2\,\E\|\bA\|^2\right).
\end{equation}
This is done in two parts.

\smallskip
\noindent\emph{(a).}
For $\upbeta=0$, we have $\mathsf m_{0,\bA}[\mu](x)=\int y\,\mu(\de y)=:\mathsf{m}[\mu]$ is
independent of $x$ and $\bA$, hence
\begin{align*}
    s_{\mu}(x)&=\|\mathsf m[\mu]\|^2=m(X),\\
s_{\mu}(x_i,x_j)&=\|\mathsf m[\mu]\|^2=m(X).
\end{align*}
Plugging this into \eqref{eq:Dij_explicit_clean} and then into \eqref{eq:F_def_clean}, one finds
\begin{align*}
g_0(X)&=\frac1d m(X)\left((d-2)+q(X)-(d-1)m(X)\right),\\
q(X)&\coloneqq \frac{1}{n^2}\sum_{i,j}R_{ij}^2\in[0,1].
\end{align*}
Therefore
\[
g_0(X)-f(m(X))
=
\frac1d m(X)\left(q(X)-2+m(X)\right),
\]
and since $m(X)\in[0,1]$ and $q(X)\in[0,1]$,
\[
|g_0(X)-f(m(X))|
\le\frac{2}{d}.
\]
\noindent\emph{(b).}
We claim that for $\upbeta\in[0,1]$, uniformly over all empirical measures $\mu$ on $\S^{d-1}$
and $x,x'\in\S^{d-1}$,
\begin{equation}\label{eq:s_mu_small_beta_clean}
\big|s_\mu(x)-\|\mathsf m[\mu]\|^2\big|
+
\big|s_\mu(x,x')-\|\mathsf m[\mu]\|^2\big|
\le C\,\upbeta^2\,\E\|\bA\|^2.
\end{equation}

\noindent
\emph{Proof of \eqref{eq:s_mu_small_beta_clean}.}
Write $\mu=\frac1n\sum_{k=1}^n\delta_{y_k}$.
For fixed $\bA$ and $x$, define 
\[
w_k(\upbeta)\coloneqq
\frac{e^{\upbeta\langle \bA x,y_k\rangle}}{\sum_{\ell=1}^n e^{\upbeta\langle \bA x,y_\ell\rangle}}.
\]
A direct differentiation gives
\[
w_k'(\upbeta)=w_k(\upbeta)\left(\langle \bA x,y_k\rangle-\sum_{\ell=1}^n w_\ell(\upbeta)\langle \bA x,y_\ell\rangle\right),
\]
hence $|w_k'(\upbeta)|\le 2\|\bA\|\,w_k(\upbeta)$, and therefore
\[
\left\|\partial_\upbeta \mathsf{m}_{\upbeta,\bA}[\mu](x)\right\|
=
\left\|\sum_{k=1}^n w_k'(\upbeta)\,y_k\right\|
\le \sum_{k=1}^n |w_k'(\upbeta)|
\le 2\|\bA\|.
\]
Differentiating once more and using $|\langle \bA x,y\rangle|\le \|\bA\|$ and $\sum_k w_k=1$
shows that 
$$
\left\|\partial_\upbeta^2 \mathsf{m}_{\upbeta,\bA}[\mu](x)\right\|\le C\|\bA\|^2$$ 
for $\upbeta\in[0,1]$.
Thus Taylor's theorem with remainder yields
\[
\mathsf{m}_{\upbeta,\bA}[\mu](x)=\mathsf m_{0,\bA}[\mu](x)+\upbeta\,\partial_\upbeta \mathsf m_{0,\bA}[\mu](x)+R_{\upbeta,\bA}(x),
\]
and 
\[
\|R_{\upbeta,\bA}(x)\|\le C\upbeta^2\|\bA\|^2,
\]
and similarly for $x'$.
Since $\partial_\upbeta m_{0,\bA}[\mu](x)$ is linear in $\bA$ (equals $\Sigma[\mu]\bA^{\sT}x$ with $\Sigma[\mu]$ the empirical covariance), it has mean zero under
$\bA\mapsto -\bA$ symmetry, hence $$\E \partial_\upbeta \mathsf m_{0,\bA}[\mu](x)=0.$$
Expanding $\|\mathsf{m}_{\upbeta,\bA}[\mu](x)\|^2$ and
$\langle \mathsf{m}_{\upbeta,\bA}[\mu](x),\mathsf{m}_{\upbeta,\bA}[\mu](x')\rangle$, taking expectations, and using
$\|\mathsf m_{0,\bA}[\mu](x)\|=\|\mathsf m[\mu]\|\le 1$ gives \eqref{eq:s_mu_small_beta_clean}.
\hfill$\square$

\smallskip
Now combine \eqref{eq:s_mu_small_beta_clean} with \eqref{eq:Dij_explicit_clean}: since
$|R_{ij}|\le 1$ and $d-2+R_{ij}^2\le d-1\le d$, we obtain
\[
\left|\cD_{ij}(t)-\cD_{ij}^{(0)}(t)\right|
\le
C\upbeta^2\,\E\|\bA\|^2,
\]
where $\cD_{ij}^{(0)}$ denotes the drift with $\upbeta=0$ (hence $s_\mu(\cdot)=\|\mathsf m[\mu]\|^2$).
Averaging over $i,j$ as in \eqref{eq:F_def_clean} gives
\[
|g(X)-g_0(X)|
\le
C\upbeta^2 \E\|\bA\|^2.
\]
Together with the $\upbeta=0$ bound in part (a), this proves
\eqref{eq:drift_mismatch_target_clean}.

Let $u\in C^1(\R_{\ge 0})$ be the solution to the logistic ODE
\begin{equation}\label{eq:logistic_u_def_clean}
\begin{cases}
\dot u(t)=f(u(t))\\
u(0)=m(0).
\end{cases}
\end{equation}
Let $e(t)\coloneqq m(t)-u(t)$. From \eqref{eq:m_decomp_clean} and \eqref{eq:logistic_u_def_clean}
\[
e(t)=\int_0^t\left(g(X(s))-f(u(s))\right)\,\de s+\widetilde M(t).
\]
Add and subtract $f(m(s))$ to obtain
\[
e(t)
=
\int_0^t\left(f(m(s))-f(u(s))\right)\,\de s
+
\int_0^t\left(g(X(s))-f(m(s))\right)\,\de s
+
\widetilde M(t).
\]
Since $f'(u)=(1-2u)$, we have $|f(u)-f(u')|\le |u-u'|$ on $[0,1]$, hence
\[
|e(t)|
\le \int_0^t |e(s)|\,\de s
+
\int_0^t |g(X(s))-f(m(s))|\,\de s
+
\sup_{s\in[0,t]}|\widetilde M(s)|.
\]
Using \eqref{eq:drift_mismatch_target_clean}, we get
\[
|e(t)|
\le
\int_0^t |e(s)|\,\de s
+
Ct\left(\frac{1}{d}+\upbeta^2\E\|\bA\|^2\right)
+
\sup_{s\in[0,t]}|\widetilde M(s)|.
\]
Gr\"onwall's inequality yields
\begin{equation}\label{eq:gronwall_e_clean}
\sup_{t\in[0,T]}|e(t)|
\le
\mathrm{e}^{T}\,
\left(
C T\left(\frac{1}{d}+\upbeta^2\,\E\|\bA\|^2\right)
+
\sup_{t\in[0,T]}|\widetilde M(t)|
\right).
\end{equation}
From \eqref{eq:martingale_Mij_clean} and the expression of $\Tilde{M}$, we have 
\begin{equation*}
    \<\widetilde{M}\>_t=\frac{1}{n^{4}}\sum_{i,j,k,\ell}\<M_{ij},M_{k\ell}\>_t.
\end{equation*}
Using the It\^o isometry for cylindrical Wiener integrals,
the quadratic variation satisfies
\begin{align*}
\de\langle M_{ij},M_{k\ell}\rangle_t
=\int_\Uptheta
&\left(\langle G_i(t,\bm{\theta}),x_j(t)\rangle+\langle x_i(t),G_j(t,\bm{\theta})\rangle\right)\\
&\left(\langle G_k(t,\bm{\theta}),x_\ell(t)\rangle+\langle x_k(t),G_\ell(t,\bm{\theta})\rangle \right)
{\uprho_*}(\de\bm{\theta})\,\de t.
\end{align*}
Using $\E[\<\bV u,a\>\<\bV v,b\>]=\<u,v\>\<a,b\>/d$, we compute all the terms
\begin{align*}
&\E\left[\left\<\bV \mathsf{m}_{\upbeta,\bA}[\mu](x_i),\proj_{x_i(t)}x_j(t)\>\,\<\bV \mathsf{m}_{\upbeta,\bA}[\mu](x_k),\proj_{x_k(t)}x_{\ell}(t)\right\>\right]\\
=\frac1d&\E\left[\left\<\mathsf{m}_{\upbeta,\bA}[\mu](x_i),\mathsf{m}_{\upbeta,\bA}[\mu](x_k)\right\>\right]\left\<\proj_{x_i(t)}x_j(t),\proj_{x_k}x_{\ell}\right\>.
\end{align*}
Combining all terms, we end up with
\begin{align*}
\de\langle M_{ij},M_{k\ell}\rangle_t
&=\frac1d\E\left[\left\<\mathsf{m}_{\upbeta,\bA}[\mu](x_i),\mathsf{m}_{\upbeta,\bA}[\mu](x_k)\right\>\right]\left\<\proj_{x_i(t)}x_j(t),\proj_{x_k}x_{\ell}\right\>\\
&+\frac1d\E\left[\left\<\mathsf{m}_{\upbeta,\bA}[\mu](x_i),\mathsf{m}_{\upbeta,\bA}[\mu](x_\ell)\right\>\right]\left\<\proj_{x_i(t)}x_j(t),\proj_{x_\ell}x_{k}\right\>\\
&+\frac1d\E\left[\left\<\mathsf{m}_{\upbeta,\bA}[\mu](x_j),\mathsf{m}_{\upbeta,\bA}[\mu](x_k)\right\>\right]\left\<\proj_{x_j(t)}x_i(t),\proj_{x_k}x_{\ell}\right\>\\
&+\frac1d\E\left[\left\<\mathsf{m}_{\upbeta,\bA}[\mu](x_j),\mathsf{m}_{\upbeta,\bA}[\mu](x_\ell)\right\>\right]\left\<\proj_{x_j(t)}x_i(t),\proj_{x_\ell}x_{k}\right\>.
\end{align*}
Using $\|\proj_{x}\|\leq 1$ and $\|\mathsf{m}_{\upbeta,\bA}[\mu](x)\|\leq 1$, we have
\[
\langle M_{ij},M_{k\ell}\>_T
\leq  \frac{4T}{d}.
\]
Summing the above bound, we have
\begin{equation*}
\<\widetilde{M}\>_t\leq \frac{4T}{d}.
\end{equation*}
By a standard maximal inequality for continuous martingales,
\[
\P\left(\sup_{t\in[0,T]}|\widetilde{M}(t)|\ge r\right)
\le 2\exp\left(-\frac{r^2 d}{8T}\right).
\]
Thus, with probability at least $1-\delta$, we have 
\begin{equation*}
\sup_{t\in[0,T]}|\widetilde M(t)|
\le
C\sqrt{\frac{4T}{d}\log\left(\frac{1}{\delta}\right)}.
\end{equation*}
Combining \eqref{eq:gronwall_e_clean} and \eqref{eq:martingale_sup_bound_clean} yields
\eqref{eq:smallbeta_main_bound}.
Finally, for $\bA=\bW\bW'^{\sT}$ we have $\|\bA\|\le \|\bW\|\|\bW'\|$, hence
$$
\E\|\bA\|^2\le \sqrt{\E\|\bW\|^4\E\|\bW'\|^4}\le C d^2\sigma_{\bA}^4.
$$
(See \cite[4.4.2]{vershynin2018high}.)
\end{proof}

\subsection{Proof of Theorem \ref{thm:large_beta_meta}}

The proof relies strongly on the Laplace method which can be applied in the setting where $\mu(t)$ has a density with respect to the Lebesgue measure.

\begin{proof}[Proof of Theorem \ref{thm:large_beta_meta}] \label{proof: mean.field} 

We proceed by a coupling argument. Consider two iid copies $(x^{(1)}(t),x^{(2)}(t))$ driven by 
\begin{equation*}
    \de x(t)=\int_{\Uptheta}G_{\bm{\theta}}[\mu(t)](x(t)) \mathsf{W}(\de \bm{\theta},\de t)-\frac{1}{2}\left(\int_{\Uptheta} \|G_{\bm{\theta}}[\mu(t)](x(t))\|^{2}{\uprho_*}(\de \bm{\theta})\right)x(t)\de t.
\end{equation*}
Set $R(t)\coloneqq\<x^{(1)}(t),x^{(2)}(t)\>$.  
Arguing as in the proof of Lemma \ref{lem:Ito_formula},
\begin{align*}
d\cdot \de R(t)&=-\frac{d-1}{2}
     \left[\E\left\|\mathsf{m}_{\upbeta,\bA}[\mu(t)](x^{1}(t))\right\|^{2}+\E\left\| \mathsf{m}_{\upbeta,\bA}[\mu(t)](x^{2}(t))\right\|^{2}\right]R(t)\\
     &+\E\left[\left\langle \mathsf{m}_{\upbeta,\bA}[\mu(t)](x^{1}(t)),\mathsf{m}_{\upbeta,\bA}[\mu(t)](x^{2}(t))\right\rangle\right]
     \left(d-2+R(t)^{2}\right)+\de M(t).
\end{align*}
The main observation is that in the regime $\upbeta\to\infty$, the linear terms in $d$ cancel. This is due to the following.

\begin{lemma}\label{lemma:Laplace_method}
For $\mu$ as in Assumption \ref{ass:low-temperature},
\begin{equation*}
   \E\left\<\mathsf{m}_{\upbeta,\bA}[\mu](x),\mathsf{m}_{\upbeta,\bA}[\mu](y)\right\>=\E\left\<\frac{\bA x}{\|\bA x\|},\frac{\bA y}{\|\bA y\|}\right\>+O\left(d^{-\frac32}+\upbeta^{-\frac12}\right)
\end{equation*}
for all $x,y\in\S^{d-1}$.
\end{lemma}

\begin{proof}[Proof of Lemma \ref{lemma:Laplace_method}]
Denote 
\[
C_{\upbeta,\mu}(x,y)\coloneqq\E\left\<\mathsf{m}_{\upbeta,\bA}[\mu](x),\mathsf{m}_{\upbeta,\bA}[\mu](y)\right\>.
\]
We have
\begin{equation*}
    \mathsf{m}_{\upbeta,\bA}[\mu](x_i)
    = \int \frac{\displaystyle e^{\|\bA x_i\|\upbeta\;\<\frac{\bA x_i}{\|\bA x_i\|},y\>}}{\int e^{\|\bA x_i\|\upbeta\; \<\frac{\bA x_i}{\|\bA x_i\|},y\>}\mu(\de y)} y\mu(\de y).
\end{equation*}
Fix $x\in\S^{d-1}$ and write
\[
u_x:=\frac{\bA x}{\|\bA x\|},\qquad
\lambda_x:=\upbeta\|\bA x\|.
\]

\subsubsection*{Step 1. Laplace method}

Let $0<\delta\leq 1$ and denote by $U_\delta(u_x)$ the geodesic ball of radius $\delta$ around $u_x$:
\[
U_\delta(u_x):=\{z\in\S^{d-1}\colon d(z,u_x)\le\delta\}.
\]
Let $\varepsilon>0$ and work under the event $\{\|\bA x\|\geq \varepsilon \}.$ 
For  $y\in \S^{d-1}\setminus U_{\delta}$, $\<u_x,y\>\geq \delta$, so $\<u_x,y\>\leq \lambda_x(1-\delta^{2}/2)$, which implies 
\[
\int_{\S^{d-1}\setminus U_{\delta}(u_x)}e^{\upbeta \<\bA x,y\>}\mu(\de y)\leq e^{\lambda_x\left(1-\frac{\delta^{2}}{2}\right)}.
\]
It suffices to estimate the integral on $U_{\delta}(u_x)$ by the Laplace method. 
Let $\Exp_{u_x}\colon \mathsf{T}_{u_x}\S^{d-1}\to\S^{d-1}$ denote the exponential map. For $\delta>0$ small enough, there exists $\delta'>0$ such that
\[
\Exp_{u_x}\colon V_{\delta'}\coloneqq\{v\in \mathsf{T}_{u_x}\S^{d-1}\colon \|v\|\le\delta'\}
\to U_\delta(u_x)
\]
is a $C^1$ diffeomorphism. In these coordinates,
\[
\Exp_{u_x}(v)=\cos(\|v\|)u_x+\sin(\|v\|)\frac{v}{\|v\|},
\]
and we have
\[
\langle u_x,\Exp_{u_x}(v)\rangle=\cos(\|v\|)
=1-\frac{\|v\|^2}{2}+O(\|v\|^4).
\]
The volume element satisfies
\[
\mathrm d\upsigma_d(\Exp_{u_x}(v))
=\left(\frac{\sin\|v\|}{\|v\|}\right)^{d-2}\mathrm dv.
\]
Thus
\begin{align*}
\mathrm{Den}_\lambda(x)&\coloneqq\int_{U_\delta(u_x)}
e^{\lambda_x\langle u_x,z\rangle}\,\mathrm \mu(\de z)\\
&=\int_{V_{\delta'}}
e^{\lambda_x\cos\|v\|}\rho(\Exp_{u_x}(v))
\left(\frac{\sin\|v\|}{\|v\|}\right)^{d-2}\mathrm dv.
\end{align*}
Now set $w\coloneqq\sqrt{\lambda_x}\,v$ so that $v=w/\sqrt{\lambda_x}$. Then
\[
e^{\lambda_x\cos\|v\|}
=e^{\lambda_x-\frac{\lambda_x\|v\|^2}{2}}\,e^{\lambda_x\left(\cos\|v\|-1+\frac12\|v\|^2\right)}
=e^{\lambda_x-\frac{\|w\|^2}{2}}\,f(\lambda_x,\|w\|),
\]
where $f(\lambda_x,\|w\|)$ is uniformly bounded and converges to $1$ as $\lambda_x\to\infty$ for each fixed $w$. Similarly,
\[
\rho(\Exp_{u_x}(v))=\rho(u_x)+O\left(\frac{\|w\|}{\sqrt{\lambda_x}}\right),\qquad
\left(\frac{\sin\|v\|}{\|v\|}\right)^{d-2}
=1+O\left(d \frac{\|w\|^2}{\lambda_x}\right),
\]
with constants depending only on bounds on $\rho$ and $\nabla \rho$. Hence,
\begin{align*}
\mathrm{Den}_\lambda(x)
&= e^{\lambda_x}\lambda_x^{-\frac{d-1}{2}}
\int_{V_{\delta'\sqrt{\lambda_x}}}
e^{-\frac{\|w\|^2}{2}}
\rho(u_x)\left(1+O(\lambda_x^{-1/2}(1+\|w\|^2))\right)\,\mathrm dw\\
&= e^{\lambda_x}\lambda_x^{-\frac{d-1}{2}}\rho(u_x)
\left((2\pi)^{\frac{d-1}{2}}+O(\lambda_x^{-\frac12})\right),
\end{align*}
where $2\pi$ comes from computing the Gaussian integral, and we used that the Gaussian tail outside $V_{\delta'\sqrt{\lambda_x}}$ is exponentially small in $\lambda_x$ and can be absorbed into the same $O(\lambda_x^{-1/2})$. Thus we can write
\[
\mathrm{Den}_\lambda(x) = I_\lambda(x)\left(1+O(\lambda_x^{-\frac12})\right),
\]
where
\[
I_\lambda(x)
\coloneqq e^{\lambda_x}\lambda_x^{-\frac{d-1}{2}}\rho(u_x)(2\pi)^{\frac{d-1}{2}}.
\]
The numerator is
\[
\mathrm{Num}_\lambda(x)
\coloneqq\int e^{\lambda_x\langle u_x,z\rangle} z\,\mathrm \mu(\de z).
\]
Repeating the same argument in exponential coordinates and using the symmetry of the Gaussian in $w$ (the odd part integrates to zero), one obtains
\[
\mathrm{Num}_\lambda(x)
= u_x\,I_\lambda(x)\left(1+O(\lambda_x^{-\frac12})\right),
\]
so that
\[
\mathsf{m}_{\upbeta,\bA}[\mu](x)
=\frac{\mathrm{Num}_\lambda(x)}{\mathrm{Den}_\lambda(x)}
= u_x + O\left(\lambda_x^{-\frac12}\right)
= \frac{\bA x}{\|\bA x\|} + O\left((\upbeta\|\bA x\|)^{-\frac12}\right).
\]
All constants here depend only on bounds on $\rho$ and $\nabla \rho$, not on $x$. On the event $\{\|\bA x\|\ge\varepsilon\}$, we thus have
\begin{equation*}
\left\|\mathsf{m}_{\upbeta,\bA}[\mu](x)-u_x \right\|
\le C(\upbeta\varepsilon)^{-\frac12}.
\end{equation*}
Furthermore, under the event $\{ \|\bA x\|<\varepsilon \}$, we have 
\[
\|\mathsf{m}_{\upbeta,\bA}[\mu](x)-u_x\|\leq 2.
\]
We then rewrite $C_{\upbeta,\mu}(x,y)$ as  
\begin{align*}
C_{\upbeta,\mu}(x,y)=\E\Big[\Big\< &\mathsf{m}_{\upbeta,\bA}[\mu](x)\Big(1_{\|\bA x\|<\varepsilon}+1_{\|\bA x\|\geq \varepsilon}\Big),\\
&\mathsf{m}_{\upbeta,\bA}[\mu](y)\Big(1_{\|\bA y\|<\varepsilon}+1_{\|\bA y\|\geq \varepsilon}\Big)\Big\>\Big].
\end{align*}
Putting all together, we gather
\begin{equation*}
\left|C_{\upbeta,\mu}(x,y)-\E\<u_x,u_y\>\right|\leq 4\,\P(\|\bA x\|\leq \varepsilon)+C(\upbeta\varepsilon)^{-\frac12}.
\end{equation*}
Using the Gaussian concentration inequality \cite[pp. 77]{van2014probability}, we upper bound the last term and obtain
\[
C_{\upbeta,\mu}(x,y)=\E\left\<\frac{\bA x}{\|\bA x\|},\frac{\bA y}{\|\bA y\|}\right\>+r(\upbeta,d),
\]
where $r(\upbeta,d)=O((\upbeta \varepsilon)^{-\frac12}+e^{-c(\varepsilon)d})$.
\end{proof}

\subsubsection*{Step 2. Expanding the expectation}

We use

\begin{lemma}\label{lem:delta_method}
For $\bA=\frac{1}{d\sigma_{\bA}^{2}}\bW^{\sT}\bW'$, we have 
\begin{equation*}
\E\left\<\frac{\bA x}{\|\bA x\|},\frac{\bA y}{\|\bA y\|}\right\>= \langle x,y\rangle\,+\frac{\<x,y\>^{3}-\<x,y\>}{2d}+O(d^{-\frac32}).
\end{equation*}
\end{lemma}

\begin{proof}[Proof of Lemma \ref{lem:delta_method}]
The proof is an application of the "delta method" to the dimensional parameters.

Define the random variables $X_d=(\<\bA x,\bA y\>,\|\bA x\|^{2},\|\bA y\|^{2})$, whose mean $m$ is $(\<x,y\>,1,1)$, and covariance $\Sigma$ is given by
\[
\Sigma
=
\begin{bmatrix}
\bullet & \frac{2\<x,y\>}{d} & \frac{2\<x,y\>}{d} \\[0.7em]
\frac{2\<x,y\>}{d} & \frac{2}{d}&\frac{2\<x,y\>^{2}}{d} \\[0.7em]
\frac{2\<x,y\>}{d} &\frac{2\<x,y\>^{2}}{d}  & \frac{2}{d}
\end{bmatrix}+O(d^{-2}).
\]
Define $g(x,y,z)\coloneqq\frac{x}{\sqrt{yz}}$. A Taylor-Lagrange expansion shows
\begin{equation*}
    \E[g(X_d)]=g(m)+\frac{1}{2}\nabla^{2}g(m) : \Sigma +O(R_d),
\end{equation*}
where $R_d$ is a third order error satisfying $|R_d|=O(d^{-\frac32})$. The Hessian evaluated in $m$ gives 
\[
\nabla^2 g(m)
=
\begin{bmatrix}
0 &\quad  -\dfrac{1}{2} & \quad -\dfrac{1}{2} \\[1em]
-\dfrac{1}{2} & \quad \dfrac{3r}{4} &\quad  \dfrac{r}{4} \\[1em]
-\dfrac{1}{2} & \quad \dfrac{r}{4} &\quad  \dfrac{3r}{4}
\end{bmatrix},
\]
where $r=\<x,y\>$.
So it leads to the expansion
\[
\E\left\<\frac{\bA x}{\|\bA x\|},\frac{\bA y}{\|\bA y\|}\right\>=\<x,y\>-\frac{1}{2d}\left( \<x,y\>-\<x,y\>^{3}\right)+\Delta_d,
\]
where $|\Delta_d|=O(d^{-\frac32})$, as desired.
\end{proof} 

Using Lemma \ref{lemma:Laplace_method}, and Lemma \ref{lem:delta_method}, we simplify the kernel expression as
\begin{equation}\label{eq:kernel_expansion}
C_{\upbeta,\mu}(x,y)=\<x,y\>-\frac{1}{2d}\left( \<x,y\>-\<x,y\>^{3}\right)+r(\upbeta,d),
\end{equation}
and plugging \eqref{eq:kernel_expansion} into the expression of the drift, we get
\begin{equation}\label{eq:mean_field_ito}
 \de R(t)= -\frac{1}{d}(R(t)-R(t)^3)\,\de t + \de M(t) + \varepsilon_{d,\upbeta}(t)\,\de t,
\end{equation}
where $|\varepsilon_{d,\upbeta}(t)| = O(d^{-\frac32}+\upbeta^{-\frac12}).$

\subsubsection*{Step 3. Martingale term}
\label{step2:martingale_sec8}

By the same reasoning as in the previous proof, we have
\begin{align*}
\de\langle M\rangle_t
&=\E\left[\left\<\mathsf{m}_{\upbeta,\bA}[\mu](x^{(1)}(t)),\mathsf{m}_{\upbeta,\bA}[\mu](x^{(2)}(t)\right\>\right]\left\<\proj_{x^{(1)}(t)}x^{(2)}(t),\proj_{x^{(2)}(t)}x^{(1)}(t)\right\> \\
&+\E\left[\left\|\mathsf{m}_{\upbeta,\bA}[\mu](x^{(1)}(t))\right\|^{2}\right] \left\|\proj_{x^{(1)}(t)}x^{(2)}(t)\right\|^{2}\\
&+\E\left[\left\|\mathsf{m}_{\upbeta,\bA}[\mu](x^{(2)}(t))\right\|^{2}\right] \left\|\proj_{x^{(2)}(t)}x^{(1)}(t)\right\|^{2},
\end{align*}
thus we have $\<M\>_t\leq Ct/d$.
This implies, with probability at least $1-\delta$, that
\begin{equation}\label{eq:martingale_sup_bound_clean_thm4}
\sup_{t\in[0,T]}|M_t|
\le
C\sqrt{\frac{4T}{d}\log\left(\frac{1}{\delta}\right)}
\end{equation}

\subsubsection*{Step 4. Combining all bounds}

Integrating  \eqref{eq:mean_field_ito}, and using \eqref{eq:martingale_sup_bound_clean_thm4} gives that with probability at least $1-\delta$, we have
\begin{equation*}
    R(t)-R(0)\leq-\frac1d\int_{0}^{t}R(s)(1-R(s)^{2})\de s+\int_0^t \varepsilon_{d,\upbeta}(s)\,\de s+C\sqrt{\frac{4T}{d}\log\left(\frac{1}{\delta}\right)}.
\end{equation*}
Taking the supremum, and using the crude bound $-1\leq R(s) \leq 1$, we have with probability at least $1-\delta,$
\begin{equation*}
    \sup_{t\in[0,T]}|R(t)-R(0)|\leq \frac{T}{d} +C\sqrt{\frac{4T}{d} \log\left(\frac{1}{\delta}\right)}+T\sup_{t\in[0,T]}|\varepsilon_{d,\upbeta}(t)|.
\end{equation*}
This concludes the proof of the first part of the statement.

\subsubsection*{Step 5. Rescaled dynamics}

Let $\hat{R}_d(t)\coloneqq R(dt).$ The evolution of $\hat{R}_d(t)$ is given by 
\begin{equation*}
    \de \hat{R}_d(t)=-\hat R_d(t)(1-\hat R_d(t)^{2})\de t+d\,\varepsilon_{d,\upbeta}(dt)\,\de t+\de \hat M_d(t),
\end{equation*}
where $\hat M_d(t)\coloneqq M(dt)$.
We compute the quadratic variation of the martingale $\hat M_d.$ Using Lemma \ref{lemma:Laplace_method}, we have 
\[
\E\left[\left\<\mathsf{m}_{\upbeta,\bA}[\mu](x^{(1)}(t)),\mathsf{m}_{\upbeta,\bA}[\mu](x^{(2)}(t))\right\>\right]=\left\<x^{(1)}(t),x^{(2)}(t)\right\>+O(d^{-1}).
\]
Using 
\[
\left\<\proj_{x^{(1)}(t)}x^{(2)}(t),\proj_{x^{(2)}(t)}x^{(1)}(t)\right\>=-\left\<x^{(1)}(t),x^{(2)}(t)\right\>+\left\<x^{(1)}(t),x^{(2)}(t)\right\>^{3}
\]
and $\|\proj_{x^{(1)}(t)} x^{(2)}(t)\|^{2}=1-\<x^{(1)}(t),x^{(2)}(t)\>^{2},$ we find that the quadratic variation (given \ref{step2:martingale_sec8}) is 
\begin{equation*}
\de\<\hat M_d\>_t=2(\hat R_d(t)^{2}-1)^{2}\de t+\epsilon(d)\,\de t,
\end{equation*}
where $|\epsilon(d)|=O(d^{-1})$. 

We proceed by a classical martingale argument per Stroock-Varadhan \cite{stroock2007multidimensional, ethier2009markov}.  
We first prove the tightness of the sequence of random variables $(\hat{R}_d)_{d\in \naturals}$. Consider
\[
\de R_d(t)=-R_d(t)(1-R_d(t)^{2})\de t+\de \hat M_d(t).
\]
We have that 
\[
\|R_d(t)-\hat{R}_d(t)\|\underset{d\rightarrow\infty}{\rightarrow} 0.
\]
Let $t,s \in \reals$, and $d\in \naturals$; by triangle inequality, convexity of $x\mapsto x^{4}$ and Burkholder-Davis-Gundy inequality, we have 
\begin{align*}
    \E\|R_d(t)-R_d(s)\|^{4}&\leq C\E\left\|\int_{s}^{t}b(R_d(u))\de u \right\|^{4} +C\E\left\|\int_{s}^{t}\sigma(u)\de B(u)\right\|^{4}\\
    &\leq C(t-s)^{4}+C(t-s)^{2},
\end{align*}
where $b(x)\coloneqq -x(1-x^2)$ and, by Dambis-Dubins-Schwarz, there exists a one-dimensional Brownian motion $B$ such that $\hat M_d(t)=\int_0^t \sigma(u)\,\de B(u)$ with $$\sigma(u)^2=\frac{\de}{\de u}\langle \hat M_d\rangle_u,$$ 
and we have used that $\de \<\hat M_d\>_t\leq 3\,\de t$ for $d$ large enough.
By the Kolmogorov continuity criterion, it ensues that that $(R_{d}(\cdot))_{d\in \naturals}$ is tight.

Define $\mathsf{L}_d$ on smooth functions by
\begin{equation*}
    \mathsf{L}_df(x)=-x(1-x^{2})f'(x)+\,(1-x^{2})^{2}f''(x)+\mathrm{err}(d,\upbeta),
\end{equation*}
where $\mathrm{err}(d,\upbeta)$ are vanishing terms when $d,\upbeta\to \infty.$
By the Itô formula, for all $f$ smooth enough, $\cM_{d,f}$ defined as 
\begin{equation*}
    \cM_{d,f}(t)=f(\hat{R}_{d}(t))-f(\hat{R}_{d}(0))-\int_{0}^{t}\mathsf{L}_{d}f(\hat{R}_d(s))\de s 
\end{equation*}
is a martingale. For any subsequential limit $\hat{R}$, we then have that $\hat{R}$ solves the following martingale problem
\begin{equation*}
\cM_{d,f}(\hat{R})=f(\hat{R}(t))-f(\hat{R}(0))-\int_{0}^{t}\mathsf{L}_{d}f(\hat{R}(s))\de s.
\end{equation*}
Since the martingale problem is uniquely defined \cite{stroock2007multidimensional}, we deduce that the solution converges to 
\begin{equation*}
    \de u(t)=-u(t)(1-u(t)^{2})+\sqrt{2}(1-u(t)^{2})\de B(t),
\end{equation*}
which yields the desired SDE.
\end{proof}

\subsection{On the behavior of \eqref{eq: sde.logistic}} \label{sec: sde.logistic}

We now discuss the long-time behavior of \eqref{eq: sde.logistic}. Define $\theta(t) = \sin u(t)$. One has $\theta(t)\in[-\pi/2,\pi/2]$ and then finds
\begin{equation*}
    \de\theta(t) = \sqrt{2} \cos \theta(t) \,\de B(t).
\end{equation*}
We write $\theta(t) = \theta(0) + \int_0^t \sqrt{2} \cos\theta(s)\, \de B(s)$ so $\theta(t)$ is a bounded martingale with quadratic variation $\langle \theta\rangle_t = \int_0^t 2\cos^2\theta(s)\de s$. 
By the Dambis-Dubins-Schwarz theorem, there exists a Brownian motion $W(t)$ such that 
\begin{equation*}
 \theta(t) = \theta(0)+W(\langle \theta\rangle_t).   
\end{equation*}
By classical results \cite[Theorem 3.19]{le2016brownian} it follows that there exists a random variable $\theta_\infty\in[-\pi/2,\pi/2]$ such that $\theta(t)\to\theta_\infty$ a.s. as $t\to\infty$. 
If $|\theta_\infty|<\pi/2$, then eventually $\cos\theta(t)$ would eventually stay bounded away from $0$, which would force $\langle\theta\rangle_t\to0$. But then $\theta(t) = \theta(0)+ W({\langle \theta\rangle_t})$ would keep fluctuating and not converge. Thus $\cos\theta_\infty = 0$ and so $\theta_\infty=\pm\pi/2$.
Consequently $u(t)\to\pm1$ a.s. as $t\to\infty$. We can also compute, by hitting time arguments, 
\begin{equation*}
\mathbb{P}_{\theta_0}\left(\theta_\infty = \frac\pi2\right) = \frac{\theta(0)}{\pi}+\frac12,
\end{equation*}
and so for $u(0)=0$, 
\begin{equation*}
    \mathbb{P}(u_\infty=-1) = \mathbb{P}(u_\infty=1) = \frac12, 
\end{equation*}
which is illustrated in Figure \ref{fig: random.sde}.

\begin{figure}[h!]
    \centering
    \includegraphics[scale=0.55]{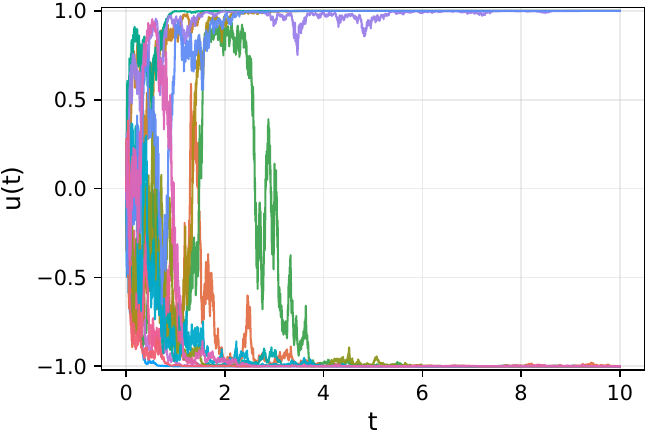}
    \includegraphics[scale=0.55]{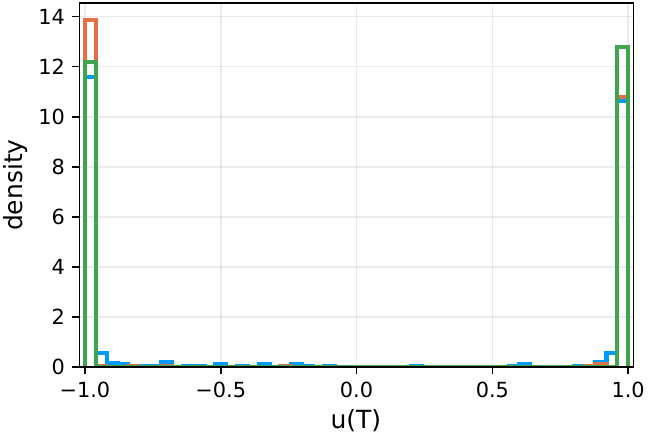}
    \caption{Realizations of solutions to \eqref{eq: sde.logistic} in Theorem \ref{thm:large_beta_meta}. }
    \label{fig: random.sde}
\end{figure}

\appendix

\section{Toolkit}\label{sec:Toolkit}

\subsection{Differentiation of the sphere}

We collect standard expressions for the gradient and Hessian on the product
manifold $(\S^{d-1})^n$, viewed as an embedded submanifold of $\reals^{d\times n}$ (or $\reals^{dn}$).
For $X=(x_1,\dots,x_n)\in(\S^{d-1})^n$ and $U=(u_1,\dots,u_n)\in \reals^{d\times n}$, recall that
the orthogonal projection onto the tangent space is given componentwise by
\[
(\proj_X U)_i = \proj_{x_i}u_i = (I-x_i x_i^{\sT})u_i.
\]

\begin{lemma}\label{lem:toolkit_geo_riem}
Let $f:(\S^{d-1})^n\to \reals$ be $C^2$, and let $\tilde f$ be a $C^2$ extension of $f$ to a
neighborhood of $(\S^{d-1})^n$ in $\reals^{d\times n}$.
Then the Riemannian gradient of $f$ on $(\S^{d-1})^n$ is
\begin{equation*}
\grad f(X)=\proj_X \nabla \tilde f(X),
\end{equation*}
where $\nabla \tilde{f}$ denotes the Euclidean gradient in $\reals^{d\times n}$ of $\tilde f$.

For any tangent vector $V\in \mathsf{T}_X(\S^{d-1})^n$, the Riemannian Hessian satisfies
\begin{equation*}
\Hess f(X)[V]
=\proj_X\left(\nabla^2 \tilde f(X)[V] - V\,\Diag\left(\diag(X^{\sT}\nabla \tilde f(X))\right)\right).
\end{equation*}
Equivalently, component-wise,
\[
\left(\Hess f(X)[V]\right)_i
=\proj_{x_i}\left(\nabla^2_{x_i}\tilde f(X)[V]_i - \langle x_i,\nabla_{x_i}\tilde f(X)\rangle\,v_i\right).
\]
If $V$ is a $C^1$ vector field on $(\S^{d-1})^n$ and $f$ is smooth, then the second-order operator
$V^2$ is
\[
(V^2 f)(X)= V(V f)(X)
= \Hess f(X)[V(X),V(X)] + \langle \grad f(X),(\nabla_V V)(X)\rangle,
\]
where $\nabla_V$ is the Levi--Civita connection on $(\S^{d-1})^n$.
\end{lemma}

Recall that if $U,V$ are smooth tangent vector fields on an embedded submanifold, and $\bar V$
is any smooth extension of $V$ to a neighborhood in the ambient Euclidean space, then the
Levi--Civita connection satisfies
\[
\nabla_U V = \proj_X(D\bar V(X)[U]).
\]

\begin{proof}
See, e.g., \cite{boumal2023introduction}.
\end{proof}

\begin{definition} \label{def: derivatives}
    For all $k\in \naturals$, we denote $\mathfrak{X}^{k}((\S^{d-1})^n)$ the set of vector fields admitting a derivative of order $k$, endowed with the norm 
    \begin{equation*}
        \|V\|_{\mathfrak{X}^{k}((\S^{d-1})^n)}\coloneqq \max_{X\in (\S^{d-1})^n}\left(\|V(X)\|+\max_{j\in[k]}\|\nabla^{j}V(X)\|\right)
    \end{equation*}
\end{definition}

\subsection{Cylindrical Wiener process} \label{sec: cylindrical.wiener}

Let $(\Uptheta,\mathscr G,{\uprho_*})$ be a probability space such that
$L^{2}({\uprho_*})$ is separable. 
Let $(\Omega, \mathscr{F}, (\mathscr{F}_t)_{t\geq0}, \mathbb{P})$ be a filtered, complete probability space with right-continuous filtration $\mathscr{F}_t$.

A cylindrical Wiener process on $L^2(\uprho_*)$ defined on $(\Omega, \mathscr{F}, (\mathscr{F}_t)_{t\geq0}, \mathbb{P})$ is a family $(\mathsf{W}(t))_{t\ge 0}$ such that
\begin{enumerate}
    \item for each $t\ge 0$, $\mathsf{W}(t): L^2(\uprho_*)\to L^2(\Omega; \mathbb{P})$ is a linear operator;
    \item for each $h\in L^2(\uprho_*)$,
the scalar-valued stochastic process $(\mathsf{W}(t)[h])_{t\ge 0}$ is an $(\mathscr{F}_t)_{t\geq0}$-Brownian motion with variance
\[
\Var(\mathsf{W}(t)[h])=t\|h\|_{L^2(\uprho_*)}^2.
\]
Equivalently, $\E[\mathsf{W}(t)[h] \mathsf{W}(s)[g]]=(s\wedge t)\langle h,g\rangle_{L^2(\uprho_*)}$.
\end{enumerate}

For a $(\mathscr{F}_t)_{t\geq0}$-progressively measurable process $(G(t,\cdot))_{t\geq0}$ with values in $L^2(\uprho_*)$ that almost surely satisfies $G\in L^2_{\mathrm{loc}}(\R_{\geq0}; L^2(\uprho_*))$, we define 
\begin{equation*}
    \int_0^t \int_{\Uptheta} G(s,\theta)\mathsf{W}(\de \theta,\de t) =: \int_0^t \ell(s)\mathsf{W}(\de s),
\end{equation*}
where $\ell(s)f = \langle G(s,\cdot), f\rangle_{L^2(\uprho_*)}$ for all $f\in L^2(\uprho_*)$. See \cite{da2014stochastic,prevot2007concise}.
Furthermore, let $(e_k)_{k\ge 1}$ be an orthonormal basis of $L^2(\uprho_*)$. Then there exist i.i.d. real-valued Brownian
motions $(\beta^k_t)_{k\ge 1}$ such that, for suitable $G$,
\begin{equation*}
\int_{0}^{t}\int_{\Uptheta} G(s,\theta)\,\mathsf{W}(\de \theta,\de s)
=\sum_{k=1}^{\infty}\int_{0}^{t}\langle G(s,\cdot),e_k\rangle_{L^2({\uprho_*})}\,\de \beta_{s}^{k},
\end{equation*}
where the series converges in $L^2(\Omega; \mathbb{P})$ (and hence in probability); see, e.g.,
\cite[Section 4.2.2]{da2014stochastic}.

\subsection{It\^o formula}

We also give the definition of the Itô SDE on $(\S^{d-1})^n$.

\begin{definition} \label{def: ito.formula}
Let $B\in C^{1}((\S^{d-1})^{n};\mathsf{T}(\S^{d-1})^{n})$ and
$G:(\S^{d-1})^n\times\Uptheta\to \mathsf{T}(\S^{d-1})^n$ be such that for ${\uprho_*}$-a.e.\ $\theta$,
the map $X\mapsto G(X,\theta)$ is $C^2$.
A continuous $(\S^{d-1})^{n}$-valued adapted process $(X(t))_{t\ge 0}$ with $X(0)=x$ is called a solution to the Itô SDE 
    \begin{equation} \label{eq: ito.sde.final}
        \de X(t)=B(X(t))\de t+\int_{\Uptheta}G(X(t),\bm{\theta})\mathsf{W}(\de \bm{\theta},\de t),
    \end{equation}
if for all $g\in C^{\infty}((\S^{d-1})^{n})$, we have
\begin{align*}
g(X(t))
= g(x)
&+ \int_{0}^{t} (Bg)(X(s))\,\de s \\
&+ \frac{1}{2}\int_{0}^{t}\int_{\Uptheta} \Hess g(X(s))[G(X(s),\theta), G(X(s),\theta)]\,{\uprho_*}(\de \theta)\,\de s \\
&+ \int_{0}^{t}\int_{\Uptheta} (G(\cdot,\theta)g)(X(s))\,\mathsf{W}(\de \theta,\de s),
\end{align*}
where $Bg=\langle \grad g,B\rangle$, $G(\cdot,\theta)g=\langle \grad g, G(\cdot,\theta)\rangle$.
\end{definition}

\begin{proposition} \label{prop:existence_SDE_sphere}
Assume $B\in C^{1}((\S^{d-1})^{n};\mathsf{T}(\S^{d-1})^{n})$ and that for ${\uprho_*}$-a.e.\ $\theta$,
$G(\cdot,\theta)\in C^{2}\left((\S^{d-1})^{n};\mathsf{T}(\S^{d-1})^{n}\right)$, with the integrability needed to
make the stochastic integral well-defined.
Then for every $x\in (\S^{d-1})^n$, there exists a unique global solution $(X(t))_{t\ge 0}$ to
\eqref{eq: ito.sde.final} with $X(0)=x$.
\end{proposition}

\begin{proof}
This is a special case of \cite[Proposition 4.2]{gess2024rsgd}. Since $(\S^{d-1})^n$ is compact,
the explosion time is infinite.
\end{proof}

\bibliography{bibliography}
\bibliographystyle{alpha}

\vspace{1cm}

\noindent
\contactcard{Hugo Koubbi}{%
CEREMADE, UMR 7534\\
Université Paris Dauphine PSL\\
75775 Paris Cedex 16, France%
}{hugo.koubbi@dauphine.psl.eu}
\hfill
\contactcard{Borjan Geshkovski}{%
Laboratoire Jacques-Louis Lions\\
Inria \& Sorbonne Université\\
75005 Paris, France%
}{borjan.geshkovski@inria.fr}

\vspace{0.6cm}

\noindent\hspace*{\fill}%
\contactcard{Philippe Rigollet}{%
Department of Mathematics\\
Massachusetts Institute of Technology\\
77 Massachusetts Ave\\
Cambridge, MA 02139, USA%
}{rigollet@mit.edu}%
\hspace*{\fill}

\end{document}